\documentclass[twoside,11pt]{amsart}

\usepackage[foot]{amsaddr}

\author{Matthew D. Kvalheim}
\address[Kvalheim, Koditschek]{School of Engineering and Applied Science, University of Pennsylvania,
	Philadelphia, PA 19104, USA}

\author{Paul Gustafson}
\address[Gustafson]{Department of Electrical Engineering, Wright State University, Dayton, OH 45435, USA}

\author{Daniel E. Koditschek}

\email{kvalheim@seas.upenn.edu, paul.gustafson@wright.edu, kod@seas.upenn.edu}

\title{Conley's Fundamental Theorem for a Class of Hybrid Systems}

\usepackage{import} 
\usepackage{enumitem} 

\usepackage{float} 
\usepackage{amsthm, amssymb,amsfonts,mathrsfs,amsxtra,dsfont}
\usepackage{amsmath, amscd}
\usepackage{bm}

\usepackage[utf8]{inputenc}
\usepackage[T1]{fontenc}
\usepackage{lmodern}
\usepackage{amsmath}
\usepackage{amssymb}
\usepackage{amsthm}
\usepackage{stmaryrd} 
\usepackage{mathtools}
\usepackage{tikz-cd}
\usepackage{subfig}
\usepackage{lineno}

\newcommand{\concept}[1]{\textbf{#1}}

\DeclareMathOperator{\id}{id}
\DeclareMathOperator{\dist}{dist}

\newcommand{\hooklongrightarrow}{\lhook\joinrel\longrightarrow}
\newcommand{\twoheadlongrightarrow}{\relbar\joinrel\twoheadrightarrow}

\newcommand{\exc}{\mathcal E}

\newcommand{\conley}{\mathcal Ch}

\newcommand{\conleyet}{\conley^{\epsilon,T}}

\newcommand{\st}{\mathcal{X}} 
\newcommand{\fs}{\mathcal{F}} 
\newcommand{\gd}{\mathcal{G}} 
\newcommand{\stp}{\mathrm{stop}}

\newcommand{\N}{\mathbb{N}}

\newcommand{\R}{\mathbb{R}}
\newcommand{\cl}{\textnormal{cl}}
\newcommand{\Sus}{\Sigma}

\newcommand{\slot}{\,\cdot\,} 

\newcommand{\T}{\mathsf{T}}

\newcommand{\interior}{\textnormal{int}}
\newcommand{\dom}{\textnormal{dom}}
\newcommand{\THS}{THS}
\newcommand{\MHS}{MHS}

\theoremstyle{definition}

\newtheorem{Lem}{Lemma}

\newtheorem{Co}{Corollary}

\newtheorem{Prop}{Proposition}
\newtheorem{Assump}{Assumption}
\newtheorem*{Th-non}{Theorem}
\newcommand{\Thnon}{Th-non}

\newtheorem{Def}{Definition}
\newtheorem*{Def*}{Definition}
\newtheorem{Ex}{Example}
\newtheorem{Rem}{Remark}

\usepackage{thm-restate} 
\newcommand{\thistheoremname}{}
\newtheorem*{genericthm}{\thistheoremname}
{\renewcommand{\thistheoremname}{Theorem~\ref{#1}$'$}%
	\begin{genericthm}}
	{\end{genericthm}}

\usepackage{marvosym}
\usepackage[
    hypertexnames,%
    citecolor=blue,%
    colorlinks=true,%
    linkcolor=red%
]{hyperref}
\usepackage[all]{hypcap}

\begin{document}

\begin{abstract}
We establish versions of Conley's (i) fundamental theorem and (ii) decomposition theorem for a broad class of hybrid dynamical systems. 
The hybrid version of (i) asserts that a globally-defined \emph{hybrid complete Lyapunov function} exists for every hybrid system in this class. 
Motivated by mechanics and control settings where physical or engineered events cause abrupt changes in  a system's governing dynamics, our results apply to a large class of Lagrangian hybrid systems (with impacts) studied extensively in the robotics literature. 
Viewed formally,  these  results generalize those of Conley and Franks for continuous-time and discrete-time dynamical systems, respectively, on metric spaces. 
However, we furnish specific examples illustrating how our statement of sufficient conditions represents merely  an early step in the longer project of establishing what formal assumptions can and cannot endow hybrid systems models with the topologically well characterized partitions of limit behavior that make Conley's theory so valuable in those classical settings.  
\end{abstract}

\maketitle

\tableofcontents

\section{Introduction}\label{sec:intro}
In \cite{norton1995fundamental}, Norton argued that the following two theorems deserve to be called the ``Fundamental Theorem of Dynamical Systems.''

\begin{\Thnon}[\cite{conley1978isolated}]
  Any continuous flow on a compact metric space decomposes into a chain recurrent part and a gradient-like part.
  There exists a continuous Lyapunov function which strictly decreases along the flow on the gradient-like part.
\end{\Thnon}

\begin{\Thnon}[\cite{franks1988variation}]
  The iteration of a homeomorphism of a compact metric space decomposes the space into a chain recurrent part and a gradient-like part.
  There exists a continuous Lyapunov function which strictly decreases under iteration of the map on the gradient-like part.  
\end{\Thnon}
See \cite{alongi2007recurrence, norton1995fundamental,robinson1999dynamical} for tutorial accounts. 

From the view of applications, these results endow models that achieve them with two important guarantees. First, the decomposition establishes a clear, deterministic notion of steady state behavior that, no matter how complicated its temporal manifestation \cite{Lorenz_1964,May_1976,Holmes_1990}, imposes a computationally effective \cite{Kalies_Mischaikow_VanderVorst_2005} spatial partition into attractor basins  \cite{Milnor_2006} whose topology persists under small perturbations.
The passage from signal to symbol afforded by such partitions has great value for analyzing natural systems \cite{Arai_Kalies_Kokubu_Mischaikow_Oka_Pilarczyk_2009,Ghrist_Van_den_Berg_Vandervorst_2003}, and has encouraged slowly growing use in the synthesis of  engineered systems as well \cite{Acar_Choset_Rizzi_Atkar_Hull_2002,Chen_Mischaikow_Laramee_Pilarczyk_Zhang_2007, Haynes_Cohen_Koditschek_2011,Haynes_Rizzi_Koditschek_2012}. Second, interpreted as a universal converse Lyapunov theorem, global analogue to the classical counterpart addressing a specific basin \cite{kellett2015classical}, the long established value for classical \cite{Sontag_1989}, multistable \cite{Forni_Angeli_2018} and hybrid control systems theory \cite{Goebel_Sanfelice_Teel_2009} is leveraged by a steadily advancing literature on constructive methods for their eventual feedback closed loops \cite{Ban_Kalies_2006,Giesl_Hafstein_2015}. In our view, one of the most important applications for Lyapunov-expression of basin partitions is their long-proven role in  sequential composition \cite{Burridge_Rizzi_Koditschek_1995} and their promise for parallel composition \cite{Cowan_2007, Topping_Vasilopoulos_De_Koditschek_2019}, increasing the expressive richness of topologically grounded type theories \cite{Awodey_Harper_2015} emerging from hybrid dynamical categories that admit them. 

\subsection{Contributions and organization of the paper}\label{sec:contributions}

Motivated by problems of robotics and biomechanics, where the making and breaking of contacts intrinsic to most tasks necessitates the introduction of hybrid systems models \cite{Koditschek_2021}, this paper addresses the question of what hybrid systems models admit a version of Conley's fundamental theorem. 
We focus on a partial extension of a particularly simple but empirically useful class \cite{Johnson_Burden_Koditschek_2016}, relative to which a closely related extension can be shown to generate a formal category equipped with the desired compositional operators  \cite{culbertson2019formal}.  
Specifically, we introduce the class of \emph{topological hybrid systems (\THS)} and the subclass of \emph{metric hybrid systems  (\MHS)}  (Def.~\ref{def:HybridSystem}) 
that roughly generalizes the model of \cite{Johnson_Burden_Koditschek_2016} (see SM \S \ref{app:JBK-relate}). 

After imposing assumptions including the \emph{trapping guard} condition (Def.~\ref{def:trapping-guards})  we prove an appropriately generalized version of Conley's decomposition theorem (Theorem~\ref{th:conley-decomp}) as well as the existence (Theorem~\ref{th:hybrid-conley}) of a globally-defined \emph{hybrid complete Lyapunov function} (Def.~\ref{def:lyapunov}). 
These are our main results.
Because we believe the methods used to prove the main results are of independent interest, we provide a rough synopsis of the proof techniques and the underlying new ideas as follows.

Given a suitable \MHS~ $H$ with state space $\st$, we embed $H$ into an \MHS~ $H'$ (the \emph{relaxed} system) with larger state space $\st'$ (see Fig.~\ref{fig:hybrid-suspension}).
We then topologically ``glue'' $\st'$ to itself along the \emph{reset map} of $H'$ to obtain a metrizable space $\Sus_H$ (the \emph{hybrid suspension}) equipped with a continuous semiflow $\Phi_H$.
As will be discussed at greater length in the literature review (\S \ref{sec:related-work}), versions of the relaxed hybrid system and hybrid suspension have previously appeared in \cite{johansson1999regularization} and \cite{ames2005homology,burden2015metrization} (see also the discussion at the end of \S \ref{sec:related-work} and in SM \S \ref{app:relaxed} and \S \ref{app:hybrid-suspension-subsubsec} for more details). 
Our central new contribution addresses the implications of the version of this construction we have introduced  for the nature of steady state behavior in the dynamics it carries.  
This entails establishing and exploiting several topological and  dynamical  properties of the  hybrid  suspension and their relationships to corresponding properties of the original hybrid system $H$ in a manner we now outline. 

It is possible to view $\st \subseteq \Sus_H$ as embedded inside $\Sus_H$ in such a way that the image of each \emph{execution} of $H$ coincides with the intersection with $\st$ of the image of a trajectory of $\Phi_H$.
Hurley's generalization \cite{hurley1995chain,hurley1998lyapunov} to semiflows of Conley's decomposition and fundamental theorems applies to $(\Sus_H,\Phi_H)$, in particular yielding a \emph{complete Lyapunov function} $V\colon \Sus_H\to \R$ for $\Phi_H$ whose restriction $V|_{\st}$ to $\st$ yields a candidate hybrid complete Lyapunov function for $H$.
What remains is to prove that various topological-dynamical objects associated to $H$---\emph{$\omega$-limit sets}, \emph{attracting-repelling pairs}, and \emph{chain equivalence classes}---coincide with the restrictions to $\st\subseteq \Sus_H$ of corresponding dynamical objects for $\Phi_H$.
Through various technical arguments we show that this is indeed the case, thereby proving Conley's decomposition and fundamental theorems for $H$ and, in particular, proving that the restriction $V|_{\st}$ is indeed a complete Lyapunov function for $H$.
For these arguments it turns out to be crucial (for several reasons---see Remarks \ref{rem:omega-compatibility-for-hybrifold-fails} and \ref{rem:chain-compatibility-for-hybrifold-fails} and SM \S\ref{app:hybrid-suspension} for more details) that the hybrid suspension technique differs from the \emph{hybrifold} technique of \cite{simic2000towards,simic2005towards}.

Our contributions additionally include illustrations of the applicability of our main results by presenting two broad \MHS~ subclasses to which they apply:  the \emph{smooth exit-boundary guarded} \MHS~   
(Prop.~\ref{prop:trapp-guard-suff-cond})  arising, for example in problems of legged locomotion \cite{Burden_Revzen_Sastry_2015, De_Burden_Koditschek_2018}; 
and an extension (Prop.~\ref{prop:sub-level-set}) of the \emph{Lagrangian hybrid systems} (Cor.~\ref{co:hybrid-lagrangian}), a class of models (or near variations thereof) studied in the robotics literature \cite{grizzle2001asymptotically, westervelt2003hybrid, ames2006there,poulakakis2009spring,or2010stability,bloch2017quasivelocities,razavi2017symmetry}.
In contrast, a simple counterexample (Ex.~\ref{ex:counterexample}, depicted in Fig.~\ref{fig:counterexample}) demonstrates that the conclusions of our version of Conley's theorems for \MHS~ can fail without the trapping guard condition.
Finally, we use two variants of the Hamiltonian bouncing ball model to illustrate how these results apply to mechanical systems which undergo impacts (and to mechanical systems which have only Zeno maximal executions, in the case of the first variant). 
Bouncing against gravity (Ex.~\ref{ex:bouncing-ball}) generates an \MHS~ satisfying the trapping guard condition, yielding the Conley decomposition and complete Lyapunov function (Fig.~\ref{fig:example2}) guaranteed by Theorems \ref{th:conley-decomp} and \ref{th:hybrid-conley}. 
In contrast, because linear time invariant vector fields are homogenous, the \MHS~ generated by bouncing losslessly against a Hooke's law spring (Ex.~\ref{ex:spring}) fails the trapping guard condition, so this system does not satisfy the \emph{hypotheses} of our main theorems; interestingly, however, this example does still satisfy our main theorems' \emph{conclusions}. 
We end the paper with some more philosophically motivated remarks concerning the virtue of parsimony arising from these results that generalize both the discrete (Ex.~\ref{ex:special-case-discrete-time}) and the continuous (Ex.~\ref{ex:special-case-continuous-time}) classical frameworks to unify the common but heretofore distinct assertions of \cite{conley1978isolated,franks1988variation}.

The remainder of this paper is organized as follows. After discussing related work below, we introduce
the basic definitions and concepts relevant to our main results in \S \ref{sec:preliminaries}. In \S \ref{sec:main-results} we state our main results.  In \S \ref{sec:app} and \S \ref{sec:ex}, we present the applications and examples (some very specific and some quite general) as just discussed. 
As discussed above, the proofs of our main results rely on the reduction of suitably guarded \MHS~ to classical dynamical systems on spaces obtained via the \emph{hybrid suspension} technique described in \S \ref{sec:hybrid-suspension} (see Fig.~\ref{fig:hybrid-suspension}) which generalizes the classical suspension of a discrete-time dynamical system \cite[p.~797,~pp.~21--22]{smale1967differentiable,brin2002introduction}.
We conclude with brief remarks of a more speculative nature in \S \ref{sec:conclusion}.
Supplementary Materials (SM) \S \ref{app:relate} compares some of our constructions with selected prior work. SM \S \ref{app:suspension} gives a primer on the classical suspension of a discrete-time dynamical system.
SM \S \ref{app:suspension-semiflow-continuous-implies-trapping-guard-condition} makes precise and proves the statement that, for a class of deterministic \THS~ satisfying mild assumptions, the trapping guard condition holds if and only if a  continuous hybrid suspension semiflow exists.  SM \S \ref{app:technical} contains the proofs of various technical lemmas used in the construction of suspension semiflows and well-behaved $(\epsilon, T)$-chains.

\subsection{Related work}\label{sec:related-work}

As reviewed above, for flows on compact metric spaces, Conley proved the existence of a complete Lyapunov function and that the chain recurrent set is the intersection of all attracting-repelling pairs \cite{conley1978isolated}. Franks proved the corresponding results for homeomorphisms of compact metric spaces \cite{franks1988variation}.
Hurley extended the decomposition theorem to maps and semiflows on arbitrary metric spaces \cite{hurley1995chain} and proved the existence of complete Lyapunov functions for maps on separable metric spaces \cite{hurley1998lyapunov}. 
Using Hurley's result, Patr{\~a}o proved the existence of a complete Lyapunov function for any semiflow on a separable metric space \cite{patrao2011existence}.
In the nondeterministic setting, McGehee and Wiandt generalized Franks' results to the setting of iterations of closed relations \cite{mcgehee2006conley,wiandt2008liapunov}; Bron\v{s}tein and Kopanski\v{i} generalized Conley's results to a class of set-valued dynamical systems such as those arising from certain  differential inclusions \cite{bronvstein1988chain}.
In the stochastic setting, Liu generalized the decomposition and fundamental theorems to random (semi-)dynamical systems such as those arising from stochastic (partial) differential equations \cite{liu2005random,liu2007randomII,liu2007randomIII}.

Motivated largely by mathematical models occurring in science and engineering, many investigators have worked to generalize results and tools from classical dynamical systems theory to the hybrid setting.
Examples include extensions of local \cite{simic2001structural} and global \cite{broucke2001structural} structural stability results, contraction analysis \cite{burden2018generalizing,burden2018contraction}, and many theoretical tools concerning periodic orbits \cite{Burden_Sastry_Koditschek_Revzen_2016} including Floquet theory \cite{Burden_Revzen_Sastry_2015} and the Poincar\'{e}-Bendixson theorem \cite{clark2019poincare,clark2020poincare,simic2002hybrid}.

As described earlier, to prove our results we introduce what we call the \emph{hybrid suspension} of a \THS~ defined in terms of what we call the \emph{relaxed} version of a \THS.
In writing this paper we have become aware that versions of the relaxed system and hybrid suspension have previously appeared in the hybrid systems literature under different names, although (to the best of our knowledge) only for classes of hybrid systems which are formally less general (imposing more structure) than \THS~ and \MHS~ along several dimensions.
The relaxed system is essentially an example of a ``temporal relaxation'' in the sense of \cite{johansson1999regularization}.
The hybrid suspension $\Sus_H$ of a \THS~ $H$ could be called a ``1-relaxed hybrid quotient space'' in the terminology of \cite{burden2015metrization} or a ``homotopy colimit'' in the terminology of \cite{ames2005homology}.
More details are given in SM \S \ref{app:relaxed} and \S \ref{app:hybrid-suspension-subsubsec}.
The hybrid suspension coincides with (a mild generalization of) the ``hybrifold'' (introduced in \cite{simic2000towards,simic2005towards}) of the relaxed version of the original hybrid system but, crucially for us, not with the ``hybrifold'' of the original hybrid system itself; see Remarks~\ref{rem:omega-compatibility-for-hybrifold-fails} and \ref{rem:chain-compatibility-for-hybrifold-fails} and SM \S \ref{app:hybrifold} for more details.

Finally, our definitions of \THS~ and \MHS~ build on a long history of formal approaches to hybrid automata \cite{simic2005towards, haghverdi2005bisimulation, Johnson_Burden_Koditschek_2016, lerman2016category, culbertson2019formal}.
Most specifically, our definition of hybrid $(\epsilon, T)$-chains is almost identical to the definition in \cite{culbertson2019formal} for smooth hybrid systems (with one important difference; see SM \S \ref{app:CGKS-chains}).

\section{Preliminaries}\label{sec:preliminaries}
\subsection{Two classes of hybrid systems}
Most definitions of hybrid systems in the literature involve variants of smooth manifolds and vector fields.
However, in the same way that ``the'' natural setting for the classical theory of smooth dynamics is given by sufficiently smooth flows generated by vector fields (or by sufficiently smooth maps) on smooth manifolds, ``the'' natural setting for the classical theory of topological dynamics is given by continuous flows (or semiflows, or continuous maps) on topological (or metric) spaces.
Since Conley's theory is part of the theory of topological dynamics, in Def.~\ref{def:HybridSystem} we define two classes of hybrid systems which we believe provide a natural setting for a topological-dynamical theory of hybrid systems.
These definitions have enabled us to obtain in this paper more natural (and general) results than would be obtainable using definitions already existing in the literature.
Beyond such aesthetic considerations, we speculate (see \S \ref{app:relaxed}---particularly, Footnote~\ref{foot:semiflow}) that these coarser  topological methods may actually be necessary  to handle the non-smooth features intrinsic to physically important  hybrid dynamics models.
Fortunately (see Remark \ref{rem:wlog-disjoint-union}), our results  apply to many previously defined classes of hybrid systems that represent special cases of our problem setting.

Following \cite[Sec.~1.3]{hirsch2006monotone}, a \concept{local semiflow} $\varphi$ on a topological space $\fs$ is a map $\varphi\colon \dom(\varphi) \to \fs$ defined on an open neighborhood $\dom(\varphi) \subseteq [0, \infty) \times \fs$ of $\{0\}\times \fs$ satisfying the following conditions, with $\varphi^t\coloneqq \varphi(t,\slot)$ and $t,s \in [0,\infty)$: (i) $\varphi^0 = \id_\fs$, (ii) $(t+s,x)\in \dom(\varphi) \iff$ both $(s,x)\in \dom(\varphi)$ and  $(t,\varphi^s(x))\in \dom(\varphi)$, and (iii) for all $(t+s,x)\in \dom(\varphi)$, $\varphi^{t+s}(x) = \varphi^t(\varphi^s(x))$.
Given $x\in \fs$, the map $t\mapsto \varphi^t(x)$ defined on some interval is called a \concept{trajectory} of $\varphi$.
The local semiflow $\varphi$ is a \concept{semiflow} if $\dom(\varphi) = [0,\infty) \times \fs$.

The following definition follows much of the terminology of \cite{culbertson2019formal}, but uses a simpler model for the discrete-time dynamics.\footnote{More specifically, we ignore any underlying graph structure and the fact that there may be various distinct ``modes''; see Remark~\ref{rem:wlog-disjoint-union} for more details.}
Our definition of \emph{topological hybrid systems (\THS)} uses a more general model for the continuous-time dynamics, i.e., local semiflows on topological spaces instead of vector fields on manifolds. 
For this reason, our definition of \emph{metric hybrid systems (\MHS)} differs from that of \cite[Def.~2.17]{culbertson2019formal}.\footnote{Applications-oriented readers might be interested to consult footnote~\ref{foot:semiflow} for a brief discussion motivating the (essentially imperative) benefits of adopting this more general framework.
}

\begin{Def}[Topological and metric hybrid systems]
\label{def:HybridSystem}
A \concept{topological hybrid system (\THS)} $H = (\st,\fs,\gd,\varphi,r)$ consists of:
\begin{description}
\item[States]  a topological \emph{state space} $\st$ whose points are the possible states of the system.
\item[Continuous-time dynamics] a continuous local semiflow $\varphi$ defined on an open \emph{flow set} $\fs \subseteq \st$.
\item[Discrete-time dynamics] a closed \emph{guard set} $\gd \subseteq \st$ equipped with a continuous \emph{reset map} $r : \gd \to \st$.
\end{description} 
If the topology of $\st$ arises from an extended metric $\dist: \st \times \st \to [0, +\infty]$, we say that $(H,\dist)$ is a \concept{metric hybrid system (\MHS)}.
(We will usually suppress the extended metric $\dist$ from the notation.)
\end{Def}

\begin{figure}
	\centering
	\def\svgwidth{1.0\columnwidth}
	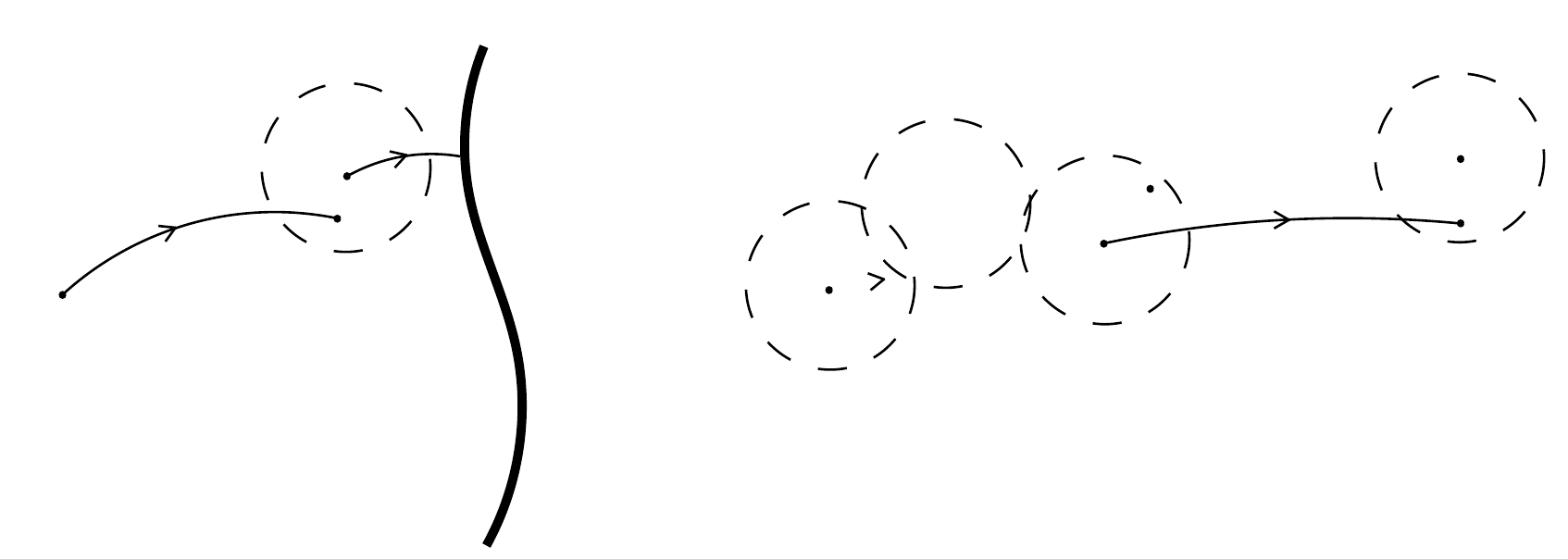
	\caption{An execution in a \THS~ with initial state $x\in \fs\subseteq \st$.
	In this example, $\tau_1 = \tau_2$ and $r(\gamma_0(\tau_1)) \in \gd$, highlighting the fact that we allow $r(\gd) \cap \gd \neq \varnothing$ in Def. \ref{def:HybridSystem}. }\label{fig:execution}
\end{figure}

\begin{Rem}\label{rem:wlog-disjoint-union}
	Typical definitions of hybrid systems specify several distinct state spaces (usually smooth manifolds with corners), often called \emph{modes}, each equipped with its own local semiflow (usually generated by a locally Lipschitz vector field).
	Each mode may contain a \emph{guard set}, and discrete transitions from the guard sets to other modes are specified by \emph{reset} maps (sometimes allowed to be more general relations). 
	Examples of references containing this style of definition include \cite{simic2000towards,simic2001structural,simic2002hybrid,lygeros2003dynamical,simic2005towards,haghverdi2005bisimulation,Burden_Revzen_Sastry_2015,lerman2016category,Johnson_Burden_Koditschek_2016, burden2018contraction, burden2018generalizing,culbertson2019formal,clark2020poincare,lerman2020networks}.
	
	At first glance one might incorrectly assume that Def. \ref{def:HybridSystem} can encapsulate only hybrid systems consisting of a single mode.
	However, this is not the case: given a hybrid system as defined in one of the aforementioned references (and having closed guards and continuous reset maps), by defining $\st$ to be the disjoint union of the modes, $\fs$ to be the disjoint union of the flow sets, $\gd$ to be the disjoint union of the guard sets, $r\colon \gd\to \st$ to be the disjoint union of the reset maps, and $\varphi$ to be the disjoint union of the local semiflows, one obtains a \THS~ as in Def. \ref{def:HybridSystem}.
	If additionally each of the modes of the given hybrid system is equipped with a compatible extended metric (by the Urysohn metrization theorem \cite[Thm~34.1]{munkres2000topology}, such a metric always exists if each mode is a smooth, paracompact manifold with corners), then one further obtains an \MHS~ as in Def. \ref{def:HybridSystem} by leaving the distance between points in the same mode unchanged and defining the distance between points in distinct modes to be infinite.
	Thus, our results (including Theorems \ref{th:conley-decomp} and \ref{th:hybrid-conley}) can be applied to such hybrid systems, as long as they satisfy the relevant additional hypotheses.
	
\end{Rem}

\begin{Def}\label{def:execution}
Given a \THS~ $H=(\st,\fs,\gd,\varphi,r)$, an \concept{execution} in $H$ is a tuple $\chi = (N, \tau, \gamma)$ of 
\begin{enumerate}[label=\ref*{def:execution}.\arabic*.\hspace{0.1cm}, ref=\ref*{def:execution}.\arabic*,leftmargin=*]
\item \textbf{Jump times:} a nondecreasing sequence $\tau = (\tau_j)_{j=0}^{N+1} \subseteq \mathbb{R} \cup \{+\infty\}$ where $N\in \N_{\geq 0} \cup\{+\infty\}$, $\tau_0 = 0$, and $(\tau_j)_{j=0}^{N} \subseteq \mathbb{R}$.
\item \textbf{Flow arcs:} a sequence of continuous maps $\gamma = (\gamma_j\colon T_j\to \st)_{j = 0}^{N}$  with $[\tau_j,\tau_{j+1}) \subseteq T_j\subseteq [\tau_j,\tau_{j+1}]\cap \R$, $\gamma_j([\tau_{j}, \tau_{j+1})) \subseteq \fs$, and such that the restriction $\gamma_j|_{[\tau_{j}, \tau_{j+1})} : [\tau_{j}, \tau_{j+1}) \to \fs$ is a trajectory segment for the local semiflow $\varphi$.  
For all $0 \leq j < N$, we additionally require that $T_j = [\tau_j,\tau_{j+1}]$, $\gamma_{j}(\tau_{j+1}) \in \gd$, and $\gamma_{j+1}(\tau_{j+1})= r(\gamma_{j}(\tau_{j+1}))$.
\end{enumerate}

We call $\gamma_0(0)$ the \concept{initial state} of $\chi$. If $N < \infty$ and $\tau_{N+1} \in T_N$, we call $\gamma_N(\tau_{N+1})$ the \concept{final state} of $\chi$.
We define the \concept{stop time} of $\chi$ to be
\[
\tau^\stp :=
\begin{cases}
\tau_{N+1} & N < \infty\\
\displaystyle \lim_{j \to \infty} \tau_j & N = \infty
\end{cases}.
\]

If the stop time of $\chi$ is infinite, we say that $\chi$ is an \concept{infinite execution}.\footnote{Our definition of ``infinite execution'' follows that of \cite[Def.~2.13]{culbertson2019formal}. 
However, this definition differs from that of \cite[p.~4]{lygeros2003dynamical}, which refers to both executions having infinite stop time \emph{and} Zeno executions as ``infinite.'' 
Since our Def.~\ref{def:deterministic-nonblocking} and \cite[Def.~III.1]{lygeros2003dynamical} both define ``nonblocking'' hybrid systems to be those for which all maximal executions are ``infinite,'' our definitions of ``nonblocking'' thus also differ in meaning.
Our definition of ``nonblocking'' differs from that of \cite[Def.~4]{Johnson_Burden_Koditschek_2016} in precisely the same way, although the latter reference does not introduce the ``infinite execution'' terminology.}
If the stop time of $\chi$ is finite but $\chi$ has infinitely many jumps ($N = \infty$), we say that $\chi$ is a \concept{Zeno execution}.  
We say that $\chi=(N,\tau,\gamma)$ is a \concept{maximal execution} if, for any execution $\chi' = (N',\tau',\gamma')$ with $\gamma'_0(0) = \gamma_0(0)$ and each $\gamma_j$ equal to $\gamma'_{j'}|_{T_j}$ for some $j'\in \{0,\ldots, N'\}$, $\chi = \chi'$.

We denote the set of executions in $H$ and executions with initial state $x\in \st$ by $\exc_H$ and $\exc_H(x)$, respectively. 
For $\chi = (N, \tau, \gamma) \in \exc_H$  and any $t\in \bigcup_{j=0}^N T_j$, we write $\chi(t) = \{\gamma_j(t) \mid 0\leq j < N + 1, t\in T_j\}$.
 We emphasize that $\chi(t)$ is generally a set with multiple elements if $t \in \{\tau_1,\ldots, \tau_N\}$, but otherwise $\chi(t)$ is a singleton.    
\end{Def}

\begin{Rem}\label{rem:def-2-big-remark}
  An important point concerning Def.~\ref{def:execution} is that, if $\gamma_j(\tau_j)\in \gd$ for some $0\leq j \leq N$, then it is possible that $\tau_{j+1} = \tau_j$, i.e., an instantaneous reset might occur (and \emph{must} occur if $H$ is \emph{deterministic};  see Def.~\ref{def:deterministic-nonblocking} below).
  In this case, the condition $\gamma_j([\tau_{j}, \tau_{j+1})) = \varnothing \subseteq \fs$ is satisfied vacuously.
  A similar remark applies to Def.~\ref{def:eps-T-chains} below.

 We also note that, since the domain of the final arc of an execution is not required to be closed, Def.~\ref{def:execution} allows the final arc to be a trajectory of $\varphi$ which ``blows up'' or ``escapes'' in finite time (i.e., it cannot be extended to a $\varphi$ trajectory defined for all nonnegative time; see \cite[Sec.~1.3]{hirsch2006monotone} for the latter terminology).	
  However, our main results concerning \THS~ $H=(\st,\fs,\gd,\varphi,r)$ include the assumption that every maximal execution of $H$ is infinite or Zeno, and this assumption implies that $\varphi$ trajectories may only ``artificially'' escape $\fs$ in finite time by converging to a limit in $\gd$.

  As in \cite[Rem.~2.16]{culbertson2019formal}, our definition of execution  allows for infinitely
  many jumps in finite time (Zeno executions), but does not allow for subsequent execution after
  the stop time. 
  In particular, while we do allow Zeno executions, in this paper we do not explicitly consider continuations of Zeno executions past the stop time. (Zeno continuations are considered, e.g., in \cite{johansson1999regularization,ames2006there,or2010stability,Johnson_Burden_Koditschek_2016,goebel2016well}.)
\end{Rem}

\begin{Rem}\label{rem:FcupZ}
If $x\in \st\setminus (\fs \cup \gd)$, then Def.~\ref{def:execution} implies that the only execution $\chi = (N,\tau,\gamma) \in \exc_H(x)$ is the \concept{trivial execution}: $N = 0$, $\tau = (0,0)$, $\gamma_0\colon \{0\}\to \{x\}$.
It follows that, if every $x\in \st$ has an execution $\chi \in \exc_H(x)$ which is not trivial, then $\st = \fs \cup \gd$.
In particular, if every $x \in \st$ has an infinite or Zeno execution $\chi \in \exc_H(x)$, it follows that $\st = \fs \cup \gd$.
\end{Rem}

\begin{Def}\label{def:deterministic-nonblocking}
A \THS~ $H=(\st,\fs,\gd,\varphi,r)$ is \textbf{deterministic} if $\gd \cap \fs = \varnothing$.    
A \THS~ $H$ is \textbf{nonblocking} if for every $x \in \st$ there is an infinite execution starting at $x$.
\end{Def}
For a deterministic \THS, maximal executions are unique; cf. \cite[Prop.~2.21]{culbertson2019formal}.
This justifies the terminology.
For a deterministic \THS, we use the notation $\chi_x$ to denote the unique maximal execution in $\exc_H(x)$.

Most of our results do \emph{not} assume the nonblocking condition; rather, most of our results (including Theorems~\ref{th:conley-decomp} and \ref{th:hybrid-conley}) assume the weaker condition that all maximal executions of a given \THS~ are either infinite or Zeno.

\subsection{Hybrid chain equivalence, recurrence, and attracting-repelling pairs}

The following definition of $(\epsilon, T)$-chains is similar to \cite[Def.~ 2.18]{culbertson2019formal} (see Remark~\ref{rem:at-least-two-arcs} and SM \S\ref{app:CGKS-chains} for a comparison).  As we show below, the corresponding Conley relation generalizes the classical notions for discrete-time and continuous-time systems on compact metric spaces. We recommend the reader glance at Fig. \ref{fig:eps-T-chain} before reading the formal definition.
We remark that $(\epsilon,T)$-chains can be viewed informally as ``executions with errors,'' with the values of $\epsilon, T$ determining the admissible errors.

\begin{Def}\label{def:eps-T-chains}
Given an \MHS~ $H=(\st,\fs,\gd,\varphi,r)$ and $\epsilon, T \geq 0$, an \concept{${(\epsilon, T)}$-chain} in $H$ is a tuple $\chi = (N, \tau, \eta, \gamma)$ of 
\begin{enumerate}[label=\ref*{def:eps-T-chains}.\arabic*.\hspace{0.1cm}, ref=\ref*{def:eps-T-chains}.\arabic*,leftmargin=*]
\item \textbf{Jump times:} a nondecreasing sequence $\tau = (\tau_j)_{j=0}^{N+1} \subseteq \mathbb{R}$ where $N\in \N_{\geq 1}$ and $\tau_0 = 0$.
\item \textbf{Continuous jump times:} a subsequence $(\tau_{\eta_k})_{k=0}^{M}$ of $\tau$ such that $\eta_0 = 0$, $\eta_M\leq N$, and $\tau_{\eta_{k}} - \tau_{\eta_{k-1}} \geq T$ for all $k \geq 1$.
\item\label{enum:flow-arcs} \textbf{Flow arcs:} a sequence $(\gamma_j)_{j = 0}^{N}$  of continuous maps $\gamma_j : [\tau_j, \tau_{j+1}] \to \st$ such that $\gamma_j([\tau_{j}, \tau_{j+1})) \subseteq \fs$ and the restriction $\gamma_j|_{[\tau_{j}, \tau_{j+1})} : [\tau_{j}, \tau_{j+1}) \to \fs$ is a trajectory segment for the local semiflow $\varphi$.  In addition, we require the following:
  \begin{enumerate}[label=\ref*{enum:flow-arcs}.\arabic*.\hspace{0.1cm}, ref=\ref*{def:eps-T-chains}.\arabic*,leftmargin=*]
        \item \textbf{Continuous-time jump condition:} If $j = \eta_k$ for some $k \geq 1$, then $\gamma_{j-1}(\tau_{j}) \in \fs$ and
        $$\dist(\gamma_{j}(\tau_{j}), \gamma_{j-1}(\tau_{j})) \le \epsilon;$$
    \item \textbf{Reset jump condition:} If $1\leq j \leq  N$ and $j \neq \eta_k$ for any $k$, then  $\gamma_{j-1}(\tau_{j}) \in \gd$ and $$\dist(\gamma_{j}(\tau_{j}), r(\gamma_{j-1}(\tau_{j}))) \le \epsilon.$$
   \end{enumerate}
\end{enumerate}

As in Def. \ref{def:execution}, we call $\gamma_0(0)$ the \concept{initial state} of $\chi$ and $\gamma_N(\tau_{N+1})$ the \concept{final state} of $\chi$.
We denote the set of $(\epsilon, T)$-chains in $H$, $(\epsilon,T)$-chains with initial state $x$, and $(\epsilon,T)$-chains with initial state $x$ and final state $y$ by $\conleyet_H$, $\conleyet_H(x)$ and $\conleyet_H(x,y)$, respectively. For $\chi = (N, \tau, \eta, \gamma) \in \conleyet_H$  and $0 \le t \le \tau_{N+1}$, we write $\chi(t) = \{\gamma_j(t) \mid 0\leq j \leq N, \tau_j \le t \le \tau_{j+1} \}$.
\end{Def}

\begin{Rem}\label{rem:at-least-two-arcs}
	Note that, unlike Def.~\ref{def:execution} (and \cite[Def.~2.18]{culbertson2019formal}), Def.~\ref{def:eps-T-chains} requires that $N \geq 1$.
	I.e., we require that an $(\epsilon,T)$-chain consist of at least two arcs.
\end{Rem}

\begin{figure}
	\centering
	\def\svgwidth{1.0\columnwidth}
	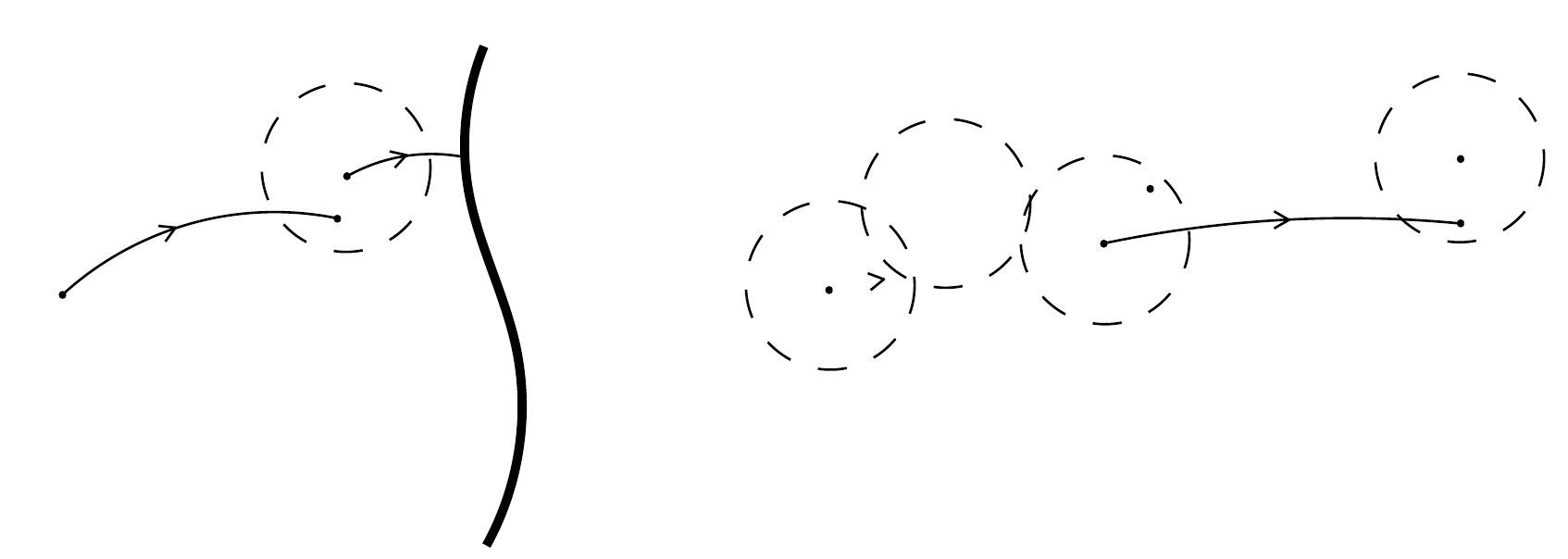
	\caption{An $(\epsilon,T)$-chain from $x\in \st$ to $y\in \st$ with $N = 5$. 
		In this example, both $x,y\in \fs\subseteq \st$. 
		All open balls are of radius $\epsilon$, and we have $\tau_1, (\tau_4-\tau_1), (\tau_5-\tau_4) \geq T$, but $(\tau_4 - \tau_3) < T$. 
		In this example, the subsequence $(\eta_k)_{k=0}^{M}$ is given by $\eta_0 = 0$, $\eta_1 = 1$, $\eta_2 = 4$, and $\eta_3 = 5$ so that there are $M = 3$ continuous-time jumps. 
		Note that the final arc is not required to be degenerate (a single point), in contrast with the classical definitions for continuous-time (semi-)dynamical systems \cite{conley1978isolated,hurley1995chain}; however, Ex.~\ref{ex:special-case-continuous-time} shows that this difference is immaterial in the continuous-time setting as far as the Conley relation (Def.~\ref{def:hybrid-conley-relation}) is concerned.
		Note also that the ``double jump'' $r(\gamma_2(\tau_2))\rightsquigarrow \gamma_3(\tau_3)$, $\gamma_3(\tau_4) \rightsquigarrow \gamma_4(\tau_4)$ is permitted, even though $(\tau_4 - \tau_3) < T$, because $(\tau_4 - \tau_1) = (\tau_{\eta_2} - \tau_{\eta_1}) \geq T$. 
		Later in this paper we will also consider \textbf{nice} $(\epsilon,T)$-chains in which ``double jumps'' are not allowed; see Def.~\ref{def:nice-chains-conley-rel} and Fig. \ref{fig:eps-T-chain-nice}.}\label{fig:eps-T-chain}
\end{figure}

\begin{Def}[Hybrid Conley relation]\label{def:hybrid-conley-relation}
	Let $H=(\st,\fs,\gd,\varphi,r)$ be an \MHS. The \concept{(hybrid) Conley relation} $\conley_H \subseteq \st \times \st$ is defined by
	$$(x,y)\in \conley_H \iff \text{for all } \epsilon, T > 0\colon  \conleyet_H(x,y) \neq \varnothing.$$
    As is standard with relations, we sometimes use the more intuitive notation $\conley_H(x,y)$ in place of $(x,y)\in \conley_H$.

\end{Def}

\begin{Rem}\label{rem:conley-rel-compact-metric-independent}
Let $H=(\st,\fs,\gd,\varphi,r)$ be a \THS~ such that $\st$ is metrizable and compact.
Then---since all compatible extended metrics on a compact metrizable space are uniformly equivalent---the Conley relation is independent of the choice of compatible metric on $\st$ making $H$ into an \MHS.
Since the hypotheses for our main results (Theorems~\ref{th:conley-decomp} and \ref{th:hybrid-conley}) include the  assumption that $\st$ is compact, the specific choice of extended metric is immaterial for the majority of our purposes in this paper.
\end{Rem}

The following two examples show that our hybrid Conley relation generalizes the classical notions in the discrete-time and continuous-time settings.

\begin{Ex}\label{ex:special-case-discrete-time}
  Discrete-time (semi-)dynamical systems given by iterating a continuous map are instances of \THS~ $H = (\st,\fs,\gd,\varphi,r)$ where $\st = \gd$ and $\fs = \varnothing$.
  If $H$ is also an \MHS, then for any chain $\chi = (N, \tau, \eta, \gamma) \in \conleyet_H(x,y)$, we always have $(\eta_k) = (0)$ and $\tau_j = 0$ for all $j$ (because time never elapses).  
  Each arc $\gamma_j$ is degenerate and corresponds to a single point.  
  Thus $T$ is irrelevant, and we recover the usual notion of an $\epsilon$-chain for a discrete-time system which can also be specified by the more standard notation
  $$\chi = (x_0 = x, x_1, \ldots, x_N = y),$$
where $x_i = \gamma_i(\tau_i) = \gamma_i(0)$ for all $0\leq i \leq N$.  Thus, our notion of $(\epsilon,T)$-chain restricts to the classical notion when $H$ is a discrete-time system.
\end{Ex}

\begin{Ex}\label{ex:special-case-continuous-time}
  Suppose $H = (\st,\fs,\gd,\varphi,r)$ is a continuous-time (semi-)dynamical system given by a continuous semiflow.  
  That is, $\st = \fs$ and $\gd = \varnothing$.  Then if $H$ is also an \MHS,  the classical notion of an $(\epsilon,T)$-chain for a semiflow \cite{conley1978isolated,hurley1995chain} is usually expressed (modulo indexing conventions) as a tuple
  $$\chi = (x_0 = x, x_1, \ldots, x_N = y; t_1, \ldots, t_N),$$
  where each $t_i \ge T$ and $\dist(\varphi^{t_{i}}(x_{i-1}), x_{i}) < \epsilon$ for all $1 \le i \le N$.  It is easy to see that a classical $(\epsilon,T)$-chain always corresponds to an $(\epsilon,T)$-chain as in Def.~\ref{def:eps-T-chains}; however, the converse does not hold. 
  This is because a classical $(\epsilon,T)$-chain always ends with a jump \cite{conley1978isolated,hurley1995chain}---or equivalently, using the terminology of Def.~\ref{def:eps-T-chains}, ends with a degenerate arc---but an $(\epsilon,T)$-chain in the sense of Def.~\ref{def:eps-T-chains} can end with a nondegenerate arc (see Fig.~\ref{fig:eps-T-chain}).

  However, if $\st = \fs$ is compact then the corresponding classical Conley relation is equal to the Conley relation $\conley_H$.  
  Indeed, if $\epsilon, T > 0$, we can use the uniform continuity of $\varphi$ on $[0,T]\times \st$ to pick $\delta \in (0, \epsilon)$ such that $\dist(p,q) < \delta$ implies $\dist(\varphi^t(p), \varphi^t(q) < \epsilon$ for all $t \in [0,T]$.  
  Let $\chi = (N, \tau, \eta, \gamma) \in \conley^{\delta,T}_H(x,y)$; if $\tau_{N+1}-\tau_N\geq T$, then we can produce a classical $(\epsilon,T)$-chain by simply adding the degenerate arc at $\gamma_{N}(\tau_{N+1})$ to $\chi$.
  If instead $\tau_{N+1}-\tau_N < T$, then we can produce a classical $(\epsilon,T)$-chain from $x$ to $y$ by replacing the final arc with the degenerate arc at its terminal point, and extending the penultimate arc by $\tau_{N+1} - \tau_N$ (the penultimate arc necessarily exists since $N \ge 1$ by Def.~\ref{def:eps-T-chains}).
\end{Ex}

\begin{Def}[Hybrid chain recurrent set]\label{def:hybrid-chain-recurrent}
	Let $H=(\st,\fs,\gd,\varphi,r)$ be an \MHS.  
	The \concept{(hybrid) chain recurrent set} $R(H) \subseteq \st$ is defined by
	$$R(H) = \{ x \in \st \mid \conley_H(x,x) \}$$
\end{Def}

\begin{Def}[Hybrid chain equivalence]\label{def:hybrid-chain-equivalence}
	Let $H=(\st,\fs,\gd,\varphi,r)$ be an \MHS. Two points $x,y \in \st$ are \concept{chain equivalent}  if $\conley_H(x,y)$ and $\conley_H(y,x)$.
	The \concept{chain equivalence class} of $x\in \st$ is the set $\{y\in \st\mid \conley_H(x,y) \textnormal { and } \conley_H(y,x)\}$.
	(Note that every chain equivalence class is a subset of $R(H)$.)
\end{Def}
The following definition generalizes the definition of $\omega$-limit set in \cite[Def.~II.7]{lygeros2003dynamical} (which considered $\omega$-limit sets of singletons only), in addition to generalizing two of the standard definitions for discrete-time and continuous-time dynamical systems \cite[p.~29, II.4.1.B]{brin2002introduction,conley1978isolated}.

\begin{Def}[Hybrid $\omega$-limit set]\label{def:omega-limit}
	Let $H = (\st,\fs,\gd,\varphi,r)$ be a \THS.
    If $U\subseteq \st$, we define the $\omega$-limit set $\omega(U)$ of $U$ via  $$\omega(U)\coloneqq \bigcap_{T > 0}\cl\left( \bigcup_{x\in U}\,\bigcup_{\chi = (N, \tau,\gamma) \in \exc_H(x)} \,  \{\gamma_j(t)\mid j + t \geq T\}\right).$$
    If $x\in \st$, we define $\omega(x)\coloneqq \omega(\{x\})$.
\end{Def}
\begin{Rem}
	If $\st$ is compact, $U\subseteq \st$, and there exists $x\in U$ with an infinite or Zeno execution $\chi\in \exc_H(x)$, then $\omega(U)$ is a decreasing intersection of nonempty compact sets and is therefore compact and nonempty.
\end{Rem}
The following result shows that, as in the classical setting, an infinite or Zeno execution with initial state $x$ converges to $\omega(x)$ if $\st$ is compact.

\begin{Prop}[Convergence to $\omega$-limit set]\label{prop:conv-omega}
    Let $H = (\st,\fs,\gd,\varphi,r)$ be a \THS~ with $\st$ compact. Then for any $x\in \st$ and infinite or Zeno execution $\chi = (N, \tau,\gamma) \in \exc_H(x)$, we have
	$$\gamma_j(t) \to \omega(x) \qquad \textnormal{as } j + t\to \infty.$$
	That is, for every neighborhood $U\supseteq \omega(x)$, there exists $M > 0$ such that $\gamma_j(t)\in U$ for all $j+t > M$.
\end{Prop}
\begin{proof}
	Suppose not. 
	Then there exists an open neighborhood $U\supseteq \omega(x)$ and subsequences $(j_k)_{k\in \N}, (t_k)_{k\in \N}$ with $j_k + t_k \to \infty$ such that $\gamma_{j_k}(t_k)\in \st\setminus U$ for all $k$.
	Compactness of $\st\setminus U$ implies that the sequence $(\gamma_{j_k}(t_k))_{k\in \N}$ has a limit (accumulation) point $y\in \st\setminus U$.
	But Def. \ref{def:omega-limit} implies that $y\in \omega(x)$, and $\omega(x)$ is disjoint from $\st\setminus U$ by the definition of $U$, so we have a contradiction. 
\end{proof}

\begin{Def}[Forward invariance]\label{def:forward-invariant}
	Let $H = (\st,\fs,\gd,\varphi,r)$ be a \THS~ and $S\subseteq \st$ a subset.
	We say that	$S$ is \concept{forward invariant} if for every $x\in S$, $\chi \in \exc_H(x)$, and $t \geq 0$, $\chi(t)\subseteq S$.
\end{Def}

The following example shows that, in spite of Prop. \ref{prop:conv-omega}, hybrid $\omega$-limit sets and chain recurrent sets can display behavior very different from that of a classical (semi-)dynamical system.
In particular, the example shows that the hybrid chain recurrent set need not generally contain the ``steady state behavior.''
Such wild deviance motivates the introduction, in \S\ref{sec:main-results}, of the \emph{trapping guard condition}; a \THS~ satisfying the trapping guard condition does not suffer from such pathologies (see Cor.~\ref{co:omega-forward-invariant} and \ref{co:hybrid-chain-recurrent-set-closed-invariant}).

\begin{figure}
	\centering
	\def\svgwidth{0.4\columnwidth}
\begingroup%
  \makeatletter%
  \providecommand\color[2][]{%
    \errmessage{(Inkscape) Color is used for the text in Inkscape, but the package 'color.sty' is not loaded}%
    \renewcommand\color[2][]{}%
  }%
  \providecommand\transparent[1]{%
    \errmessage{(Inkscape) Transparency is used (non-zero) for the text in Inkscape, but the package 'transparent.sty' is not loaded}%
    \renewcommand\transparent[1]{}%
  }%
  \providecommand\rotatebox[2]{#2}%
  \newcommand*\fsize{\dimexpr\f@size pt\relax}%
  \newcommand*\lineheight[1]{\fontsize{\fsize}{#1\fsize}\selectfont}%
  \ifx\svgwidth\undefined%
    \setlength{\unitlength}{193.96692019bp}%
    \ifx\svgscale\undefined%
      \relax%
    \else%
      \setlength{\unitlength}{\unitlength * \real{\svgscale}}%
    \fi%
  \else%
    \setlength{\unitlength}{\svgwidth}%
  \fi%
  \global\let\svgwidth\undefined%
  \global\let\svgscale\undefined%
  \makeatother%
  \begin{picture}(1,0.16799845)%
    \lineheight{1}%
    \setlength\tabcolsep{0pt}%
    \put(0,0){\includegraphics[width=\unitlength,page=1]{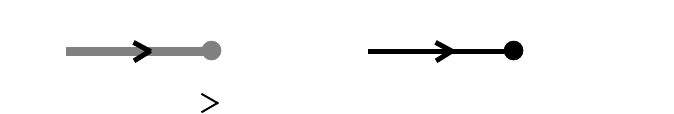}}%
    \put(0.33039117,0.13111328){\color[rgb]{0,0,0}\makebox(0,0)[lt]{\lineheight{1.25}\smash{\begin{tabular}[t]{l}$-2$\end{tabular}}}}%
    \put(0,0){\includegraphics[width=\unitlength,page=2]{omega-chain-pathology.pdf}}%
    \put(0.70932179,0.13111328){\color[rgb]{0,0,0}\makebox(0,0)[lt]{\lineheight{1.25}\smash{\begin{tabular}[t]{l}$0$\end{tabular}}}}%
    \put(0,0){\includegraphics[width=\unitlength,page=3]{omega-chain-pathology.pdf}}%
    \put(-0.00213974,0.13111328){\color[rgb]{0,0,0}\makebox(0,0)[lt]{\lineheight{1.25}\smash{\begin{tabular}[t]{l}$-3$\end{tabular}}}}%
  \end{picture}%
\endgroup%

	\caption{Depicted here is the \MHS~ $H = (\st,\fs,\gd,\varphi,r)$ discussed in Ex.~\ref{ex:omega-chain-pathologies}. 
	For this example $\st = [-3,-2] \cup [-1,0]\cup\{1\}\subseteq \R$, $\gd = \{-3,-2,0\}$, and $\fs = \st\setminus \gd$. The local semiflow $\varphi$ is generated by the vector fields $\frac{\partial}{\partial t}$ on $[-3,-2]$, $-t\frac{\partial}{\partial t}$ on $[-1,0]$, and $0 \frac{\partial}{\partial t}$ on $\{1\}$.  
	$r\colon \gd\to \st$ is defined by $r(-3) = -1$, $r(-2) = -3$, and $r(0) = 1$.
	Clearly $H$ is deterministic and nonblocking.
	For this example, $R(H) = (-3,-2]\cup \{1\}$ (shown in gray).
	Furthermore, (i,ii) $\omega(-1) = \{0\}$ is neither forward invariant nor contained in $R(H)$, (iii) $R(H)$ is not forward invariant since $r(r(-2)) = -1 \not \in R(H)$, and (iv) $R(H)$ is not closed in $\st$.  }\label{fig:omega-chain-pathology}
\end{figure}

\begin{Ex}
\label{ex:omega-chain-pathologies}
Fig.~\ref{fig:omega-chain-pathology} and its caption specify an example which shows that, for a general \MHS~ $H = (\st,\fs,\gd,\varphi,r)$, the following hold: (i) $\omega$-limit sets need not be forward invariant, (ii)  $\omega$-limit sets need not be contained in $R(H)$, (iii) $R(H)$ need not be forward invariant, and (iv) $R(H)$ need not be closed.  
Furthermore, this example shows that these pathologies can occur even if $\st$ is compact and $H$ is deterministic and nonblocking.
\end{Ex}

\begin{Def}[Hybrid trapping sets and attracting-repelling pairs] \label{def:attracting-repelling}
	Let $H = (\st,\fs,\gd,\varphi,r)$ be a \THS.
	We say a precompact set $U\subseteq \st$ is a \concept{trapping set} if the following hold.
	\begin{itemize}
		\item $U$ is forward invariant. 
		\item There exists $T > 0$ such that
		$$\cl\left( \bigcup_{x\in U}\,\bigcup_{\chi = (N,\tau,\gamma) \in \exc_H(x)} \,  \{\gamma_j(t)\mid j + t \geq T\}\right)\subseteq \interior(U).$$
	\end{itemize}
    If $U$ is a trapping set, we define the \concept{attracting set} $A$ determined by $U$ to be $A\coloneqq \omega(U).$
    We define the \concept{repelling set} $A^*$ dual to $A$ as $A^*\coloneqq \{x\in \st\mid \omega(x) \cap A = \varnothing\}$.\footnote{For a continuous-time dynamical system given by a \emph{flow}, Conley showed that $A^*$ is also an attracting set for the time-reversed flow \cite[p.~32]{conley1978isolated}, so that there is a duality between attracting and repelling sets.
    For a semiflow, this duality no longer holds as stated because the time-reversed flow is not well-defined. However, we still choose to use the terminology ``dual repelling set'' for the case of a semiflow (as is somewhat common in the literature \cite{rybakowski1983morse,rybakowski1985morse,franzosa1988connection}) and, more generally, for the case of a \THS. 
    Additionally, we have chosen to use the terminology ``attracting and repelling sets'' rather than ``attractors and repellers'' \cite[II.5.1]{conley1978isolated} because ``care is needed since the literature contains many variations on the precise definitions'' of the latter \cite{Milnor_2006} (see also \cite{milnor1985concept,milnor1985conceptII}).}
\end{Def}

\section{Main results}\label{sec:main-results}
Our main results concern deterministic \THS~having \emph{trapping guards}, which we define below using the maximum flow time $\mu\colon \st \to [0,+\infty]$ given by
\begin{equation}\label{eq:max-flow-time-all}
\mu(x)\coloneqq
\begin{cases}
\sup \{t\in [0,\infty)\colon (t,x) \in \dom(\varphi)\}, & x \in \fs \\
0, & x \not \in \fs
\end{cases}.
\end{equation}

\begin{Def}[Trapping guards]\label{def:trapping-guards}
	Let $H=(\st,\fs,\gd,\varphi,r)$ be a deterministic \THS.  	
	Let $\cl(\dom(\varphi))$ be the closure of $\dom(\varphi)$ in $[0,\infty) \times \st$.
	We say that $\gd$ is a \concept{flow-induced retract} if there exists a neighborhood $U \subseteq \st$ of $\gd$ and a continuous retraction $\rho \colon U \to \gd$ ($\rho|_\gd = \id_\gd$) such that (i) the $U$-restricted maximum flow time $\mu|_{U}\colon U\to \R$ is continuous, and (ii) such that $\varphi|_{\dom(\varphi)\cap ([0,\infty) \times U)}$ admits a unique continuous extension $\widehat{\varphi}$ to $\cl(\dom(\varphi))\cap ([0,\infty) \times U)$ given by
\begin{equation}\label{eq:semiflow-local-extension}
\widehat{\varphi}(t,x) = 
\begin{cases} 
\varphi^t(x), & (t,x)\in \dom(\varphi) \cap ([0,\infty) \times U)\\ 
\rho(x), & x\in U,\, t = \mu(x) 
\end{cases}.     
\end{equation}	
	We say that $\rho\colon U \to \gd$ is a \concept{flow-induced retraction}.
	If $\gd$ is a flow-induced retract, we also say that $H$ has a \concept{trapping guard} $\gd$ or that $H$ satisfies the \concept{trapping guard condition}.
\end{Def}

\begin{Rem}
	Note that, in particular, the trapping guard condition rules out the possibility that trajectories of a continuous extension of $\varphi$ ``graze'' $\gd$.
	However, the trapping guard condition also rules out behavior unrelated to grazing, such as the possibility that $\gd$ repel trajectories of $\varphi$ initialized near $\gd$.
\end{Rem}

We justify Def.~\ref{def:trapping-guards} in several ways.
First, Ex.~\ref{ex:omega-chain-pathologies} shows that, in the absence of the trapping guard condition, $\omega$-limit and chain recurrent sets need not satisfy many of the properties which are standard in the setting of classical dynamical systems.
Second, \THS~ satisfying this condition encompass a wide variety of physically-relevant examples, as shown in the next sections. 
Third, Ex. \ref{ex:counterexample} demonstrates that our main theorems can fail without the trapping guard condition even for very simple \MHS, hence some such condition is necessary.  
Finally, in Remark \ref{rem:suspension-semiflow-continuous-iff-trapping-guard} and SM \S \ref{app:suspension-semiflow-continuous-implies-trapping-guard-condition} we point out that Def. \ref{def:trapping-guards} arises naturally from mathematical considerations.

Our second main result involves the notion of a complete Lyapunov function introduced by Conley \cite[p.~39]{conley1978isolated}; here we generalize the definition to \MHS.

\begin{Def}[Hybrid complete Lyapunov function]\label{def:lyapunov}
	Let $H=(\st,\fs,\gd,\varphi,r)$ be an \MHS.  A \concept{(hybrid) complete Lyapunov function} for $H$ is a continuous function $L: \st \to \mathbb{R}$ satisfying the following conditions.
	\begin{enumerate}[label=\ref*{def:lyapunov}.\arabic*.\hspace{0.1cm}, ref=\ref*{def:lyapunov}.\arabic*,leftmargin=1.3cm]
		\item \label{enum:flow-dec} For every $x\in \fs\setminus R(H)$, $\chi \in \exc_H(x)$, $t > 0$, and $y \in \chi(t)$, $L(y) < L(x)$.	
		\item \label{enum:reset-dec} If $x\in \gd\setminus R(H)$, then $L(r(x))< L(x)$.
		
		\item \label{enum:lyap-chain} For all $x,y \in R(H)$:  $x$ and $y$ are chain equivalent if and only if  $L(x) = L(y)$.
		 
		\item \label{enum:nowhere-dense} $L(R(H))$ is nowhere dense in $\R$.
	\end{enumerate}
\end{Def}	

In most of the remainder of this paper, we either make or verify Assumptions \ref{assump:deterministic}, \ref{assump:inf-or-Zeno}, and \ref{assump:trapping-guard} below regarding a \THS~ $H = (\st,\fs,\gd,\varphi,r)$. 
\begin{Assump}\label{assump:deterministic}
	$H$ is deterministic.
\end{Assump}

\begin{Assump}\label{assump:inf-or-Zeno}
	For every $x \in \st$, there is an infinite or Zeno execution starting at $x$.
\end{Assump}

\begin{Assump}\label{assump:trapping-guard}
	$H$ satisfies the trapping guard condition (Def.~\ref{def:trapping-guards}).
\end{Assump}

We also make the following assumption on $H = (\st,\fs,\gd,\varphi,r)$ in Theorems \ref{th:conley-decomp} and \ref{th:hybrid-conley}.

\begin{Assump}\label{assump:compact}
	$\st$ is compact.
\end{Assump}

We now state our main results, but postpone the proofs to \S \ref{sec:proofs-main-theorems}.
\begin{restatable}[Conley's decomposition theorem for \MHS]{Th}{ThmConleyDecomp}\label{th:conley-decomp}
    Let $H=(\st,\fs,\gd,\varphi,r)$ be a metric hybrid system satisfying Assumptions \ref{assump:deterministic}, \ref{assump:inf-or-Zeno}, \ref{assump:trapping-guard}, and \ref{assump:compact}.
    Then the hybrid chain recurrent set $R(H)$ admits a \concept{Conley decomposition}:
	\begin{equation}
	R(H) = \bigcap \{A\cup A^*\mid A \textnormal{ is an attracting set for } H.\}.
	\end{equation} 
	Furthermore, $x,y\in \st$ are chain equivalent if and only if either $x,y\in A$ or $x,y\in A^*$ for every attracting-repelling pair $(A,A^*).$
\end{restatable}

\begin{restatable}[Conley's fundamental theorem for \MHS]{Th}{ThmFund}\label{th:hybrid-conley}	
	Let $H=(\st,\fs,\gd,\varphi,r)$ be a metric hybrid system satisfying Assumptions \ref{assump:deterministic}, \ref{assump:inf-or-Zeno}, \ref{assump:trapping-guard}, and \ref{assump:compact}.
	Then there exists a complete Lyapunov function for $H$.
\end{restatable}	
Simple examples illustrating these theorems---including an example which shows that the conclusions can fail without Assumption~\ref{assump:trapping-guard}---are given in \S\ref{sec:ex}.
\S\ref{sec:app} contains propositions which guarantee that Theorems~\ref{th:conley-decomp} and \ref{th:hybrid-conley} apply to various general classes of hybrid systems appearing in the literature.

\begin{Rem}
	As discussed in Ex. \ref{ex:special-case-discrete-time} and \ref{ex:special-case-continuous-time}, a discrete-time (semi-)dynamical system given by iterating a continuous map is a deterministic \THS~ with $\st = \gd$, and a continuous-time (semi-)dynamical system given by a continuous semiflow is a deterministic \THS~ with $\st = \fs$.
	Viewed as hybrid systems, a continuous-time system has only infinite maximal executions, a discrete-time system has only Zeno maximal executions, and in both cases the trapping guard condition is satisfied vacuously.
	Hence Assumptions \ref{assump:deterministic}, \ref{assump:inf-or-Zeno}, and \ref{assump:trapping-guard} are satisfied in both cases.
	Additionally, Ex.~\ref{ex:special-case-continuous-time} shows that, although Def.~\ref{def:eps-T-chains} does not specialize to the classical definition of $(\epsilon, T)$-chains in the continuous-time setting \cite{conley1978isolated,hurley1995chain}, the Conley relations  defined using both definitions coincide. 
	Hence Theorems \ref{th:conley-decomp} and \ref{th:hybrid-conley} strictly generalize the corresponding theorems of \cite{conley1978isolated,franks1988variation}.
\end{Rem}

\begin{Rem}\label{rem:noncompact-but-restrict}
	Let $H=(\st,\fs,\gd,\varphi,r)$ be an \MHS~ not satisfying the hypotheses of Theorems \ref{th:conley-decomp} and \ref{th:hybrid-conley}. 
	If there exists a forward invariant subset $K\subseteq \st$ such that Assumptions~\ref{assump:deterministic}, \ref{assump:inf-or-Zeno}, \ref{assump:trapping-guard}, and \ref{assump:compact} are satisfied by the ``restricted'' \MHS~ $H_{K} = (\st\cap K, \fs\cap K, \gd \cap K, \varphi_{K}, r|_{\gd\cap K})$ (where $\varphi_K$ is the restriction of $\varphi$ to $\dom(\varphi) \cap ([0,\infty) \times K)$), then Theorems~\ref{th:conley-decomp} and \ref{th:hybrid-conley} can still be be applied to $H_K$.
\end{Rem}

\section{Applications}  \label{sec:app}
This section assumes some basic knowledge of smooth manifold theory \cite{tu2010intro,lee2013smooth} (and, for Cor.~\ref{co:hybrid-lagrangian}, geometric mechanics \cite{marsden1994introduction}), and can safely be skipped by the reader whose background and/or motivation are lacking.
We adopt the definition of ``smooth manifold (with boundary)'' from \cite[p.~48]{tu2010intro} which allows different connected components to have different dimensions.\footnote{
This variant affords substantial economy of expression. 
For example, this convention obviates not only the need to introduce the ``smooth hybrid manifold'' terminology of \cite[Sec.~III.A,~Sec.~A.4]{Burden_Revzen_Sastry_2015,Johnson_Burden_Koditschek_2016}, but also the concomitant technical reasoning required to use it. See also \cite{tu2017math-se} for clarifications.\label{foot:tu}
 }

\subsection{General classes of \MHS~to which Theorems~\ref{th:conley-decomp} and \ref{th:hybrid-conley} apply}

The following propositions provide general classes of hybrid systems to which Theorems \ref{th:conley-decomp} and \ref{th:hybrid-conley} apply. We refer to hybrid systems satisfying the hypotheses of Prop.~\ref{prop:trapp-guard-suff-cond} as \concept{smooth exit-boundary guarded \THS}, and by \concept{smooth exit-boundary guarded \MHS} if a compatible extended metric on $\st$ is also specified.

We emphasize (as explained in Footnote~\ref{foot:tu}) that we are using the definition of ``smooth manifold (with boundary)'' from \cite[p.~48]{tu2010intro} which allows different connected components to have different dimensions.
\begin{Prop}\label{prop:trapp-guard-suff-cond}
  Let $H=(\st,\fs,\gd,\varphi,r)$ be a \THS~ for which $\st$ is a smooth manifold with boundary, $\gd$ is a clopen subset of $\partial \st$, $\fs \coloneqq \st \setminus \gd$, and the local semiflow $\varphi$ is generated by a locally Lipschitz\footnote{This is a well-defined, metric-independent notion which depends only on the smooth structure of $\st$ \cite[Rem.~1]{kvalheim2019existence}.} vector field $X$ on $\st$.

  Assume that $X$ is non-strictly inward pointing at each point of $\partial \st \setminus \gd$, and is non-strictly outward pointing at every point of $\gd$.  
  Further assume that, for each $z\in \gd$, the maximal integral curve $t\mapsto \sigma(t)$ of $X$ satisfying $\sigma(0) = z$ is not defined for any positive values of $t$.
  
  Then $H$ satisfies Assumptions~\ref{assump:deterministic} and \ref{assump:trapping-guard}.
  If $\st$ is compact, then $H$ also satisfies Assumptions \ref{assump:inf-or-Zeno} and (trivially) \ref{assump:compact}.
  If additionally $\st$ is equipped with any compatible extended metric making $H$ an \MHS, then $H$ admits a complete Lyapunov function and $R(H)$ admits a Conley decomposition.
\end{Prop}
\begin{Rem}
	The conditions that (i) $X$ point non-strictly outward on $\gd$ and (ii) ``the maximal integral curve $t\mapsto \sigma(t)$ of $X$ satisfying $\sigma(0) = z$ is not defined for any positive values of $t$'' (for all $z\in \gd$)  are implied by the stronger assumption, appearing in the hybrid systems literature \cite{Burden_Revzen_Sastry_2015,De_Burden_Koditschek_2018,clark2019poincare,clark2020poincare}, that $X$ be strictly outward pointing at every point of $\gd$. 
\end{Rem}

\begin{proof}
    By definition \cite[p.~48]{tu2010intro} we have that $\st = \bigsqcup_{j\in J}M_j$ is a disjoint union of smooth manifolds with each $M_j$ of constant dimension.
	For each $j\in J$ we view $M_j\subseteq \R^{n_j}$ as a properly embedded submanifold (by Whitney's theorem), extend $X_j\coloneqq X|_{M_j}$ arbitrarily to a locally Lipschitz vector field $\widehat{X}_j$ on $\R^{n_j}$ with local flow $\widehat{\Phi}_j\colon \dom(\widehat{\Phi}_j)\subseteq \R\times \R^{n_j} \to \R^{n_j}$, and let $\pi_2^j\colon \R \times \R^{n_j} \to \R^{n_j}$ be the projection onto the second factor.
	Then $$U_j\coloneqq M_j \cap \pi_2^j\left((\widehat{\Phi}_j)^{-1}(\R^{n_j} \setminus M_j)\right) \subseteq M_j$$ is an open neighborhood\footnote{This follows since $\R^{n_j} \setminus M_j$ is open, $\widehat{\Phi}_j$ is continuous, $\dom(\widehat{\varphi}_j)$ is open, $\pi_2^j$ is an open map, and every $z\in \gd\cap M_j$ immediately flows into $\R^{n_j} \setminus M_j$; this last property follows from the assumption that the $X$-integral curve through every $z\in \gd$ is not defined for any positive times.} of $\gd\cap M_j$ in $M_j$ such that (i) for all $x\in U_j$ and $T \geq 0$, $$\left[\widehat{\Phi}_j^{[0,T]}(x) \subseteq \cl(U_j) \subseteq M_j\right] \implies \left[\widehat{\Phi}_j^{[0,T]}(x) \subseteq U_j\right],$$ and (ii) for every $x\in U_j$, there exists $t > 0$ such that $\widehat{\Phi}_j^t(x) \in \gd \cap M_j$.
	Since also $\gd\cap M_j$ is closed in $U_j$ and every point of $\gd\cap M_j$ immediately flows into $\R^{n_j}\setminus M_j$, (i) and (ii)  imply (using the terminology of \cite[p.~24]{conley1978isolated}) that $U_j$ is a \emph{Wazewski set} for $\widehat{\Phi}_j$ with \emph{eventual exit set} $U_j$ and \emph{immediate exit set} $\gd \cap M_j$.
	In this case, the proof of Wazewski's theorem \cite[p.~25]{conley1978isolated} and the definition of the disjoint union topology show that $\gd$ is a trapping guard with $\bigsqcup_{j\in J}U_j$ the domain of a flow-induced retraction, so Assumption \ref{assump:trapping-guard} is verified.
	
	Assumption \ref{assump:deterministic} is verified by definition since $\fs\cap \gd = \varnothing$.
	If $\st$ is also compact, then (since $X$ points non-strictly inward at $\partial \st \setminus \gd$) $X$ integral curves either intersect $\gd$ or exist for all forward time, so Assumption \ref{assump:inf-or-Zeno} is verified.
	Theorems~\ref{th:conley-decomp} and \ref{th:hybrid-conley} imply the final statement of the proposition. 
\end{proof}	

Prop. \ref{prop:trapp-guard-suff-cond} gives one broad class of \MHS~ which satisfy the hypotheses of Theorems \ref{th:conley-decomp} and \ref{th:hybrid-conley}.
In the following Prop.~\ref{prop:sub-level-set}, we present another broad class of \MHS~ to which these theorems also apply.
We then specialize Prop.~\ref{prop:sub-level-set} to a class of \concept{Lagrangian hybrid systems} in Cor.~\ref{co:hybrid-lagrangian}.
Instances of the latter class of systems (or slight variations thereof) have been studied, for example, in \cite{grizzle2001asymptotically, westervelt2003hybrid, ames2006there,poulakakis2009spring,or2010stability,bloch2017quasivelocities,razavi2017symmetry}.
Looking ahead to \S \ref{sec:ex}, this class of systems substantially generalizes the ``gravitational-force bouncing ball'' of Ex. \ref{ex:bouncing-ball} (but does not encompass systems like the ``spring-force bouncing ball'' of Ex. \ref{ex:spring}). 

The statement of Prop.~\ref{prop:sub-level-set} involves \emph{Lie derivatives}.
Given a smooth manifold $M$, a $C^k$ function $\psi\colon M\to \R$, and a complete $C^k$ vector field $X$ on $M$ with flow $\Phi\colon \R \times M \to M$, we let $\mathcal{L}_X \psi\coloneqq \frac{\partial}{\partial t}(\psi\circ \Phi^t)|_{t=0} = d\psi\cdot X$ denote the \concept{Lie derivative} of $\psi$ along $X$; for all $m\in \{2,\ldots, k\}$ we inductively define $\mathcal{L}_X^m \psi \coloneqq \mathcal{L}_X (\mathcal{L}_X^{m-1} \psi)$.

Again using the definition of ``smooth manifold (with boundary)'' from \cite[p.~48]{tu2010intro} (again, see Footnote~\ref{foot:tu}) we state the following result.

\begin{Prop} \label{prop:sub-level-set}
	Let $M$ be a smooth manifold, $k\in \N \cup \{+\infty\}$, $X$ be a complete $C^{k}$ vector field on $M$ with flow $\Phi\colon \R \times M \to M$, and $\psi\colon M\to \R$ be a $C^k$ function.
	We define $\st\coloneqq \psi^{-1}([0,\infty))$, $\gd\coloneqq \psi^{-1}(0) \cap (\mathcal{L}_X \psi)^{-1}((-\infty,0])$, and $\fs\coloneqq \st \setminus \gd$. We assume given a continuous map $r\colon \gd\to \st$.
	Defining the local semiflow $\varphi$ to be the restriction of the flow $\Phi$ to $(\Phi)^{-1}(\fs) \cap ([0,\infty)\times \fs)$ and equipping $M$ with any compatible extended metric
          yields an \MHS~ $H = (\st,\fs,\gd,\varphi,r)$.          
Further assume the following:
\begin{itemize}
\item There exists a compact set $K\subseteq \st$ such that $r(\gd \cap K)\subseteq K$ and, for all $x\in K$ and $T \geq 0$,\footnote{Conley called this condition \emph{positive invariance of} $K$ \emph{relative to} $\st$ \cite[p.~46]{conley1978isolated}.} 
    $$\left[\Phi^{[0,T]}(x)\subseteq \st\right] \implies \left[\Phi^{[0,T]}(x)\subseteq K\right].$$

          \item For all $x\in \gd \cap K \cap (\mathcal{L}_{X}\psi)^{-1}(0)$, there exists an integer $m\in [2,k]$ such that
	\begin{equation}\label{eq:lie-derivatives}
	\mathcal{L}_{X}^1\psi(x) = \cdots = \mathcal{L}_{X}^{m-1}\psi(x) = 0 \qquad \textnormal{and} \qquad \mathcal{L}_{X}^{m}\psi(x) < 0. 
	\end{equation}
\end{itemize}

Then the restricted system $$H_{K} = (\st_{K},\fs_{K},\gd_{K},\varphi_{K},r_{K})= (\st\cap K, \fs\cap K, \gd \cap K, \varphi_{K}, r|_{\gd\cap K})$$ (where $\varphi_K$ is the restriction of $\varphi$ to $\dom(\varphi) \cap ([0,\infty) \times K)$) is well-defined and satisfies Assumptions~\ref{assump:deterministic}, \ref{assump:inf-or-Zeno}, \ref{assump:trapping-guard}, and \ref{assump:compact}.
Hence $H_K$ admits a complete Lyapunov function and  $R(H_K)$ admits a Conley decomposition.
\end{Prop}
\begin{proof}
The first bulleted condition above implies that $H_K$ is a well-defined \MHS.
$H$ satisfies Assumption~\ref{assump:deterministic} since $\fs\cap \gd = \varnothing$, and $H$ satisfies Assumption~\ref{assump:inf-or-Zeno} since $\st = \fs \cup \gd$ and $X$ is complete.
Hence $H_K$ also satisfies Assumptions \ref{assump:deterministic} and \ref{assump:inf-or-Zeno}, and $H_K$ trivially satisfies Assumption~\ref{assump:compact} since $K$ is compact.

We now verify Assumption~\ref{assump:trapping-guard} for $H_K$.
To do this, we prove that $\gd\cap K$ is a trapping guard for $H_K$ by constructing a neighborhood $U$ of $\gd\cap K$ in $K$ satisfying the conditions of Def. \ref{def:trapping-guards}.
We first define the impact time $\mu\colon \st\to [0,+\infty]$ via
\begin{equation}\label{eq:mu-def}
  \mu(x)\coloneqq \sup\{t\geq 0\mid \Phi^{[0,t]}(x)\subseteq \st\} = \inf\{t\geq 0\mid \Phi^t(x) \in M \setminus \st\},
  \end{equation}
  where the second equality holds because the mean value theorem, the definition of $\gd$, and continuity of $\Phi$ imply that (i) $\Phi^{[0,\epsilon]}(z)\not \subseteq \st$ for every $z\in \gd$ and $\epsilon > 0$ and (ii) $\Phi^{\mu(x)}(x)\in \gd$ for every $x\in \st$.  
  
  We now show that $\gd \cap K$ has a neighborhood $U$ in $K$ such that $\mu(x) < \infty$ for every $x \in U$.
  Suppose not. 
  Then, using \eqref{eq:mu-def} and the definition of $\st$, for some $z \in  \gd \cap K$ there exists a sequence $(x_n)_{n\in \N} \subseteq K$ with $x_n \to z$ and $\psi(\Phi^t(x_n)) \ge 0$ for all $t \ge 0$. This and the fact that $\gd\subseteq \psi^{-1}(0)$ imply that $\psi(\Phi^t(z)) = 0$ for all $t \ge 0$.  Thus, all Lie derivatives of $\psi$ at $z$ vanish, contradicting \eqref{eq:lie-derivatives}.

  Next, we show that $\mu|_U$ is continuous.  For any $x\in U$ and $\epsilon > 0$, there exists $T$ in $(\mu(x), \mu(x) + \epsilon)$ with $\Phi^{T}(x) \in M \setminus \st$ by \eqref{eq:mu-def}.  
  Since $M\setminus \st$ is open in $M$, the continuity of $\Phi^T$ yields a neighborhood $V\ni x$ with $\Phi^{T}(V)\subseteq M\setminus \st$. Hence $\mu(V) \subseteq [0, T] \subseteq [0,\mu(x)+\epsilon]$, so $\mu|_U$ is upper semicontinuous.
  To show that $\mu|_U$ is lower semicontinuous, since $(\mu|_U)^{-1}(0) = \gd\cap K$ it suffices to show that, for every $x\in U\setminus \gd$ and every $\epsilon \in (0,\mu(x))$, $x$ has a neighborhood $V$ with $\mu(V) \subseteq [\mu(x) - \epsilon, \infty)$.  
  Fix $x\in U\setminus \gd$.
  By taking $\epsilon$ smaller if necessary, we may assume that $0 < \epsilon < \mu(x)$; pick $T \in (\mu(x) - \epsilon, \mu(x))$.
  By \eqref{eq:mu-def}, we have $\Phi^{[0,T]}(x) \subseteq \st \setminus \gd$. 
  By the continuity of $\Phi$, compactness of $[0,T]$, and openness of $\st\setminus \gd$ in $\st$, there exists a neighborhood $V\ni x$ with $\Phi^{[0,T]}(V) \subseteq \st \setminus \gd$, so  $\mu(V) \subseteq [T,\infty) \subseteq [\mu(x) - \epsilon, \infty)$ as desired.
  
  It is easy to see that the definition \eqref{eq:mu-def} of $\mu$ coincides with that of \eqref{eq:max-flow-time-all}. 
  Since $\mu|_U$ is continuous and $\mu|_{\gd\cap K} \equiv 0$, the flow induced-retraction $\rho\colon U\to \gd\cap K$ defined by $\rho(x)\coloneqq \Phi^{\mu(x)}(x)$ is continuous.
  Hence  $\widehat{\varphi}_K\coloneqq  \Phi|_{\cl(\dom(\varphi_K)) \cap ([0,\infty)\times U)}$ is a continuous extension of $\varphi_K$ to $\cl(\dom(\varphi_K)) \cap ([0,\infty)\times U)$ as required in Def. \ref{def:trapping-guards}, and this extension is unique since $M$ is Hausdorff.
  Thus, $\gd\cap K$ is a trapping guard for $H_K$, so $H_K$ satisfies Assumptions~\ref{assump:deterministic}, \ref{assump:inf-or-Zeno}, \ref{assump:trapping-guard}, and \ref{assump:compact}.
  Theorems~\ref{th:conley-decomp} and \ref{th:hybrid-conley} imply the final statement of the proposition.   
\end{proof}

We now specialize Prop. \ref{prop:sub-level-set} to show that Theorems \ref{th:conley-decomp} and \ref{th:hybrid-conley} can also be applied to a broad class of mechanical systems with unilateral constraints which undergo impacts, and which generalize the bouncing ball system that we will study in Ex.~\ref{ex:bouncing-ball}.
Slightly generalizing\footnote{As opposed (presumably) to \cite[Def.~1]{ames2006there}, we do not require $Q$ to be a manifold of fixed dimension, and we do not require that $h^{-1}(0)$ be a smooth manifold. 
We could have also required less smoothness of $L$ and $h$, but we do not bother with this here; the interested reader can refer to Prop.~\ref{prop:sub-level-set} for more refined smoothness assumptions.} \cite[Def.~1]{ames2006there}, we define a \concept{hybrid Lagrangian} to be a tuple $$\mathbf{L} = (Q,L,h),$$ 
	where:
	\begin{itemize}
		\item $Q$ is a smooth manifold (the \emph{configuration space}, still assuming the terminological convention discussed in Footnote~\ref{foot:tu}) with tangent bundle $\pi\colon \T Q\to Q$,
		\item $L\colon \T Q\to \R$ is a smooth, hyperregular \emph{Lagrangian} \cite[Sec.~7.3--7.4]{marsden1994introduction} so that the associated \emph{Lagrangian vector field} $X_L\colon \T Q\to \T(\T Q)$ is well-defined and smooth (and second order: $\T \pi \circ X_L = \id_{\T Q}$), and 
		\item $h\colon Q\to \R$ is a smooth function.
	\end{itemize}
We assume that the vector field $X_L$ is complete, so that it generates a smooth flow $\Phi\colon \R \times \T Q\to \T Q$. 

\begin{Co} \label{co:hybrid-lagrangian}
  Let $E\colon \T Q\to \R$ be the (total) \emph{energy} associated to a hybrid Lagrangian $\mathbf{L} = (Q, L, h)$; i.e., $E$ is the pullback of the Hamiltonian associated to $L$ via the Legendre transform \cite[p.~183,~p.~186]{marsden1994introduction}.  Define $M\coloneqq \T Q$, $X\coloneqq X_L$, and $\psi\coloneqq h \circ \pi$; $\st$, $\fs$, $\gd$, and $\varphi$ are then specified as in Prop. \ref{prop:sub-level-set}.
Assuming $r\colon \gd \to \st$ is a given continuous map\footnote{In \cite{ames2006there}, the reset map $r$ is given a rather specific definition, which generalizes the reset map of Ex. \ref{ex:bouncing-ball}, but we do not require the use of this specific definition here.}
and endowing $\T Q$ with any compatible extended metric yields an \MHS~ $H = (\st,\fs,\gd,\varphi,r)$ associated to $\mathbf{L}$.
Further assume the following:
    \begin{itemize}
    	\item There exists $E_0 \in \R$ such that some connected component $K := S_{E_0}$ of $\st \cap E^{-1}{((-\infty,E_0])}$ is compact and satisfies $r(\gd \cap S_{E_0})\subseteq S_{E_0}$.
    	\item The set $\gd \cap S_{E_0} \cap (\mathcal{L}_{X_L}(h\circ \pi))^{-1}(0)$ satisfies the condition containing Equation \eqref{eq:lie-derivatives}.
    \end{itemize}

    Then the restricted system $$H_{E_0} = (\st_{E_0},\fs_{E_0},\gd_{E_0},\varphi_{E_0},r_{E_0})\coloneqq (\st\cap S_{E_0}, \fs\cap S_{E_0}, \gd \cap S_{E_0}, \varphi_{E_0}, r|_{\gd\cap S_{E_0}})$$ (where $\varphi_{E_0}$ is the restriction of $\varphi$ to $\dom(\varphi) \cap ([0,\infty) \times S_{E_0})$) is well-defined and satisfies Assumptions~\ref{assump:deterministic}, \ref{assump:inf-or-Zeno}, \ref{assump:trapping-guard}, and \ref{assump:compact}.
    Hence $H_{E_0}$ admits a complete Lyapunov function and $R(H_{E_0})$ admits a Conley decomposition.
\end{Co}
\begin{proof}
    Conservation of energy \cite[Prop.~7.3.1]{marsden1994introduction} implies that $E^{-1}((-\infty,E_0])$ is forward invariant under the flow of $X_L$, and so the first bulleted condition implies that $S_{E_0}$ is forward invariant for $H$; hence we may apply Prop. \ref{prop:sub-level-set} to the restricted system $H_{E_0}$.
\end{proof}

\section{Examples} \label{sec:ex}
We first show in Ex. \ref{ex:counterexample} that, even if an \MHS~ satisfies all hypotheses of Theorems \ref{th:conley-decomp} and \ref{th:hybrid-conley} except the trapping guard condition, the conclusions of both theorems can fail.
We then proceed to give two examples motivated by toy models in classical mechanics. The first (Ex. \ref{ex:bouncing-ball}, a special case of  Cor.~\ref{co:hybrid-lagrangian}) illustrates a system to which Theorems \ref{th:conley-decomp} and \ref{th:hybrid-conley} apply. The second (Ex. \ref{ex:spring}) illustrates a system to which they do not.

\begin{figure}
	\centering
	\def\svgwidth{0.4\columnwidth}
\begingroup%
  \makeatletter%
  \providecommand\color[2][]{%
    \errmessage{(Inkscape) Color is used for the text in Inkscape, but the package 'color.sty' is not loaded}%
    \renewcommand\color[2][]{}%
  }%
  \providecommand\transparent[1]{%
    \errmessage{(Inkscape) Transparency is used (non-zero) for the text in Inkscape, but the package 'transparent.sty' is not loaded}%
    \renewcommand\transparent[1]{}%
  }%
  \providecommand\rotatebox[2]{#2}%
  \newcommand*\fsize{\dimexpr\f@size pt\relax}%
  \newcommand*\lineheight[1]{\fontsize{\fsize}{#1\fsize}\selectfont}%
  \ifx\svgwidth\undefined%
    \setlength{\unitlength}{186.29077861bp}%
    \ifx\svgscale\undefined%
      \relax%
    \else%
      \setlength{\unitlength}{\unitlength * \real{\svgscale}}%
    \fi%
  \else%
    \setlength{\unitlength}{\svgwidth}%
  \fi%
  \global\let\svgwidth\undefined%
  \global\let\svgscale\undefined%
  \makeatother%
  \begin{picture}(1,0.41043888)%
    \lineheight{1}%
    \setlength\tabcolsep{0pt}%
    \put(0,0){\includegraphics[width=\unitlength,page=1]{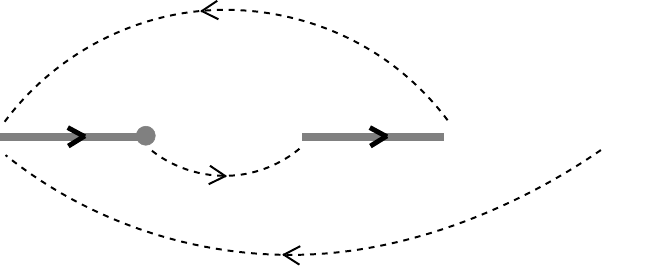}}%
    \put(0.24240326,0.24026569){\color[rgb]{0,0,0}\makebox(0,0)[lt]{\lineheight{1.25}\smash{\begin{tabular}[t]{l}$0$\end{tabular}}}}%
    \put(0.70941514,0.24026569){\color[rgb]{0,0,0}\makebox(0,0)[lt]{\lineheight{1.25}\smash{\begin{tabular}[t]{l}$2$\end{tabular}}}}%
    \put(0.95097288,0.24026569){\color[rgb]{0,0,0}\makebox(0,0)[lt]{\lineheight{1.25}\smash{\begin{tabular}[t]{l}$3$\end{tabular}}}}%
    \put(0,0){\includegraphics[width=\unitlength,page=2]{counterexample-5-V2.pdf}}%
  \end{picture}%
\endgroup%

	\caption{Depicted here is the \MHS~ $H = (\st,\fs,\gd,\varphi,r)$ of Ex. \ref{ex:counterexample}. The chain recurrent set is the union of the two thick gray line segments and gray dots. For this example $\st = [-1,0] \cup [1,3]\subseteq \R$, $\gd = \{0,2,3\}$, and $\fs = \st\setminus \gd$. The local semiflow $\varphi$ is generated by the vector field $-t\frac{\partial}{\partial t}$ on $[-1,0]$ and by the constant vector field $\frac{\partial}{\partial t}$ on $[1,3]$. 
	$r$ is defined by $r(0) = 1$ and $r(2) = r(3) = -1$. The \MHS~ $H$ satisfies all hypotheses of Theorems \ref{th:conley-decomp} and \ref{th:hybrid-conley} except for the trapping guard condition, and $H$ violates the conclusions of both of these theorems.}\label{fig:counterexample}
\end{figure}

\begin{Ex}[Failure in the absence of trapping guards]\label{ex:counterexample}
	This example shows that the conclusions of both Theorems \ref{th:conley-decomp} and \ref{th:hybrid-conley} can fail if Assumption~\ref{assump:trapping-guard} (the trapping guard hypothesis) is violated, even if Assumptions~\ref{assump:deterministic}, \ref{assump:inf-or-Zeno}, and \ref{assump:compact} are satisfied. 
	Define the \MHS~ $H = (\st,\fs,\gd,\varphi, r)$ as follows.
	Let $$\st\coloneqq [-1,0] \cup [1,3]\subseteq \R, \qquad \gd\coloneqq \{0,2,3\}, \qquad  \fs\coloneqq \st\setminus \gd.$$ 
	We let $\varphi$ be generated by the vector field $-t\frac{\partial}{\partial t}$ on $[-1,0]$ and by the constant vector field $\frac{\partial}{\partial t}$ on $[1,3]$. 
	We define $r\colon \gd\to \st$ by $r(0) = 1$ and $r(2) = r(3) = -1$.
	Clearly $H$ is deterministic and nonblocking.
	
    It is easy to check that $R(H) = [-1,0]\cup [1,2]$, that $R(H)$ consists of a single chain equivalence class, and that every subset $J\subseteq \st$ satisfies $\omega(J) = \{0\}$.
    But $\{0\}$ is not forward invariant, so there are no nontrivial attracting-repelling pairs.
    It follows that the conclusions of Theorem \ref{th:conley-decomp} are violated. 

	We further claim that no complete Lyapunov function $L\colon \st\to \R$ for $H$ exists, so that the conclusion of Theorem \ref{th:hybrid-conley} is also violated.  Indeed, suppose such an $L$ exists.  By the continuity of $L$ and Condition~\ref{enum:flow-dec} of Def.~\ref{def:lyapunov}, we have $L(2) > L(3)$.  By Condition~\ref{enum:reset-dec}, we have $L(3) > L(-1)$.  Finally, Condition~\ref{enum:lyap-chain} implies that $L(-1) = L(2)$.  This implies that $L(2) > L(3) > L(-1) = L(2)$, a contradiction.
        
	We emphasize that $H$ satisfies all hypotheses of Theorems \ref{th:conley-decomp} and \ref{th:hybrid-conley} except the trapping guard condition.
\end{Ex}

\begin{figure}
	\centering
	\includegraphics[width=0.49\linewidth]{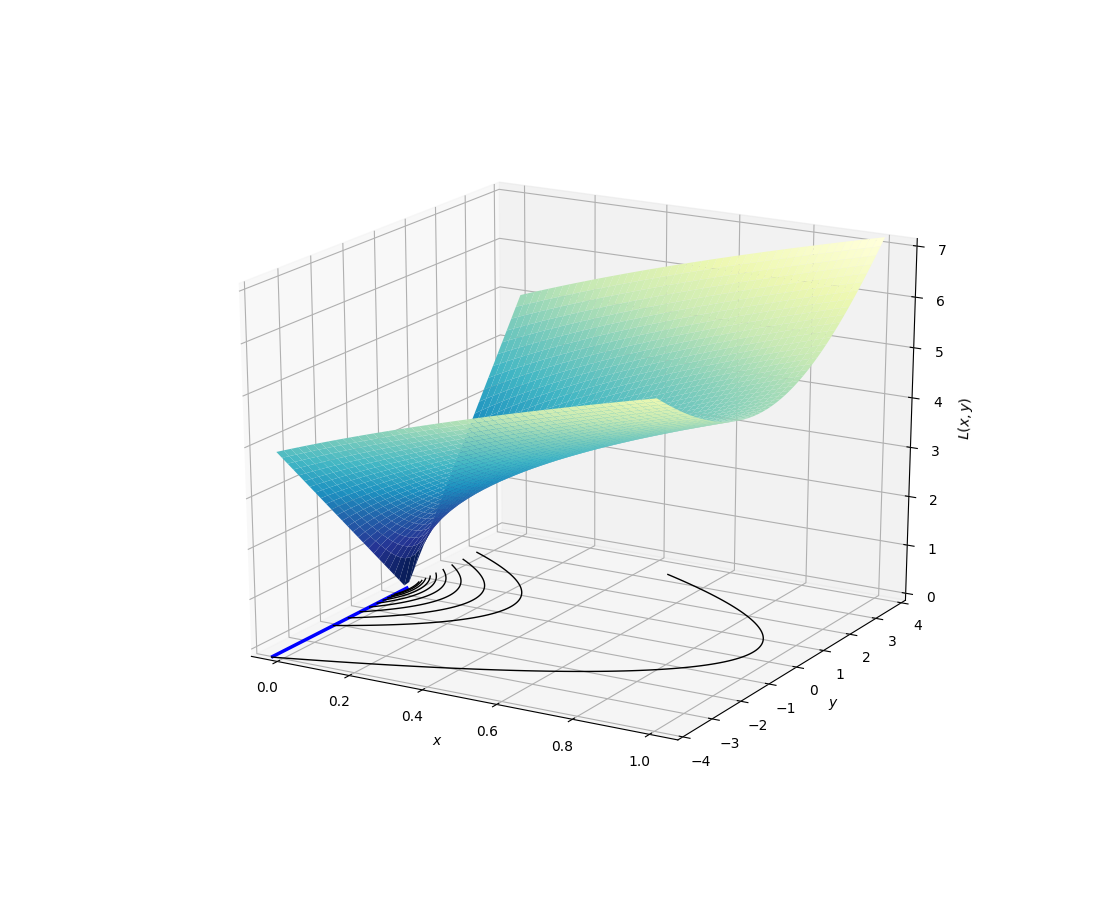}
	\includegraphics[width=0.49\linewidth]{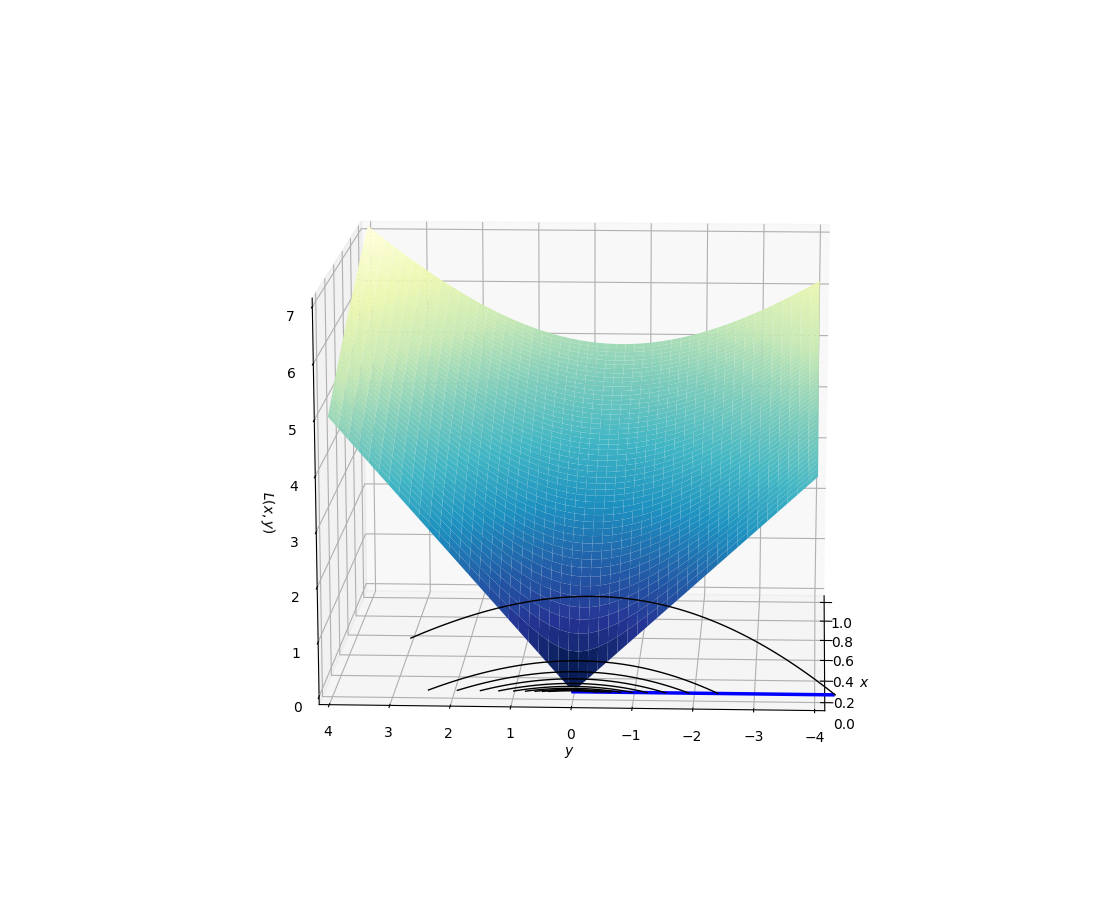}
	\caption{Depicted here are objects associated with the ``bouncing ball'' \MHS~ $H = (\st,\fs,\gd,\varphi,r)$ of Ex. \ref{ex:bouncing-ball} with coefficient of restitution $d = \frac{4}{5}$. 
		$\st$ is the closed half plane $\{x\geq 0\}$, $\fs$ is the interior $\{x > 0\}$ of $\st$, and $\gd$ is the ray $\{x = 0, y \leq 0\}$ shown in blue. 
		A portion of the unique execution in the hybrid system of Ex. \ref{ex:bouncing-ball} with initial condition $(x,y) = (\frac{1}{2},3)$ is shown in black; since $d < 1$, this execution is Zeno.
		Shown above is the graph of the complete Lyapunov function $L$ given by Equation \eqref{eq:example2-lyap-func} with choice of constants $b = \frac{7}{5}$, $a = \frac{999}{1000}\cdot \frac{(1-d)g}{2\sqrt{2} d} b$. }\label{fig:example2}
\end{figure}

\begin{Ex}[Bouncing ball]\label{ex:bouncing-ball}
	Let $g> 0$ and $X$ be the complete analytic vector field on $\R^2$ given by
	$$X(x,y) =  y \frac{\partial}{\partial x}- g \frac{\partial}{\partial y}.$$
	If we think of $x$ as the vertical position of a point particle with unit mass moving under the influence of gravity, and think of $y$ as its velocity, then the dynamics determined by $X$ are equivalent to those determined by Newton's second law of motion, $\ddot{x} = -g$. 
	Letting $\Phi\colon \R \times \R^2 \to \R^2$ be the analytic flow generated by $X$, we now construct an \MHS~ which represents a ``bouncing ball,'' and we show that this \MHS~ satisfies all hypotheses of Theorems \ref{th:conley-decomp} and \ref{th:hybrid-conley} after restriction to a compact forward invariant set.	
	
	Let $\st$ be the closed half plane $\st = \{x\geq 0\}$; we endow $\st$ with the metric induced by the Euclidean distance. 
	Let $\fs$ be the interior of $\st$, and define $\gd\coloneqq \{(0,y)\mid y \leq 0\}$.
	Let $\varphi$ be the continuous semiflow given by the restriction of $\Phi$ to $(\Phi)^{-1}(\fs) \cap (\fs \times [0,\infty))$, and let the reset $r\colon \gd\to \st$ be given by
	\begin{equation*}
	r(0,y)\coloneqq (0,- d y) ,
	\end{equation*}
	where $d\in [0,1]$ is the coefficient of restitution (which is related to the energy lost by the ball at impact). 
	Clearly the \MHS~ $H = (\st,\fs,\gd,\varphi,r)$ is deterministic (Assumption~\ref{assump:deterministic}).

	We now verify that $H$ satisfies the trapping guard condition (Assumption~\ref{assump:trapping-guard}).
	We define the neighborhood $U$ of Def. \ref{def:trapping-guards} to be $U\coloneqq \st$.
	Integrating the vector field $X$ analytically yields $$\Phi^t(x,y) = \left(x + yt - \frac{g}{2}t^2, y - gt\right).$$
	Setting the first component on the right hand side equal to $0$ and solving the resulting quadratic equation, we find that the maximum flow time $\mu\colon \st\to [0,+\infty]$ (Equation \ref{eq:max-flow-time-all}) is given by 
	$$\mu(x,y) = \frac{y+\sqrt{y^2 + 2 x g}}{g},$$
	which is clearly continuous.
	We define the continuous flow-induced retraction $\rho\colon \st \to \gd$ by $$\rho(x,y)\coloneqq \Phi^{\mu(x,y)}(x,y) = (0,-\sqrt{y^2 + 2 x g}).$$
	Hence $\widehat{\varphi}\coloneqq  \Phi|_{\cl(\dom(\varphi)) \cap [0,\infty)\times \st}$ is a continuous extension of $\varphi$ to the closure of $\dom(\varphi)$ in $[0,\infty)\times \st$, as required in Def. \ref{def:trapping-guards}, and this extension is unique since $\R^2$ is Hausdorff.
	
	The above shows that $H$ satisfies the trapping guard condition.
    To investigate Zeno executions, we compute the total time elapsed during the execution initialized at $(x,y)$:
	$$\sum_{n=0}^\infty \mu \circ (r \circ \rho)^{\circ n}(x,y) = \sum_{n=0}^\infty \mu\left(0,d^n\sqrt{y^2+2xg}\right) = \frac{2 \sqrt{y^2+2xg}}{g} \sum_{n=0}^\infty d^n.$$
	Since the last sum is finite if and only if $d\in [0,1)$, it follows that (i) every maximal execution in $H$ initialized in $\st\setminus \{\mathbf{0}\}$ is Zeno if $0\leq d < 1$, and (ii) every maximal execution in $H$ initialized in $\st\setminus \{\mathbf{0}\}$ is infinite if $d = 1$.
	(For any value of $d\in [0,1]$, the execution initialized at $\mathbf{0}$ is Zeno.)
	Hence $H$ satisfies Assumption~\ref{assump:inf-or-Zeno}.
	
	The preceding shows that $H$ satisfies Assumptions~\ref{assump:deterministic}, \ref{assump:inf-or-Zeno}, and \ref{assump:trapping-guard} but not the compactness Assumption~\ref{assump:compact}.
    However, for any finite $E_0 > 0$, the ``energy'' sublevel set $S_{E_0}\coloneqq \{(x,y)\in \st\mid \frac{1}{2}y^2 + gx \leq E_0\}$ is compact and forward invariant for $H$.
    Therefore, the restricted system $$H_{E_0} = (\st_{E_0},\fs_{E_0},\gd_{E_0},\varphi_{E_0},r_{E_0})= (\st\cap S_{E_0}, \fs\cap S_{E_0}, \gd \cap S_{E_0}, \varphi_{E_0}, r|_{\gd\cap S_{E_0}})$$ (where $\varphi_{E_0}$ is the restriction of $\varphi$ to $\dom(\varphi) \cap ([0,\infty) \times S_{E_0})$) satisfies Assumptions \ref{assump:deterministic}, \ref{assump:inf-or-Zeno}, \ref{assump:trapping-guard}, and \ref{assump:compact} and hence all hypotheses of Theorems \ref{th:conley-decomp} and \ref{th:hybrid-conley} (cf. Remark \ref{rem:noncompact-but-restrict}).
    
    When the coefficient of restitution $d = 1$, it is easy to see that the only attracting-repelling pairs for $H_{E_0}$ are the trivial pairs $(A,A^*) = (\st_{E_0},\varnothing)$ and $(A,A^*) = (\varnothing, \st_{E_0})$.
    According to Theorem \ref{th:conley-decomp}, the chain recurrent set $R(H_{E_0})$ is therefore given by $R(H_{E_0}) = \st_{E_0}$.
    According to Theorem \ref{th:hybrid-conley}, $H_{E_0}$ has a complete Lyapunov function; since $R(H_{E_0}) = \st_{E_0}$ and since $\st_{E_0}$ is connected, the complete Lyapunov functions are precisely the constant real-valued functions on $\st_{E_0}$.
    
    When the coefficient of restitution $d \in [0,1)$, it is easy to see that $H_{E_0}$ has a unique nontrivial attracting-repelling pair $(A,A^*)$: the attracting set $A$ is the origin $\mathbf{0}$, and the dual repelling set $A^*$ is the empty set. 
	According to Theorem \ref{th:conley-decomp}, the chain recurrent set is therefore given by $R(H_{E_0}) = \{\mathbf{0}\}$.
	According to Theorem \ref{th:hybrid-conley}, $H_{E_0}$ has a complete Lyapunov function.
	One family of such complete Lyapunov functions are given by a linear combination of the maximum flow time and the square root of the total energy,
	\begin{equation}\label{eq:example2-lyap-func}
	L(x,y) = a \mu(x,y) + b \sqrt{\frac{1}{2}y^2 + gx},
	\end{equation} 
	where $a, b> 0$ are constants satisfying $\frac{1-d}{\sqrt{2}}b > \frac{2 d}{g} a$.
	The first term on the right strictly decreases along the continuous-time dynamics,  while the second term is constant.
	On $\gd\setminus \{\mathbf{0}\}$, the first term on the right strictly decreases upon applying the reset map, while the second term strictly increases; however, the inequality $\frac{1-d}{\sqrt{2}}b > \frac{2 d}{g} a$ ensures that
	\begin{align*}
	L(r(0,y)) -  L(0,y) &=  a\cdot \left(\frac{2 d |y|}{g} - 0\right) + b\cdot \left(\sqrt{\frac{1}{2}d^2y^2} - \sqrt{\frac{1}{2}y^2}\right)\\
	& = \left(a \frac{2d}{g}- b\left(\frac{1-d}{\sqrt{2}} \right)\right)|y| < 0,
	\end{align*}
	as desired.
	Because $L$ does not depend on ${E_0}$, $L$ is also a complete Lyapunov function for the full system $H$.
	
	In closing, we remark that this example is a special case of the hybrid Lagrangian systems considered in Cor.~\ref{co:hybrid-lagrangian} with (in the notation of the corollary) $Q\coloneqq \R$, $L(x,y)\coloneqq \frac{1}{2}y^2 - gx$, $h(x)\coloneqq x$, and $r(0,y)\coloneqq (0,- d y)$.
	Since the restriction of $E(x,y) = \frac{1}{2}y^2 + gx$ to $\st$ is a proper function, $S_{E_0}$ is compact for every $E_0 > 0$.
	In this example, $\gd\cap S_{E_0} \cap (\mathcal{L}_X \psi)^{-1}(0) = \{\mathbf{0}\}$ is just the origin, and \eqref{eq:lie-derivatives} is satisfied since $\mathcal{L}_{X_L}^2(h\circ \pi) = \ddot{x} = -g < 0$ on all of $\st$.
	Thus, Cor.~\ref{co:hybrid-lagrangian} directly implies that  $H_{E_0}$ satisfies all hypotheses of Theorems~\ref{th:conley-decomp} and \ref{th:hybrid-conley}, although the above ``hands-on'' verification seems instructive.
\end{Ex}

\begin{Ex}\label{ex:spring}
  In this example we consider a system which, while superficially similar to the bouncing ball of Ex. \ref{ex:bouncing-ball}, does not satisfy the trapping guard condition (Assumption~\ref{assump:trapping-guard}), so the hypotheses of Theorems \ref{th:conley-decomp} and \ref{th:hybrid-conley} are not satisfied.  Nevertheless, as we show below, this system does satisfy the conclusions of these two theorems.  Thus, Theorems \ref{th:conley-decomp} and \ref{th:hybrid-conley} could potentially be sharpened in future work.
  
	We define the deterministic (Assumption~\ref{assump:deterministic}) \MHS~ $H = (\st,\fs,\gd,\varphi,r)$ to be given exactly as in Ex. \ref{ex:bouncing-ball}, except we replace the vector field $X$ of Ex. \ref{ex:bouncing-ball} with $$X(x,y) = y \frac{\partial}{\partial x} - x \frac{\partial}{\partial y}.$$ 
	Intuitively, $X$ is the vector field obtained from Newton's second law of motion where the gravitational force of Ex. \ref{ex:bouncing-ball} is replaced by a Hookean spring with unit stiffness and rest position $x = 0$ (so $\ddot x = -x$).
	
	A computation as in Ex. \ref{ex:bouncing-ball} verifies Assumption~\ref{assump:inf-or-Zeno} since, in this example, every maximal execution initialized in $\st\setminus \{\mathbf{0}\}$ is infinite, while the execution initialized at $\mathbf{0}$ is Zeno.
	On $\R^2 \setminus \{\mathbf{0}\}$, $X$ is given in polar coordinates $x = \rho \cos\theta, y = \rho \sin \theta$ as $$X(\rho,\theta) = - \frac{\partial}{\partial \theta}.$$
	Hence the maximum flow time $\mu\colon \st\to [0,+\infty]$ (Equation \ref{eq:max-flow-time-all}) is given by
	\begin{equation*}
    \begin{cases} 
	\mu(\rho,\theta) =  \theta + \frac{\pi}{2}, & \theta \in [-\frac{\pi}{2},\frac{\pi}{2}] \textnormal{ and } \rho\neq 0\\ 
	\mu(\rho,\theta) = 0, & \theta \in [-\frac{\pi}{2},\frac{\pi}{2}] \textnormal{ and } \rho= 0 \\
	\end{cases},   
	\end{equation*}  
	which is discontinuous at the origin $\mathbf{0}\in \gd\subseteq \st$, so $H$ does not satisfy the trapping guard condition (Assumption~\ref{assump:trapping-guard}).
	Since also $\mathbf{0}$ belongs to every energy sublevel set $S_{E_0} \coloneqq \{(x,y) \mid \frac{1}{2}\rho^2 \leq E_0\}$ for which $S_{E_0} \cap \st \neq \varnothing$, the restricted system $$H_{E_0} = (\st_{E_0},\fs_{E_0},\gd_{E_0},\varphi_{E_0},r_{E_0})= (\st\cap S_{E_0}, \fs\cap S_{E_0}, \gd \cap S_{E_0}, \varphi_{E_0}, r|_{\gd\cap S_{E_0}})$$ defined as in Ex. \ref{ex:bouncing-ball} does not satisfy the hypotheses of Theorems \ref{th:conley-decomp} and \ref{th:hybrid-conley} for any value of $E_0$ such that $S_{E_0}\cap \st \neq \varnothing$.
	
	However, in this example it is nevertheless easy to check that the conclusions of Theorem \ref{th:conley-decomp} and \ref{th:hybrid-conley} still hold (for $H$, not merely $H_{E_0}$).
	When the coefficient of restitution $d = 1$, $H$ has only the trivial attracting-repelling pairs, $R(H) = \st$, and any constant function on $\st$ is a complete Lyapunov function.
	 
	When $d\in [0,1)$,  the only nontrivial attracting-repelling pair for $H$ is $(\{\mathbf{0}\}, \varnothing)$, $R(H) = \{\mathbf{0}\}$, and one family of complete Lyapunov functions on $\st$ is given in polar coordinates by $$L(\rho,\theta) = a\rho\mu(\rho,\theta)  + b \rho,$$ where $a,b > 0$ satisfy $a< \frac{(1-d)b}{d \pi}$.
	To see that $L$ is indeed a complete Lyapunov function, note that $\rho$ is constant along $\varphi$ trajectories while $\mu$  strictly decreases, and $L$ decreases along resets from $\gd\setminus \{\mathbf{0}\}$ since
	$$L\left(d\rho,\pi/2\right) - L\left(\rho,-\pi/2\right) = \left(ad \pi + b(d-1) \right) \rho < 0.$$
\end{Ex}

\section{Proofs of the main results}\label{sec:proofs}
This section culminates (in \S\ref{sec:proofs-main-theorems}) with the proofs of Theorems~\ref{th:conley-decomp} and \ref{th:hybrid-conley}.
The earlier subsections develop the necessary tools.

Our development of these tools relies heavily on topologically ``gluing'' multiple topological spaces together and to themselves.
The motivation for doing so---inspired by the ``hybrifold'' technique \cite{simic2000towards,simic2005towards}---is to somehow reduce the study of a hybrid dynamical system to that of an associated ``classical'' continuous-time dynamical system for which comparatively well-developed theory already exists.
``Gluing'' is made precise using the concepts of \emph{quotient maps} and \emph{quotient spaces} from elementary point-set topology.
A \concept{quotient map} $\pi\colon X\to Y$ between topological spaces $X$ and $Y$ is a continuous surjective map having the property that $U\subseteq Y$ is open in $Y$ if and only if $\pi^{-1}(U)$ is open in $X$.
Given an equivalence relation $\sim$ on a topological space $X$, we employ the standard notation $X/\sim$ for the topological \concept{quotient space} of $X$ by $\sim$.
Viewed as a set, $X/\sim$ is the set of all equivalence classes, but $X/\sim$ is also given a topology.  
If we denote by $\pi\colon X\to X/\sim$ the map sending points to their equivalence classes, the topology on $X/\sim$ is uniquely defined by requiring that $\pi$ be a quotient map.
One thinks of $X/\sim$ as obtained from $X$ by ``gluing'' points $x,y\in X$ together if $x\sim y$.
We refer readers wanting more details to the standard introductory textbook \cite[p.~137,~p.~139]{munkres2000topology}).

\subsection{Suspension of a hybrid system}\label{sec:hybrid-suspension}

\begin{figure}
	\centering
	\def\svgwidth{1.0\columnwidth}
	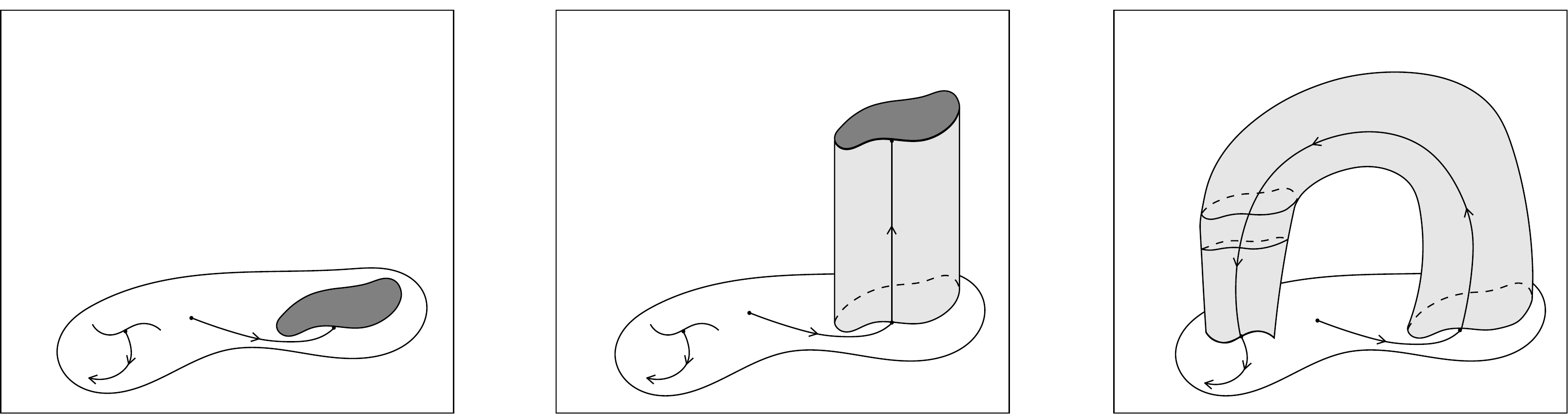
	\caption{Illustration of the constructions in \S \ref{sec:hybrid-suspension}. Left: a \THS~ $H = (\st,\fs,\gd,\varphi,r)$. 
	Middle: its relaxed version $H' = (\st',\fs',\gd',\varphi',r')$; see Def. \ref{def:relaxed}. 
	Right: the hybrid suspension $(\Sus_H, \Phi_H)$ of $H$; see Def. \ref{def:hybrid-suspension}.
	We remark that $(\Sus_H, \Phi_H)$ strictly generalizes the classical suspension of a discrete-time dynamical system \cite[p.~797,~pp.~21--22]{smale1967differentiable,brin2002introduction}; indeed, if $\st = \gd$ our construction reduces to the classical one (see Ex.~\ref{ex:special-case-discrete-time}, and see SM \S\ref{app:hybrid-suspension} and SM \S \ref{app:suspension} for more details).}\label{fig:hybrid-suspension}
\end{figure}

Given a \THS~ $H = (\st,\fs,\gd,\varphi,r)$ satisfying Assumptions~\ref{assump:deterministic}, \ref{assump:inf-or-Zeno}, and \ref{assump:trapping-guard}, in this section we will construct two new systems: (i) a \THS~ $H' = (\st',\fs',\gd',\varphi',r')$ with $\st'$ containing $\st$ as a proper subset which we term the \emph{relaxed} hybrid system, and (ii) a continuous semiflow on a quotient $\Sus_H$ of $H'$ which we term the \emph{hybrid suspension} of $H$.
The basic ideas are contained in Fig.~\ref{fig:hybrid-suspension}.
For the interested reader, SM \S \ref{app:hybrid-suspension} contains a comparison of $H'$ and $\Sus_H$ to related constructions appearing in the literature.

The relaxed system $H'$ formalizes the idea of requiring that $H$ executions ``wait'' one time unit after impacting the guard before resetting.
For us, the relaxed system is primarily a means to an end; it is an intermediate step in constructing the hybrid suspension, and it is also useful in analyzing some of the latter system's properties.
We construct the relaxed system after the following two lemmas.
The statement of Lemma ~\ref{lem:extend} involves the maximum flow time $\mu\colon \st\to [0,+\infty]$ defined in \eqref{eq:max-flow-time-all}; we defer the proof of Lemma~\ref{lem:extend} to SM \S \ref{app:technical}.

\begin{restatable}{Lem}{Extend} \label{lem:extend}
	Let $H = (\st, \fs, \gd, \varphi, r)$ be a \THS~ satisfying Assumptions~\ref{assump:deterministic}, \ref{assump:inf-or-Zeno}, and \ref{assump:trapping-guard}.
	Then $\mu\colon \st \to [0,+\infty]$ is continuous, the closure $\cl(\dom(\varphi))$ of $\dom(\varphi)$ in $[0, \infty) \times \st$ satisfies 
	\begin{equation}\label{eq:lem-ext-closure-expression}
	\cl(\dom(\varphi)) = \{(t,x)\in [0,\infty)\times \cl(\fs) \mid t \leq \mu(x)\},
	\end{equation}
	and $\varphi$ has a unique continuous extension $\widetilde{\varphi}$ defined on $\cl(\dom(\varphi))$ satisfying $\widetilde{\varphi}^{\mu(x)}(x) \in \gd$ for all $x\in \cl(\fs) \cap \mu^{-1}([0,\infty))$.
	
	Furthermore, $\widetilde{\varphi}$ satisfies the following conditions, with $t,s \in [0,\infty)$: (i) $\widetilde{\varphi}^0 = \id_{\cl(\fs)}$, (ii) $(t+s,x)\in \cl(\dom(\varphi)) \iff$ both $(s,x)\in \cl(\dom(\varphi))$ and  $(t,\widetilde{\varphi}^s(x))\in \cl(\dom(\varphi))$, and (iii) for all $(t+s,x)\in \cl(\dom(\varphi))$, $\widetilde{\varphi}^{t+s}(x) = \widetilde{\varphi}^t(\widetilde{\varphi}^s(x))$.
\end{restatable}

\begin{Lem}\label{lem:closed-quotient-metrizable}
	Let $X$ be a metrizable space and $A\subseteq X$ a closed subset.
	Let $f\colon A\to X$ be a continuous and closed map satisfying $f(A) \cap A = \varnothing$.
	Define an equivalence relation $\sim$ on $X$ by identifying each $a\in A$ with $f(a)\in X$.
	Then the quotient map $\pi\colon X\to X/\sim$ is closed.
	If additionally $f$ has compact fibers (i.e., $f^{-1}(x)$ is compact for all $x\in X$), then the quotient space $X/\sim$ is metrizable.
\end{Lem}
\begin{Rem}\label{rem:compact-guard-automatically-closed-proper}
	If $A$ is compact and $f$ is continuous, then $f$ is automatically closed and proper; hence also $f^{-1}(x)$ is compact for all $x\in X$.
\end{Rem}
\begin{proof}    	
    We first show that $\pi$ is closed. 
    Letting $C\subseteq X$ be any closed subset, we compute 
    \begin{equation}\label{eq:quotient-metrizable}
    \pi^{-1}(\pi(C)) = C \cup f^{-1}(C) \cup f(C\cap A)
    \end{equation}
    since $f(A) \cap A = \varnothing$.
    The first set on the right is closed by definition, and the second term is closed since $A$ is closed and $f$ is continuous.
    The third term is closed since $f$ is a closed map. 
    Hence $\pi^{-1}(\pi(C))$ is closed, and this in turn implies that $\pi(C)$ is closed by the definition of the quotient topology.    
    Hence $\pi$ is closed.
    
    We now prove that $X/\sim$ is metrizable under the additional assumption that $f$ has compact fibers.
    A theorem of Stone \cite[Thm~1]{stone1956metrizability} implies that, if $\pi$ is closed, then $X/\sim$ is metrizable if and only if $\pi^{-1}(\pi(x))$ has compact boundary for all $x\in X$. 
    Therefore, it suffices to show that $\pi^{-1}(\pi(x))$ has compact boundary for all $x$.

    Fix $x\in X$ and substitute $C = \{x\}$ in \eqref{eq:quotient-metrizable} to obtain
    $$\pi^{-1}(\pi(x)) = \{x\} \cup f^{-1}(x) \cup f(\{x\}\cap A).$$
    The first and third terms on the right are compact because they are either singletons or empty.
    The second term on the right is compact by our assumption that $f$ has compact fibers.
    Hence $\pi^{-1}(\pi(x))$ is compact, and this in turn implies that $\pi^{-1}(\pi(x))$ has compact boundary. 
    This completes the proof. 
\end{proof}

The following definition uses the maximum flow time $\mu\colon \st\to [0,+\infty]$ defined in \eqref{eq:max-flow-time-all}.
As mentioned in \S\ref{sec:contributions} and \ref{sec:related-work}, a version of the following definition appears in \cite{johansson1999regularization}; see \S \ref{app:relaxed} for more details.
\begin{Def}[Relaxed hybrid system]\label{def:relaxed}
	Let $H=(\st,\fs,\gd,\varphi,r)$ be a \THS~ satisfying Assumptions~\ref{assump:deterministic}, \ref{assump:inf-or-Zeno}, and \ref{assump:trapping-guard}.
	We define the \concept{relaxed hybrid system} $H'=(\st',\fs',\gd',\varphi',r')$ associated to $H$ as follows.
	Let $\asymp$ be the equivalence relation on $Y\coloneqq \st\sqcup (\gd \times [0,1])$ which identifies $\gd$ with $\gd\times \{0\}$, $\pi_0\colon Y \to Y/\asymp$ be the corresponding quotient map, and $\pi_1\colon [0,\infty) \times Y  \to [0,\infty) \times (Y/\asymp)$ be the quotient map $\pi_1 \coloneqq \id_{[0,\infty)} \times \pi_0$.\footnote{That $\pi_1$ is a quotient map (where $[0,\infty)\times Y$ has the product topology) follows since (i) $\pi_0$ is a quotient map, (ii) $[0,\infty)$ is a locally compact Hausdorff space, and (iii) the product of a quotient map and the identity map of a locally compact Hausdorff space is always a quotient map \cite[Lem~4.72]{lee2010introduction}.\label{foot:quotient-product}}
	Let $U\supseteq \gd$ be the domain of a flow-induced retraction as in Def. \ref{def:trapping-guards}, and  let $\widetilde{\varphi}$ be the unique continuous extension of $\varphi$ to the closure $\cl(\dom(\varphi))$ of $\dom(\varphi)$ in $\st$ ensured by Lemma \ref{lem:extend}.
	We define the spaces
	\begin{align*}
	\st' &\coloneqq \pi_0\left(\st \sqcup (\gd\times [0,1])\right) \qquad \qquad 
	\gd'\coloneqq \pi_0(\gd\times \{1\}) \qquad \qquad
	\fs'\coloneqq \pi_0\left(\st \sqcup (\gd\times [0,1))\right)\\
	\dom(\varphi')&\coloneqq \pi_1\left( \Big\{(t,x)\mid 0\leq t < \mu(x) + 1 \Big\} \sqcup \bigcup_{t\in [0,1)} \{t\}\times \Big(\gd\times [0,1-t)\Big) \right),
	\end{align*}
	where $\st'$ has the quotient topology, and maps $r'\colon \gd'\to \st'$, $\varphi'\colon \dom(\varphi')\to \fs'$ via
	\begin{align*}
	r'(\pi_0(z,1))&\coloneqq \pi_0(r(z)), \qquad z\in \gd\\	
	\varphi'(t,\pi_0(x)) &= 
	\begin{cases} 
	  \pi_0(\widetilde{\varphi}^t(x)), & (t,x)\in \cl(\dom(\varphi)) \\
           \pi_0(\widetilde{\varphi}^{\mu(x)}(x), t - \mu(x)), & x\in \cl(\fs), t\in [\mu(x), \mu(x) + 1) \\
\pi_0(z,s+t), & x = (z,s) \in \gd \times [0,1-t) 
	\end{cases}.      
	\end{align*}
\end{Def}

\begin{Prop}\label{prop:relaxed}
Let $H=(\st,\fs,\gd,\varphi,r)$ be a \THS~ satisfying Assumptions~\ref{assump:deterministic}, \ref{assump:inf-or-Zeno}, and \ref{assump:trapping-guard}.
Then the relaxed hybrid system $H'=(\st',\fs',\gd',\varphi',r')$ associated to $H$ is also a \THS~ satisfying Assumptions~\ref{assump:deterministic}, \ref{assump:inf-or-Zeno}, and \ref{assump:trapping-guard}.
Furthermore, $H'$ is nonblocking.
	
If $H$ satisfies Assumption~\ref{assump:compact}, then so does $H'$.
If $\st$ is metrizable and $r$ is a closed map with compact fibers, then $\st'$ is metrizable.
In particular, if $\st$ is metrizable and compact, then so is $\st'$.
\end{Prop}
\begin{Rem}[Relaxation converts Zeno executions to infinite executions]\label{rem:relaxed-sys-deterministic-nonblocking}
The nonblocking statement of Prop.~\ref{prop:relaxed} shows that the original motivation for relaxation---elimination of Zeno behavior  \cite{johansson1999regularization}---still holds in our setting.
\end{Rem}

\begin{proof}
By the definition of the disjoint union and quotient topologies, $\gd'$ is closed in $\st'$, $\fs'$ is open in $\st'$, and $\dom(\varphi')$ is open in $[0,\infty)\times \st'$.
Since $r'$ can be written as a composition of continuous functions, $r'$ is continuous.
Lemma~\ref{lem:extend} implies that the three functions defining $\varphi'$ agree on the intersections of their domains; since each such domain is closed in $\dom(\varphi')$, the pasting lemma of point-set topology \cite[Thm~18.3]{munkres2000topology} implies that $\varphi'$ is continuous.
From the properties of $\widetilde{\varphi}$ established in Lemma~\ref{lem:extend}, it is clear that $\varphi'$ is also a local semiflow.
Hence $H'$ is a \THS, and $H'$ is deterministic (Assumption~\ref{assump:deterministic}) since $\gd' \cap \fs' = \varnothing$.
Furthermore, $\gd'$ is clearly a trapping guard (Assumption~\ref{assump:trapping-guard}).
Since $H$ satisfies Assumption~\ref{assump:inf-or-Zeno} and since every maximal $H'$ execution contains every arc of some maximal $H$ execution, $H'$ also satisfies Assumption~\ref{assump:inf-or-Zeno}.
However, $H'$ has no Zeno executions because all but possibly one resets of a maximal $H'$ execution are preceded by an arc defined on an interval of length greater than one.

If $\st$ is compact (Assumption~\ref{assump:compact}), then so is $\gd$ and hence also $\gd\times [0,1]$, so $\st \sqcup (\gd\times [0,1])$ and (by continuity of $\pi_0$) $\st'$ are therefore compact.		
If $r$ is a closed map with compact fibers, then the same is true of $r'$, and $r'(\gd') \cap \gd' = \varnothing$ by construction; thus if $\st$ is also metrizable, Lemma \ref{lem:closed-quotient-metrizable} implies that $\st'$ is metrizable.
By Remark \ref{rem:compact-guard-automatically-closed-proper}, these conditions on $r$ are automatically satisfied if $\st$ is compact.
This completes the proof.
\end{proof}

For later use we record the following  result which implies \cite[Ex.~2.29]{lee2010introduction} that the ``obvious embedding'' $\iota\colon \st\to \st'$ defined by $\iota\coloneqq \pi_0|_\st$ is indeed a topological embedding (a homeomorphism onto its image). 
\begin{Lem}\label{lem:pi_0-closed}
The quotient map $\pi_0\colon \st\sqcup (\gd\times [0,1])\to \st'$ is closed.
\end{Lem}
\begin{proof}
Let $j\colon \gd\to \gd\times \{0\}$ be the obvious identification and
$C\subseteq \st\sqcup (\gd\times [0,1])$ be closed.
Then, since $\gd$ is closed in $\st$ and $j$ is a homeomorphism, $$\pi_0^{-1}(\pi_0(C)) = C \cup j^{-1}(C) \cup j(C\cap \gd)$$
is closed by the definition of the disjoint union topology.
By the definition of the quotient topology, $\pi_0(C)$ is therefore closed.
\end{proof}

Using the relaxed hybrid system associated to $H=(\st,\fs,\gd,\varphi,r)$, we now define the \emph{hybrid suspension} $(\Sus_H,\Phi_H)$ of $H$.  
As mentioned in \S\ref{sec:contributions} and \ref{sec:related-work}, versions of the hybrid suspension have appeared in \cite{ames2005homology,burden2015metrization} under different names; see \S \ref{app:hybrid-suspension-subsubsec} for more details.
We choose to call this space a ``suspension'' because it generalizes the classical suspension of a discrete-time dynamical system \cite[p.~797,~pp.~21--22]{smale1967differentiable,brin2002introduction}; indeed, if $\st = \gd$ our construction reduces to the classical one (see SM \S \ref{app:suspension} for details).
We will later see that the dynamics of $(\Sus_H,\Phi_H)$ are particularly compatible with those of $H$ (Prop.~\ref{prop:omega-limit-set-suspension}, \ref{prop:attracting-repelling-pairs-suspension}, \ref{prop:hybrid-chain-equiv-vs-suspension-chain-equiv}).
(Contrast this with the (generalized) \emph{hybrifold} semiflow appearing in the literature and discussed in detail in SM \S \ref{app:hybrifold}; see also Remarks~\ref{rem:omega-compatibility-for-hybrifold-fails} and \ref{rem:chain-compatibility-for-hybrifold-fails}.)

\begin{Def}[Hybrid suspension]\label{def:hybrid-suspension}
	Let $H=(\st,\fs,\gd,\varphi,r)$ be a \THS~ satisfying Assumptions~\ref{assump:deterministic}, \ref{assump:inf-or-Zeno}, and \ref{assump:trapping-guard}.
	Let $H' = (\st',\fs',\gd',\varphi',r')$ be the relaxed hybrid system of Def. \ref{def:relaxed}.
    Let $\sim$ be the equivalence relation on $\st'$ which identifies each $z'\in \gd'\subseteq \st'$ with $r'(z')\in \st'$.
    We define the topological space $\Sus_H\coloneqq \st'/\sim$ and let $\pi\colon \st'\to \Sus_H$ be the quotient map.
    Since $H'$ is nonblocking, we may define $\Phi_H\colon [0,\infty)\times \Sus_H\to \Sus_H$ by declaring $\Phi_H^t(\pi(x'))$ to be the unique element of the singleton $\pi(\chi_{x'}(t))$ for $x'\in \st'$ and $t\geq 0$.
    We define $(\Sus_H,\Phi_H)$ to be the \concept{hybrid suspension} and \concept{suspension semiflow} of $H$.
    For brevity, we sometimes simply refer to the pair $(\Sus_H,\Phi_H)$ as the hybrid suspension.	
\end{Def}
The following proposition justifies the \emph{suspension semiflow} nomenclature.
\begin{Prop}\label{prop:hybrid-suspension}
Let $H=(\st,\fs,\gd,\varphi,r)$ be a \THS~ satisfying Assumptions~\ref{assump:deterministic}, \ref{assump:inf-or-Zeno}, and \ref{assump:trapping-guard}.
Let $(\Sus_H,\Phi_H)$ be the hybrid suspension.
Then $\Phi_H\colon [0,\infty)\times \Sus_H \to \Sus_H$ is a continuous semiflow.

If $H$ satisfies Assumption~\ref{assump:compact}, then $\Sus_H$ is compact.
If $\st$ is metrizable and $r$ is a closed map with compact fibers, then $\Sus_H$ is metrizable.
In particular, if $\st$ is metrizable and compact, then so is $\Sus_H$.
\end{Prop}
\begin{proof}	
Letting $\chi_x$ denote the unique $H'$ execution through $x\in I'$ and using the notation $\chi_x(t)$ of Def.~\ref{def:execution}, we may view $(t,x)\mapsto \chi_x(t)$  as a set-valued map (SVM) $\chi$.
Since Prop.~\ref{prop:relaxed} implies that $H'$ is nonblocking, the SVM $\chi\colon [0,\infty)\times \st'\to 2^{\st'}$ has only nonempty values.
Prop.~\ref{prop:relaxed} also implies that $H'=(\st',\fs',\gd',\varphi',r')$ satisfies Assumptions~\ref{assump:deterministic}, \ref{assump:inf-or-Zeno}, and \ref{assump:trapping-guard}, so Lemma~\ref{lem:extend} implies that $\varphi'$ has a unique continuous extension $\widetilde{\varphi}'$ defined on $\cl(\dom(\varphi'))$.
This and continuity of $r'$ in turn imply that $\chi$ is upper semicontinuous (USC) \cite[p.~40,~Def.~1]{aubin1984differential}.
Since (i) $\pi$ is continuous, (ii) the composition of USC SVMs is USC \cite[p.~41,~Prop.~1]{aubin1984differential}, and (iii) singleton-valued USC SVMs are equivalent to continuous (single-valued) maps, $\pi\circ \chi$ may be identified with a continuous map $[0,\infty)\times \st' \to \Sus_H$.
Since $\id_{[0,\infty)}\times \pi$ is a quotient map (cf. Footnote~\ref{foot:quotient-product}) and $\Phi_H\circ (\id_{[0,\infty)}\times \pi) = \pi \circ \chi $, the universal property of the quotient topology \cite[Thm~3.70]{lee2010introduction} implies that $\Phi_H$ is continuous. 
Since $\varphi'$ is a local semiflow, it is clear from the definition of $\chi$ that $\chi_x(t+s) = \chi_{y}(t)$ for any $y\in \chi_x(s)$, so $\Phi_H^{t+s}(\pi(x)) = \pi(\chi_y(t)) = \Phi_H^t(\pi(y)) = \Phi_H^t(\Phi_H^s(\pi(x)))$ and hence $\Phi_H$ is a semiflow. 

If $\st$ is compact ($H$ satisfies Assumption~\ref{assump:compact}), then so is $\st'$ by Prop.~\ref{prop:relaxed}, and therefore $\Sus_H = \pi(\st')$ is compact. 
if $r$ is a closed map with compact fibers, then the same is true of $r'$, and since also $r'(\gd')\cap \gd' = \varnothing$, two applications of Lemma~\ref{lem:closed-quotient-metrizable} imply that $\st'$ and $\Sus_H$ are metrizable.
By Remark \ref{rem:compact-guard-automatically-closed-proper}, these conditions on $r$ are automatically satisfied if the guard if $\st$ is compact.
This completes the proof.
\end{proof}

\begin{Rem}[Further motivation for the trapping guard condition.]\label{rem:suspension-semiflow-continuous-iff-trapping-guard}
Let $H = (\st,\fs,\gd,\varphi,r)$ be a \THS~  satisfying Assumptions~\ref{assump:deterministic} and \ref{assump:inf-or-Zeno}. 
Under the additional assumption that $H$ satisfies the trapping guard condition (Assumption~\ref{assump:trapping-guard}), in Def.~\ref{def:relaxed} and \ref{def:hybrid-suspension} we defined the spaces $\st'$, $\Sus_H$ and maps $\pi_0$, $\iota$, and $\pi$.
We also constructed the suspension semiflow $\Phi_H\colon [0,\infty)\times \Sus_H \to \Sus_H$ and showed that $\Phi_H$ is continuous.
It is immediate from the definitions that $\Phi_H$ satisfies the following two properties.
\begin{enumerate}[label=\arabic*.\hspace{0.1cm}, ref=\arabic*,leftmargin=1.3cm]
	\item\label{enum:rem-Phi_H-cond-1} $\Phi_H^t(\pi\circ \pi_0(z,s)) = \pi\circ\pi_0(z,t+s)$ for all $z\in \gd$, $s\in [0,1]$, and $t\in [0,1-s]$.
	\item\label{enum:rem-Phi_H-cond-2} For all $(t,x)\in \dom(\varphi)$, $\pi\circ\iota(\varphi^t(x)) = \Phi_H^t(\pi\circ\iota(x))$.
\end{enumerate}

While for convenience of exposition we only defined the quantities $\st'$, $\Sus_H$, $\pi_0$, $\iota$, and $\pi$ under Assumptions~\ref{assump:deterministic}, \ref{assump:inf-or-Zeno}, and \ref{assump:trapping-guard} (in particular, under Assumption~\ref{assump:trapping-guard}), their definitions make sense (verbatim) for \emph{any} \THS.
Thus, for an arbitrary \THS~ $H$, it makes sense to ask the following question: under what circumstances does there exist a well-defined ``suspension semiflow'' $\Phi$ on $\Sus_H$ for $H$ in the sense that $\Phi$ satisfies Conditions \ref{enum:rem-Phi_H-cond-1} and \ref{enum:rem-Phi_H-cond-2} (stated above for $\Phi_H$)?

In SM \S \ref{app:suspension-semiflow-continuous-implies-trapping-guard-condition} we prove (Cor.~\ref{co:suspension-semiflow-continuous-implies-trapping-guard-condition}) that, if $H = (\st,\fs,\gd,\varphi,r)$ is a \THS~ satisfying Assumptions~\ref{assump:deterministic} and \ref{assump:inf-or-Zeno} with Hausdorff state space $\st$, then there exists a continuous suspension semiflow $\Phi\colon [0,\infty)\times \Sus_H\to \Sus_H$ in the above sense \emph{if and only if} $H$ satisfies the trapping guard condition (Assumption~\ref{assump:trapping-guard}).
This further motivates the trapping guard condition.
\end{Rem}

Prop.~\ref{prop:omega-limit-set-suspension} below relates hybrid $\omega$-limit sets (Def. \ref{def:omega-limit}) to those of the hybrid suspension semiflow, and can be viewed as motivation for the definition of hybrid $\omega$-limit sets. 
\begin{Prop}\label{prop:omega-limit-set-suspension}
	Let $H=(\st,\fs,\gd,\varphi,r)$ be a \THS~ satisfying Assumptions~\ref{assump:deterministic}, \ref{assump:inf-or-Zeno}, and \ref{assump:trapping-guard} and with $r\colon \gd\to \st$ a closed map.  
	Let $H' = (\st',\fs',\gd',\varphi',r')$ be the relaxed hybrid system of Def. \ref{def:relaxed}, $\pi_0\colon \st \sqcup(\gd\times [0,1])\to \st'$ be the quotient of Def.~\ref{def:relaxed}, $\iota\colon \st\to \st'$ be the embedding $\pi_0|_\st$, and $(\Sus_H,\Phi_H)$ be the hybrid suspension of Def. \ref{def:hybrid-suspension} with quotient $\pi\colon \st'\to \Sus_H$.
	Let $\rho_0\colon \pi_0(\gd\times [0,1]) \to \gd$ be the composition of the straight-line retraction $\pi_0(\gd\times [0,1])\to \iota(\gd)$ with the identification $\iota(\gd) \approx \gd$.  
    For all $B\subseteq \Sus_H$, define $$\gd_B\coloneqq \rho_0\left(\pi^{-1}(B) \cap \pi_0(\gd\times [0,1])\right).$$	
	
	Then for all $B\subseteq \Sus_H$:
	\begin{equation}\label{eq:prop-omega-1}
	 (\pi\circ\iota)^{-1}\left(\omega(B)\right) = \omega((\pi\circ\iota)^{-1}(B)) \cup \omega(\gd_B).
	\end{equation}
	In particular, for all $A\subseteq \st$:
	\begin{equation}\label{eq:prop-omega-2}
	\omega(A) = (\pi\circ\iota)^{-1}(\omega(\pi\circ\iota(A))).
	\end{equation}
\end{Prop}
\begin{Rem}\label{rem:compact-guard-hausdorff-automatically-closed-map}
	If $\gd$ is compact and $\st$ is Hausdorff, then $r$ is automatically a closed map.
	In particular, the proposition holds if $H$ satisfies the hypotheses of Theorems~\ref{th:conley-decomp} and \ref{th:hybrid-conley} (i.e., if $H$ is an \MHS~ satisfying Assumptions~\ref{assump:deterministic}, \ref{assump:inf-or-Zeno}, \ref{assump:trapping-guard}, and \ref{assump:compact}.)
\end{Rem}

\begin{Rem}\label{rem:omega-compatibility-for-hybrifold-fails}
	 The statement analogous to Prop.~\ref{prop:omega-limit-set-suspension} obtained by replacing the hybrid suspension with the semiflow on the (generalized) hybrifold $M_H$ (defined in SM \S \ref{app:hybrifold}) is false.
     For example, consider a \THS~ $H$ with $\st = \gd = \{0,1\}$, $r(\gd) = \{0\}$, and $\fs = \varnothing$ (cf. Ex.~\ref{ex:special-case-discrete-time}).
     Then $M_{H}$ is a singleton $\{*\}$. 
     Letting $\pi_{M_H}\colon \st\to M_H$ be the quotient defined in SM \S \ref{app:hybrifold} and $j\in\{0,1\}$, we compute $$\omega(j)=\{0\}\neq \{0,1\} = \pi_{M_H}^{-1}(\{*\}) = \pi_{M_H}^{-1}(\omega(\pi_{M_H}(j))),$$
	 so the analogue of \eqref{eq:prop-omega-2} is false for the (generalized) hybrifold semiflow.
	 The same example shows that the statement analogous to Prop.~\ref{prop:attracting-repelling-pairs-suspension} below for the (generalized) hybrifold semiflow is also false.
\end{Rem}

\begin{proof}[Proof of Prop.~\ref{prop:omega-limit-set-suspension}]	
	For purposes of readability, for this proof we define $f\coloneqq \pi\circ\iota $ and $g\coloneqq (\pi\circ \pi_0)|_{\gd\times [0,1]}$, and we introduce the following notation for $T> 0$ and subsets $A\subseteq \st$, $B\subseteq \Sus_H$:
	\begin{equation*}
	R_{A,T}\coloneqq \bigcup_{x\in A}\,\bigcup_{(N, \tau,\gamma) \in \exc_H(x)} \,  \{\gamma_j(t)\mid j + t \geq T\}, \qquad S_{B,T} \coloneqq \Phi_H^{[T,\infty)}(B).
	\end{equation*}
	In the definition of $R_{A,T}$, note that each $\exc_H(x)$ contains only a single maximal execution since we assume that $H$ is deterministic.
	For later use we note that $f$ is a closed map since it is the composition of  (i) $\pi$, which is a closed map by Lemma~\ref{lem:closed-quotient-metrizable}, and (ii) $\iota$, which is a closed map since, for any closed $C\subseteq \st$, $\pi_0^{-1}(\iota(C)) = C \sqcup ((C\cap \gd)\times \{0\})$ is closed in $(\st\sqcup \gd\times [0,1])$.
	Furthermore, $g$ is a closed map by composition and restriction since (i) $\pi$ is closed, (ii) $\pi_0$ is closed by Lemma~\ref{lem:pi_0-closed}, and (iii) $\gd\times [0,1]$ is closed in $\st\sqcup (\gd\times [0,1])$. 
	
	In order to prove the proposition, we first need to establish two facts: for any $B\subseteq  \Sus_H$, 
	\begin{equation}\label{eq:prop:omega_sus_decomp}
	\omega(B) = \left(\omega(B \cap f(\st)) \right) \cup \omega(f(\gd_B)),
	\end{equation}
	and
	\begin{equation}\label{eq:prop-closure-intersect}
	f(\st) \cap \bigcap_{T>0}\cl\left(S_{B,T}\right) = 	
	\bigcap_{T>0}\cl\left(f(\st) \cap S_{B,T}\right).
	\end{equation}

	We begin by establishing \eqref{eq:prop:omega_sus_decomp}.
	First note that, for all $T \geq 0$, the definitions of $\gd_B$ and $\Phi_H$  immediately imply that
	\begin{equation}\label{eq:prop:omega-sandwich-1}
	\Phi_H^{[1+T,\infty)}(f(\gd_B))\subseteq \Phi_H^{[T,\infty)}(B\setminus f(\st)) \subseteq \Phi_H^{[T,\infty)}(f(\gd_B)).
	\end{equation}
	Equation \eqref{eq:prop:omega-sandwich-1} and the definition of $\omega$-limit sets imply that $\omega(f(\gd_B)) = \omega(B\setminus f(\st))$.
	Since the $\omega$-limit set of a finite union is equal to the union of the $\omega$-limit sets\footnote{Proof: since the closure of a finite union is the union of the closures, we compute $\omega(U\cup V) = \bigcap_{T>0}\cl(\Phi_{H}^{[T,\infty)}(U \cup V)) = \bigcap_{T>0}\cl(\Phi_{H}^{[T,\infty)}(U)) \cup \cl(\Phi_{H}^{[T,\infty)}(V)) = \omega(U) \cup \omega(V)$ for any sets $U,V$. 
    (This result is stated in \cite[II.4.1.C]{conley1978isolated} for the special case of a \emph{flow}.) Repeating the same proof mutatis mutandis shows that \emph{hybrid} $\omega$-limit sets also possess this \concept{finite-union property}.\label{foot:omega-union-property}}, $\omega(B) = \omega(B\cap f(\st)) \cup  \omega(B\setminus f(\st))$, from which \eqref{eq:prop:omega_sus_decomp} now follows.
	
	We now establish \eqref{eq:prop-closure-intersect} for fixed $B\subseteq \Sus_H$; for readability, we define $C_B$ to be the set on the left side of \eqref{eq:prop-closure-intersect} and $D_B$ to be the set on the right. 
	Since $f$ is a closed map, $f(\st)$ is closed in $\Sus_H$, so it is immediate from general topology that $D_B\subseteq C_B$.
	Hence we need only prove that $C_B\subseteq D_B$. 
	To obtain a contradiction, suppose that this is not the case, so that there exists $x\in C_B \setminus D_B$.   
	Clearly we must have $x\not \in \interior(f(\st))$, so $x$ belongs to the boundary $g(\gd\times \{0,1\})$ of $f(\st)$.  
	Since $x \not\in D_B$, there exists $T_0 > 0$ and a neighborhood $V_x$ of $x$ such that $V_x \cap f(\st) \cap S_{B,T_0} = \varnothing$.   
	Since the family $(S_{B,T})_{T>0}$ decreases in $T$,  this implies that
	\begin{equation}\label{eq:Vx-intersect-cond}
	V_x \cap f(\st) \cap S_{B, T} = \varnothing \textnormal{ for all } T \ge T_0.
	\end{equation}
	
    Let $Y\coloneqq \gd\times ([0,1/4) \cup (3/4,1])$ and $\tilde{\nu}\colon Y \to \gd \times \{0,1\}$ be the straight-line retraction.   
    Since $g$ is a closed map, $g$ is a quotient map onto its image.
    Furthermore, since $Y$ is an open $g$-saturated\footnote{Proof: $((r(\gd)\cap \gd)\times \{0\}) \cup (r^{-1}(\gd)\times \{1\}) \subseteq \gd\times \{0,1\} \subseteq Y$, so $g^{-1}(g(Y)) = Y \cup ((r(\gd)\cap \gd)\times \{0\}) \cup (r^{-1}(\gd)\times \{1\})= Y$.} subset of $\gd\times [0,1]$, it follows that (i) $g(Y)$ is open relative to $g(\gd\times[0,1])$ and (ii) $g|_Y\colon Y\to g(Y)$ is also a quotient map \cite[Prop.~3.62.d]{lee2010introduction}. 
    We also note that, if $g(z,t) = g(z',t')$ for some distinct points $(z,t), (z',t') \in Y$, then necessarily $(z,t),(z',t')\in \gd\times \{0,1\}$ and hence $(g \circ \tilde{\nu})(z,t) = g(z, t) = g(z',t') = (g \circ \tilde{\nu})(z', t')$ since $\tilde{\nu}|_{\gd\times \{0,1\}}=\id_{\gd\times\{0,1\}}$.
    Thus, by the universal property of the quotient topology \cite[Thm~3.70]{lee2010introduction}, the map $g \circ \tilde{\nu}$ descends to a continuous trajectory-preserving retraction $\nu\colon g(U) \to g(\gd\times \{0,1\})$.
    We now define a new set $U_x\coloneqq V_x \cap \nu^{-1}(V_x \cap g(\gd\times \{0,1\}))$. 
    Since $U_x$ is open relative to $g(Y)$ which is in turn open relative to $g(\gd\times [0,1])$, $U_x$ is also open relative to $g(\gd\times [0,1])$ and hence the set $U_x \cup f(\st)$ is a neighborhood of $x$.  
    Thus, since $x \in C_B$, for every $T \ge T_0$ there exists $y_{T} \in S_{B,T} \cap (U_x \cup f(\st)) \cap V_x =  S_{B,T} \cap U_x$.  Thus, by the definitions of the suspension semiflow, $\nu$, and $U_x$, we have $\nu(y_{T}) \in V_x \cap f(\st) \cap S_{B,T - \frac{1}{4}}$ for any $T \ge T_0 + \frac{1}{4}$, contradicting \eqref{eq:Vx-intersect-cond}.
    This establishes \eqref{eq:prop-closure-intersect}.

	Armed with \eqref{eq:prop:omega_sus_decomp} and \eqref{eq:prop-closure-intersect}, we now proceed to prove the proposition.
    Since $B\cap f(\st)$ and $f(\gd_B)$ are subsets of $f(\st)$, it is immediate from the definitions that, for all $T\geq 0$,
    \begin{align}\label{eq:prop:omega-sandwich-2}
    f(\st) \cap S_{B\cap f(\st),T} = f(R_{f^{-1}(B),T}), \qquad 
   f(\st) \cap S_{f(\gd_B),T} =  f(R_{\gd_B,T}).
    \end{align}
 
    The following computation, to be justified after, proves \eqref{eq:prop-omega-1}.
    \begin{align*}
    f^{-1}(\omega(B)) &= f^{-1}\left(\omega(B\cap f(\st)) \cup \omega(f(\gd_B))\right)\\
    &= f^{-1}\left(\bigcap_{T>0} \cl\left(S_{B\cap f(\st),T}\right)   \cup \bigcap_{T>0} \cl\left(S_{f(\gd_B),T}\right)\right) \\
    &= f^{-1}\left(\bigcap_{T>0} f(\st) \cap \cl\left(S_{B\cap f(\st),T}\right)   \cup \bigcap_{T>0} f(\st) \cap \cl\left(S_{f(\gd_B),T}\right)\right)\\
    &= f^{-1}\left(\bigcap_{T>0} \cl\left(f(\st) \cap S_{B\cap f(\st),T}\right)   \cup \bigcap_{T>0} \cl\left( f(\st) \cap S_{f(\gd_B),T}\right)\right)\\
    &= \bigcap_{T>0} f^{-1}\left(\cl\left(f(R_{f^{-1}(B),T})\right)\right)   \cup \bigcap_{T>0} f^{-1}\left(\cl\left(f(R_{\gd_B,T})\right)\right)\\
    &= \bigcap_{T>0} \cl\left(R_{f^{-1}(B),T}\right)   \cup \bigcap_{T>0} \cl\left(R_{\gd_B,T}\right)\\
    &\eqqcolon \omega(f^{-1}(B)) \cup \omega(\gd_B),
    \end{align*}
    The first equality follows from taking the preimage of both sides of \eqref{eq:prop:omega_sus_decomp}.
    The second equality follows from the definition of $\omega$-limit set.
    The third equality follows since intersecting a set with $f(\st)$ does not change its $f$-preimage.
    The fourth equality follows from \eqref{eq:prop-closure-intersect}. 
    The fifth equality follows from \eqref{eq:prop:omega-sandwich-2} and the distributivity of preimages over intersections and unions.
    The sixth equality is justified by the facts that (i) $f$ is a continuous and closed map, so taking $f$-images commutes with taking closures \cite[Prop.~2.30]{lee2010introduction}, and (ii) $f$ is injective, so $f^{-1}(f(X)) = X$ for any $X \subseteq \Sus_H$.

	
	To complete the proof, it remains only to verify \eqref{eq:prop-omega-2}.
	Fix $A\subseteq \st$.
	We have $\gd_{f(A)} \subseteq A \cup r^{-1}(A)$ by the definition of $\gd_{f(A)}$.
	Since $\omega(r^{-1}(A))\subseteq \omega(A)$ by the definition of $\omega$-limit sets, the finite-union property of $\omega$-limit sets (footnote~\ref{foot:omega-union-property}) implies that $\omega(\gd_{f(A)})\subseteq \omega(A) \cup \omega(r^{-1}(A)) \subseteq \omega(A)$.
	Taking $B = f(A)$ in \eqref{eq:prop-omega-1}, we thus obtain $$f^{-1}(\omega(f(A))) = \omega(f^{-1}(f(A))) \cup \omega(\gd_{f(A)}).$$
	Since $f$ is injective, $A = f^{-1}(f(A))$; substituting this into the first term on the right above and using $\omega(\gd_{f(A)})\subseteq \omega(A)$ yields  \eqref{eq:prop-omega-2}.
	This completes the proof.
\end{proof}

\begin{Co}\label{co:omega-forward-invariant}
	Let $H=(\st,\fs,\gd,\varphi,r)$ be a \THS~ satisfying Assumptions~\ref{assump:deterministic}, \ref{assump:inf-or-Zeno}, and \ref{assump:trapping-guard} with $r\colon \gd\to \st$ a closed map.  
	Then for any $U\subseteq \st$, $\omega(U)$ is forward invariant.
\end{Co}
\begin{proof}
	Let all notation be as in Prop.~\ref{prop:omega-limit-set-suspension} above, let $U\subseteq \st$ be an arbitrary subset, and define $X\coloneqq \pi\circ \iota(U)$.
	Since $\pi\circ \iota$ is injective, it follows that $U = (\pi\circ \iota)^{-1}(X)$.
	Hence Equation~\ref{eq:prop-omega-2} implies that
	\begin{equation}\label{eq:omega-cor-eq}
	\omega(U) =  (\pi\circ \iota)^{-1}(\omega(X)).
	\end{equation}
	It is well-known from classical dynamical systems theory that $\omega$-limit sets of continuous semiflows are forward invariant (this is also easy to prove directly), so $\omega(X)$ is forward invariant for $\Phi_H$.
    Since the collection of maximal $H$ executions is precisely the collection of $(\pi\circ\iota)$-preimages of $\Phi_H$ trajectories, Equation \eqref{eq:omega-cor-eq} implies that $\omega(U)$ is forward invariant.	
\end{proof}

\begin{Prop}\label{prop:attracting-repelling-pairs-suspension}
	Let $H=(\st,\fs,\gd,\varphi,r)$ be a \THS~ satisfying satisfying Assumptions~\ref{assump:deterministic}, \ref{assump:inf-or-Zeno}, \ref{assump:trapping-guard}, and \ref{assump:compact} with $r\colon \gd\to \st$ a closed map.  
	Let $H' = (\st',\fs',\gd',\varphi',r')$ be the relaxed hybrid system of Def. \ref{def:relaxed}, $\iota\colon \st\to \st'$ be the obvious embedding, and let $(\Sus_H,\Phi_H)$ be the hybrid suspension of Def. \ref{def:hybrid-suspension} with quotient $\pi\colon \st'\to \Sus_H$.
	
	Then $(A,A^*)$ is an attracting-repelling pair for $H$ if and only if  $(A,A^*) = ((\pi\circ \iota)^{-1}(B), (\pi\circ \iota)^{-1}(B^*))$ for some attracting-repelling pair $(B,B^*)$ for $\Phi_H$.	
\end{Prop}

\begin{proof}
	For purposes of readability, for this proof we define $f\coloneqq \pi \circ \iota$.	
	We begin by noting that, if $W\subseteq \Sus_H$ is any forward invariant set, then it follows from the definition of $\Phi_H$ that $$\Phi^{[1+T,\infty)}(W\cap f(\st)) \subseteq \Phi^{[1+T,\infty)}(W)\subseteq \Phi^{[T,\infty)}(W\cap f(\st))$$ for any $T \geq 0$. 
	By the definition of $\omega$-limit set, this implies that
	\begin{equation}\label{eq:prop4-omega-forward-inv}
	\omega(W) = \omega(W\cap f(\st)).
	\end{equation}
	Now let $U \coloneqq f^{-1}(W)$.
	Since $r$ is a closed map, \eqref{eq:prop4-omega-forward-inv} together with Equation \eqref{eq:prop-omega-2} of Prop. \ref{prop:omega-limit-set-suspension} yield 
	\begin{equation}\label{eq:prop4-preimage-omega-forward-inv}
	f^{-1}(\omega(W)) = f^{-1}(\omega(W\cap f(\st))) = f^{-1}(\omega(f(U))) = \omega(U).
	\end{equation}	
	
	Now let $(B,B^*)$ be an attracting-repelling pair for $\Phi_H$ and let $W$ be an open trapping neighborhood for $B$.
	Then $U\coloneqq f^{-1}(W)$ is open by continuity of $f$, and $\cl(U) \subseteq f^{-1}(\cl(W))$ by continuity of $f$.
	Hence clearly $U$ satisfies the conditions of Def. \ref{def:attracting-repelling} defining a trapping neighborhood, so $U$ determines an attracting set $A \coloneqq \omega(U)$.
	Since $W$ is forward invariant, \eqref{eq:prop4-preimage-omega-forward-inv} implies that $A = f^{-1}(\omega(W)) = f^{-1}(B)$. 
	
	We now show that $A^* = f^{-1}(B^*)$.
	Equation \eqref{eq:prop-omega-2} of Prop. \ref{prop:omega-limit-set-suspension} implies that $\omega(x) = f^{-1}(\omega(f(x)))$ for all $x\in \st$.
	Since also $A = f^{-1}(B)$, it follows that, for all $x\in \st$,
	$$\left[\omega(x) \cap A = \varnothing \right] \iff \left[ f^{-1}(\omega(f(x))) \cap f^{-1}(B) = \varnothing \right] \iff \left[ \omega(f(x))\cap B \cap f(\st) = \varnothing \right].$$
    The latter in turn holds if and only if $\omega(f(x)) \cap B = \varnothing$,  since  $\omega(f(x)) \cap B$ is forward invariant and $\Sus_H\setminus f(\st)$ contains no nonempty forward invariant subset.
	Thus, the repelling set $A^*$ dual to $A$ is given by $A^* = f^{-1}(B^*)$.
	We have now shown that, for every attracting-repelling pair $(B,B^*)$ for $\Phi_H$, $(f^{-1}(B), f^{-1}(B^*))$ is an attracting-repelling pair for $H$.

	To prove the converse, let $(A,A^*)$ be an attracting-repelling pair for $H$ and let $U\supseteq A$ be an open trapping neighborhood as in Def. \ref{def:attracting-repelling}.
	Then by the definition of the disjoint union and quotient topologies, $U'_0\coloneqq \pi_0\left(U \sqcup ((U\cap \gd)\times [0,1])\right)$ is an open subset of $\st'$, where $\pi_0\colon \st \sqcup(\gd\times [0,1])\to \st'$ is the quotient of Def.~\ref{def:relaxed}.
	Now define the set $U'\coloneqq U'_0 \cup \pi_0(r^{-1}(U)\times (0,1])\subseteq \st'$, which is also open (for similar reasons).
	$U'$ is saturated with respect to $\pi$ since 
	$$\pi^{-1}(\pi(U')) = r'(U'\cap \gd') \cup U' \cup (r')^{-1}(U') = U' \cup (r')^{-1}(U') = U'.$$ 
	The first equality follows from the definition of $\pi$ and the fact that $r'(\gd')\cap \gd' = \varnothing$.
	The second equality follows using the definition of $U'$ and the fact that $U$ is forward invariant: $r'(U'\cap \gd') = \iota\left(r(U \cap \gd)\cup r(r^{-1}(U))\right)\subseteq \iota(U)\subseteq U'$, so $r'(U'\cap \gd') \cup U' = U'$.
	The third equality follows since $r'(\gd')\subseteq \iota(\st)$ and since $U' \cap \iota(\st) = \iota(U)$, so $(r')^{-1}(U') = (r')^{-1}(\iota(U)) = \pi_0\left(r^{-1}(U) \times \{1\}\right) \subseteq U'$.

	Since $U'$ is open and saturated, $W\coloneqq \pi(U')\subseteq \Sus_H$ is open, and $\cl(W) = \pi(\cl(U'))$ since, by Lemma~\ref{lem:closed-quotient-metrizable},  $\pi$ is a closed map \cite[Prop.~2.30]{lee2010introduction}.
	These facts, together with the definitions of $\Phi_H$ and $U'$ and the fact that $U'$ is a trapping neighborhood, imply that $W$ is a trapping neighborhood for some attracting set $B\coloneqq \omega(W)$.
	Using the definition of $U'$ and the fact that $U'$ is saturated, it follows that $f^{-1}(W) = U$.
	Since $W$ is forward invariant, \eqref{eq:prop4-preimage-omega-forward-inv} therefore implies that $A = \omega(U) = f^{-1}(\omega(W)) =  f^{-1}(B)$.
	Finally, repeating the argument from the first part of the proof verbatim shows that $A^* = f^{-1}(B^*)$.
	Hence every attracting repelling-pair $(A,A^*)$ for $H$ is of the form $(f^{-1}(B), f^{-1}(B^*))$ for some attracting-repelling pair $(B,B^*)$ for $\Phi_H$.
	This completes the proof.	
\end{proof}

\subsection{Chain equivalence in the hybrid suspension}\label{sec:chain-equiv-susp}
To prove our main results, we will relate the chain recurrent set of an \MHS~ with the chain recurrent set of its hybrid suspension semiflow.
Prop.~\ref{prop:hybrid-chain-equiv-vs-suspension-chain-equiv} below establishes this relationship.
We prove Prop.~\ref{prop:hybrid-chain-equiv-vs-suspension-chain-equiv} using the following  three technical lemmas, the proofs of which are deferred to SM \S \ref{app:technical}.

The first technical lemma (Lemma~\ref{lem:niceify}) shows that---under Assumptions~\ref{assump:deterministic}, \ref{assump:inf-or-Zeno}, \ref{assump:trapping-guard}, and \ref{assump:compact}---the Conley relation is unaffected by restricting the set of hybrid $(\epsilon,T)$-chains to a ``nice'' subset for which there must be time-separation between any reset jump and a subsequent continuous-time jump.
In other words, such \emph{nice} chains are not allowed to have ``double jumps'' like the one shown in Fig. \ref{fig:eps-T-chain}.
We make this notion precise in Def.~\ref{def:nice-chains-conley-rel}; a nice chain is shown in Fig. \ref{fig:eps-T-chain-nice}.

\begin{figure}
	\centering
	\def\svgwidth{1.0\columnwidth}
	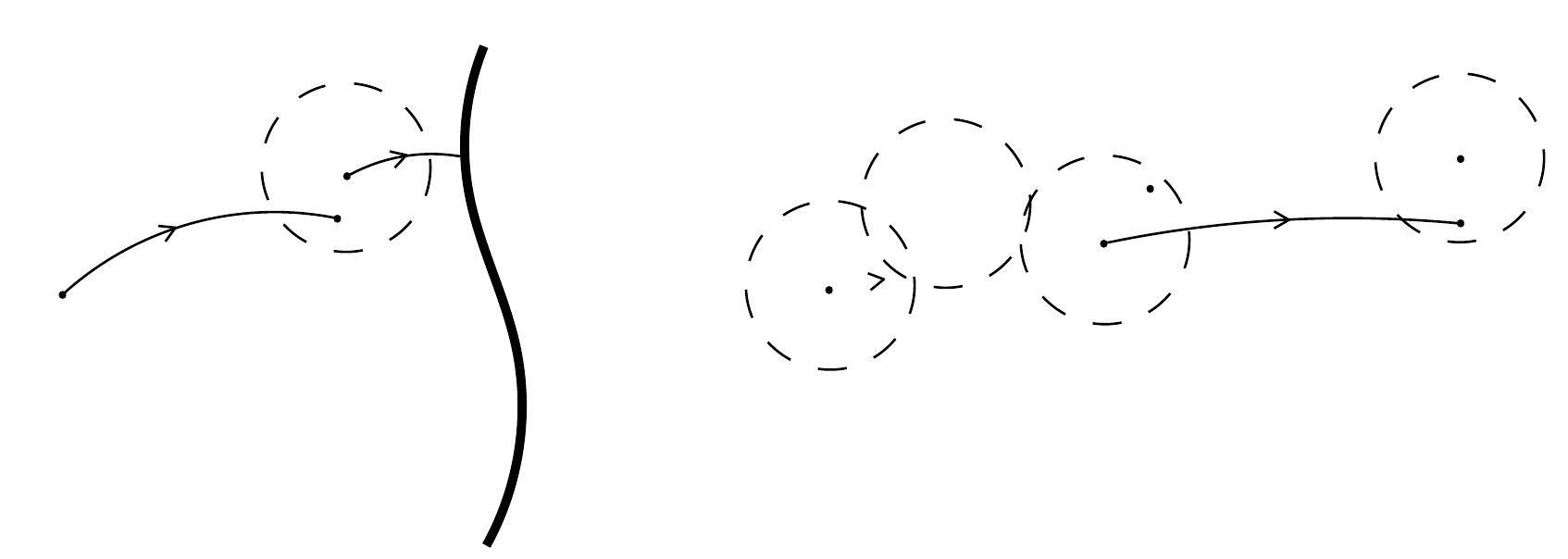
	\caption{A \textbf{nice} $(\epsilon, T)$-chain from $x\in \st$ to $y\in \st$. Notice that, unlike in the $(\epsilon,T)$-chain of Fig.~\ref{fig:eps-T-chain}, no ``double jumps'' are allowed.}\label{fig:eps-T-chain-nice}
\end{figure}

\begin{Def}\label{def:nice-chains-conley-rel}
	Let $H$ be an \MHS.
	We say that a chain $\chi = (N, \tau, \eta, \gamma) \in \conleyet_H$ is \textbf{nice} if $\tau_{\eta_k} - \tau_{(\eta_k-1)} \ge T$ for all $k \ge 1$. 
	We denote the set of nice $(\epsilon, T)$-chains in $H$ by $\widehat{\conleyet_H}$.
	We let $\widehat{\conley_H}$ denote the Conley relation with respect to the set of nice chains, i.e.,
	$$(x,y)\in \widehat{\conley_H} \iff \text{for all } \epsilon, T > 0\colon  \widehat{\conleyet_H}(x,y) \neq \varnothing.$$
	As in Def. \ref{def:hybrid-conley-relation}, we sometimes use the more intuitive notation $\widehat{\conley_H}(x,y)$ in place of $(x,y)\in \widehat{\conley_H}$.	
\end{Def}

\begin{restatable}{Lem}{Nicefy} \label{lem:niceify}
	Let $H=(\st,\fs,\gd,\varphi,r)$ be a \MHS~ satisfying Assumptions~\ref{assump:deterministic}, \ref{assump:inf-or-Zeno}, \ref{assump:trapping-guard}, and \ref{assump:compact}.			
	Then
	$$\conley_H = \widehat{\conley_H}.$$
	In particular, two points of $\st$ are chain equivalent if and only if they are chain equivalent with respect to nice chains only.
\end{restatable}

The following lemma shows that---under Assumptions~\ref{assump:deterministic}, \ref{assump:inf-or-Zeno}, \ref{assump:trapping-guard}, and \ref{assump:compact}---the Conley relation for $H$ is unaffected by allowing additional jumps to occur within the relaxed hybrid system $H'$.
\begin{restatable}{Lem}{RelaxedConley}\label{lem:relaxed-conley-relation}
	Let $H=(\st,\fs,\gd,\varphi,r)$ be an \MHS~ satisfying Assumptions~\ref{assump:deterministic}, \ref{assump:inf-or-Zeno}, \ref{assump:trapping-guard}, and \ref{assump:compact}.			
	Let $H' = (\st', \fs', \gd', \varphi', r')$ be the associated relaxed hybrid system equipped with any compatible extended metric making $H'$ an \MHS.
	Let $\iota : \st \to \st'$ be the obvious embedding.  Then for all $x,y \in \st$, we have
	$$\conley_H(x,y) \iff \conley_{H'}(\iota(x), \iota(y))$$
\end{restatable}

The following lemma is a minor adaptation of  \cite[Prop.~B.2.19]{colonius2000dynamics}.
\begin{restatable}{Lem}{Colonius}\label{lem:colonius-result-generalized}
Let $X$ be a compact metric space and $\Phi\colon [0,\infty)\times X \to X$ be a continuous semiflow. 
Consider $x,y \in X$ and fix $T > 0$.
If for every $\epsilon > 0$ there exist $(\epsilon,T)$-chains from (i) $x$ to $y$ and (ii) $y$ to $x$, then $x$ and $y$ are chain equivalent.
\end{restatable}
\begin{Rem}
Formally speaking, Lemma~\ref{lem:colonius-result-generalized} is discussing $(\epsilon,T)$-chains and chain equivalence as defined in Def.~\ref{def:eps-T-chains} and \ref{def:hybrid-chain-equivalence}, as opposed to the standard definitions for semiflows (see Ex.~\ref{ex:special-case-continuous-time}) to which \cite[Prop.~B.2.19]{colonius2000dynamics} applies. 
However, the proofs are similar for either definition.
Furthermore, our proof actually directly shows that $x$ and $y$ are also chain equivalent according to the standard definition, but this also follows from the stated conclusion and Ex.~\ref{ex:special-case-continuous-time}.
\end{Rem}
We now come to the main result of this section.

\begin{Prop}[The chain equivalence classes of a compact hybrid system and its suspension]\label{prop:hybrid-chain-equiv-vs-suspension-chain-equiv}
Let $H=(\st,\fs,\gd,\varphi,r)$ be an \MHS~ satisfying Assumptions \ref{assump:deterministic}, \ref{assump:inf-or-Zeno}, \ref{assump:trapping-guard}, and \ref{assump:compact}.			
Let $H' = (\st',\fs',\gd',\varphi',r')$ be the relaxed hybrid system of Def. \ref{def:relaxed}, let $\iota\colon \st\to \st'$ be the obvious embedding, and let $(\Sus_H,\Phi_H)$ be the hybrid suspension of Def. \ref{def:hybrid-suspension} with quotient $\pi\colon \st'\to \Sus_H$.

Then for any choice of compatible extended metric on $\Sus_H$, $x,y\in \st$ are chain equivalent for $H$ if and only if $\pi\circ \iota (x),\pi\circ \iota (y)\in \Sus_H$ are chain equivalent for $\Phi_H$. 
In particular, it follows that $R(H) = (\pi \circ \iota)^{-1}(R(\Phi_H))$, where $R(\Phi_H)$ is the chain recurrent set for $\Phi_H$.
\end{Prop}

\begin{Rem}\label{rem:chain-compatibility-for-hybrifold-fails}
	Prop.~\ref{prop:hybrid-chain-equiv-vs-suspension-chain-equiv} would become false if the hybrid suspension $\Sus_H$ was replaced by the (generalized) hybrifold $M_H$ of $H$ defined in SM \S \ref{app:hybrifold}.
	For example, consider a discrete-time dynamical system, i.e., a hybrid system $H$ with $\st = \gd$ and $\fs=\varnothing$ (cf. Ex.~\ref{ex:special-case-discrete-time}).
	In general $M_H$ may not be metrizable (even if $\st$ is a compact metric space), as illustrated by the example of a discrete-time dynamical system given by iterating an irrational rotation of the circle.
	However, even if $M_H$ happens to be metrizable, every point of $M_H$ is a stationary point for the (generalized) hybrifold semiflow, and therefore every point of $M_H$ is chain recurrent.
	On the other hand, the chain recurrent set of $H$ is arbitrary.
\end{Rem}

The following immediate corollary of Prop. \ref{prop:hybrid-chain-equiv-vs-suspension-chain-equiv} concerns the classical suspension of a discrete-time dynamical system and is probably well-known, although we could not find a reference in the literature.
See SM \S \ref{app:suspension} or \cite[p.~797,~pp.~21--22]{smale1967differentiable,brin2002introduction} for a primer on the classical suspension semiflow.
\begin{Co}\label{co:chain-equiv-vs-classical-suspension-chain-equiv}
	Consider the discrete-time dynamical system defined by a continuous map $f:X\to X$ of a compact metric space.
	Let $\Sigma_f \coloneqq \frac{X\times [0,1]}{(x,1)\sim (0,f(x))}$ be the classical suspension (mapping torus) of $f$, $\Phi\colon [0,\infty) \times \Sigma_f\to \Sigma_f$ be the suspension semiflow,  $\iota\colon X \to  X\times [0,1]$ be the embedding $X\hookrightarrow X\times \{0\}$, and $\pi\colon X\times [0,1]\to \Sigma_f$ be the quotient map.
	
	Then for any choice of compatible extended metric on $\Sus_f$, $x,y\in X$ are chain equivalent for $f$ if and only if $\pi\circ\iota(x),\pi\circ \iota(y)$ are chain equivalent for $\Phi$.
	In particular, $R(f) = (\pi\circ\iota)^{-1}(R(\Phi))$.
\end{Co}
The classical analogue of the following corollary of Prop.~\ref{prop:hybrid-chain-equiv-vs-suspension-chain-equiv} is well-known.
\begin{Co}\label{co:hybrid-chain-recurrent-set-closed-invariant}
	Let $H = (\st,\fs,\gd,\varphi,r)$ be an \MHS~ satisfying Assumptions \ref{assump:deterministic}, \ref{assump:inf-or-Zeno}, \ref{assump:trapping-guard}, and \ref{assump:compact}.
	Then the chain equivalence classes of $H$ are closed and forward invariant.
	Similarly, $R(H)$ is closed and forward invariant.
	Additionally, $\omega(x)\subseteq R(H)$ for any $x\in \st$.
\end{Co}
\begin{proof}[Proof of Cor.~\ref{co:hybrid-chain-recurrent-set-closed-invariant}]
	For the standard, classical definition of the Conley relation for the semiflow $\Phi_H$ (and for any choice of compatible metric on $\Sus_H$), it is well known that all chain equivalence classes for $\Phi_H$, as well as the chain recurrent set for $\Phi_H$, are closed and forward invariant (this is also easy to prove directly).
	By Ex.~\ref{ex:special-case-continuous-time}, the same is true if our definition of the Conley relation (Def.~\ref{def:hybrid-conley-relation}) for $\Phi_H$ is used instead.
	It follows that the same is true for $H$ since (i) Prop.~\ref{prop:hybrid-chain-equiv-vs-suspension-chain-equiv} implies that the chain equivalence classes for $H$ (and hence also $R(H)$) are $(\pi\circ\iota)$-preimages of the chain equivalence classes for $\Phi_H$, (ii) $\pi\circ \iota$ is continuous, and (iii)  the collection of maximal $H$ executions is precisely the collection of $(\pi\circ\iota)$-preimages of $\Phi_H$ trajectories.
	
	Fix $x\in \st$ and define $[x]\coloneqq \pi\circ \iota(x)$. 
	The final assertion follows from Prop.~\ref{prop:hybrid-chain-equiv-vs-suspension-chain-equiv} and Equation \eqref{eq:prop-omega-2} of Prop.~\ref{prop:omega-limit-set-suspension} (which applies by Remark~\ref{rem:compact-guard-hausdorff-automatically-closed-map}): $$\omega(x) = (\pi\circ \iota)^{-1}(\omega([x])) \subseteq (\pi\circ \iota)^{-1}(R(\Phi_H)) = R(H).$$
	The set inclusion follows from the well-known fact that, for a semiflow, the chain recurrent set contains the $\omega$-limit set of any point (cf. \cite[II.6.3.C]{conley1978isolated}).
\end{proof}

\begin{proof}[Proof of Prop.~\ref{prop:hybrid-chain-equiv-vs-suspension-chain-equiv}]
In the following, we let $d_{\st}$ be the given extended metric on $\st$ and $d_{\st'}, d_{\Sus_H}$ be any metrics on $\st', \Sus_H$ which are compatible with their respective topologies, and we use the notation $[x] \coloneqq \pi(x)$ for $x\in \st'$.
Through a mild abuse of notation, we also use the notation $[x]\coloneqq \pi\circ \iota(x)$ for $x\in \st$ and $[(z,t)]\coloneqq \pi \circ \pi_0(z,t)$, where $\pi_0\colon \st \sqcup (\gd\times [0,1]) \to \st'$ is the quotient of Def.~\ref{def:relaxed}.

We first show that, if $x,y \in \st$ are chain equivalent, then $[x]$ and $[y]$ are chain equivalent for $\Phi_H$.\footnote{In this proof we are using Def.~\ref{def:eps-T-chains} and \ref{def:hybrid-conley-relation} for the definition of the Conley relation for $\Phi_H$, although Ex.~\ref{ex:special-case-continuous-time} shows that this is equivalent to the classical definition.}
Since $\st$ is compact it follows that the map $\pi \circ \iota \colon \st\to \Sus_H$ is uniformly continuous with respect to $d_{\st}, d_{\Sus_H}$.
This implies that, for any $\epsilon > 0$, there exists $\delta > 0$ such that every \emph{nice} $(\delta, T+1)$-chain for $H$ maps to an $(\epsilon, 1)$-chain for $\Phi_H$ under $\pi\circ \iota$.
Lemma \ref{lem:niceify} implies that, for every $\delta, T > 0$, there are nice $(\delta, T+1)$-chains from (i) $x$ to $y$ and (ii) $y$ to $x$, so it follows that there are $(\epsilon, 1)$-chains from (i) $[x]$ to $[y]$ and (ii) $[y]$ to $[x]$ for every $\epsilon > 0$.
Hence Lemma \ref{lem:colonius-result-generalized} implies that $[x]$ and $[y]$ are chain equivalent for $\Phi_H$.

To complete the proof we now show that, if $x,y\in R(H)$ are such that $[x]$ and $[y]$ are chain equivalent for $\Phi_H$, then $x$ and $y$ are chain equivalent.
It suffices to prove the stronger claim that $[x]$ being Conley related to $[y]$ for $\Phi_H$ implies $\conley_H(x,y)$ for arbitrary $x,y \in \st$. 
And in order to prove this, by Lemma \ref{lem:relaxed-conley-relation} it suffices to prove that $[x]$ being Conley related to $[y]$ for $\Phi_H$ implies $\conley_{H'}(\iota(x),\iota(y))$.
We prove the latter claim below in two steps which we first briefly describe in the following paragraph.
To improve readability in the remainder of the proof, we henceforth use the notation $[S]\coloneqq \pi\circ \iota(S)\subseteq \Sus_H$ if $S\subseteq \st$, $[S]\coloneqq \pi(S)\subseteq \Sus_H$ if $S\subseteq \st'$, and $[S]\coloneqq \pi\circ\pi_0(S)$ if $S\subseteq \st \sqcup (\gd\times [0,1])$.  

In Step 1 below, we will show that $(\epsilon,T)$-chains for arbitrary $\epsilon, T > 0$ can be constructed between $[x],[y]\in \pi\circ \iota(\st)$ with the property that no jump points in the chain belong to $[\gd\times (\frac{1}{2},1)]$.
We refer to such chains as $(\epsilon, T)$-\concept{special chains}.
In Step 2 we will use Step 1 as a tool to prove that, for any $\epsilon, T> 0$, there exists an $(\epsilon, T)$-chain from $x$ to $y$ for the relaxed system $H'$ if $[x],[y]$ are Conley related for $\Phi_H$.
This will show that $\conley_{H'}(x,y)$, and hence $\conley_{H}(x,y)$ by Lemma~\ref{lem:relaxed-conley-relation}, as desired.

\textbf{Step 1:} 
Define $U' \coloneqq \pi_0(\gd\times [\frac{1}{4},1])\subseteq \st'$, $V'\coloneqq \pi_0(\gd\times (\frac{1}{2},1))\subseteq U'$, and the continuous maps $\rho'\colon U' \to \gd'$ via $\rho'(\pi_0(z,t))\coloneqq \pi_0(z,1)$ and $\mu'\colon U' \to [0,\frac{3}{4}]$ via $\mu'(\pi_0(z,t))\coloneqq 1-t$.
Since $U'$ and $[U']$ are compact metrizable spaces, $\pi|_{U'}\colon U'\to [U']$ is a quotient map since it is a continuous, closed, and surjective map.
The maps  $\pi|_{U'} \circ \rho'$ and $\mu'$ are both constant on fibers of $\pi|_{U'}$ since the restriction of $\pi$ to $\pi_0(\gd\times [\frac{1}{4},1))$ is injective, $\mu'|_{\gd'} \equiv 0$, and $\left(\pi|_{U'} \circ \rho'\right)|_{\gd'} = \pi|_{\gd'}$ since $\rho'|_{\gd'}=\id_{\gd'}$.
Thus, by the universal property of the quotient topology \cite[Thm~3.70]{lee2010introduction}, $\pi|_{U'} \circ \rho'$ and $\mu'$ descend to a continuous retraction $\nu\colon [U'] \to [\gd']$ and a continuous map $\alpha \colon [U'] \to [0,\frac{3}{4}]$, respectively.
Note that $\nu$ preserves trajectories, and that $\alpha$ is the ``time-to-impact-$[\gd']$ map'' for points in $[U']$.

Now fix $\epsilon, T > 0$ and let $[x],[y]\in [\st]$ be Conley related for $\Phi_H$.
Since $[U']$, $\Sus_H$, and $\cl([V'])$ are compact, there exists $\delta\in (0,\epsilon)$ such that
\begin{align}
[x],[y] \in [U'], \, d_{\Sus_H}([x],[y]) < \delta &\implies d_{\Sus_H}(\nu([x]),\nu([y])) < \epsilon\label{eq:nu-continuity}\\
[x]\in [V'],\, [y]\in \Sus_H\setminus \interior([U']), \, d_{\Sigma_H}([x],[y]) < \delta  &\implies d_{\Sus_H}(\nu([x]), [y]) < \epsilon.\label{eq:near-top-of-cylinder}
\end{align}
Indeed, suppose there did not exist $\delta > 0$ such that (\ref{eq:near-top-of-cylinder}) held.
Then there exist sequences $(v_n)_{n\in \N} \subseteq [V']$ and $(w_n)_{n\in \N} \subseteq \Sus_H \setminus \interior([U'])$ with $d_{\Sus_H}(v_n, w_n) \to 0$ and $d_{\Sus_H}(\nu(v_n), w_n) \ge \epsilon$ for all $n$.  
By passing to subsequences, we may assume $v_n \to v \in \cl([V'])$ and $w_n \to w \in \Sus_H\setminus \interior([U'])$ with $d_{\Sus_H}(\nu(v),w) \ge \epsilon$.  
Since $d_{\Sus_H}(v_n, w_n) \to 0$, we have $v = w$.  
Thus, $v \in \cl([V']) \cap (\Sus_H \setminus \interior([U'])) = [\gd']$ and $d_{\Sus_H}(\nu(v), v) \ge \epsilon$, a contradiction since $\nu|_{[\gd']}=\id_{[\gd']}$.

Let\footnote{The semiflow $\Phi_H$ defines an \MHS~ $\widetilde{H}$ with guard $\varnothing$ (cf. Ex.~\ref{ex:special-case-continuous-time}). To avoid introducing extra notation, here we conflate $\Phi_H$ with $\widetilde{H}$ by writing, e.g., ${\conley^{\delta, T+3/4}_{\Phi_H}([x], [y])}$ instead of ${\conley^{\delta, T+3/4}_{\widetilde{H}}([x], [y])}$. 
We additionally remark that, since the guard for $\widetilde{H}$ is $\varnothing$, chains can only have continuous-time jumps, so every chain $\chi = (N,\tau,\eta,\gamma)$ for $\Phi_H$ satisfies $\eta = (0,1,2,\ldots, N)$.}  $\chi^{(0)} = (N, \tau^{(0)},\eta^{(0)},\gamma^{(0)}) \in {\conley^{\delta, T+3/4}_{\Phi_H}([x], [y])}$; recall from Def.~\ref{def:eps-T-chains} that $N\geq 1$.
We will now modify the chain $\chi^{(0)}$ inductively.
Fix $i\in \{0,\ldots, N-1\}$ and assume that, if $i \geq 1$,\footnote{If $i = 0$ we assume nothing, so that the base case of the induction argument is included in this one.}  we have modified the first $(i+1)$ arcs ($\gamma_0^{(0)},\ldots, \gamma_i^{(0)}$) to obtain a chain $\chi^{(i)} = (N,\tau^{(i)},\eta^{(i)},\gamma^{(i)})\in {\conley^{\epsilon, T}_{\Phi_H}([x], [y])}$ such that (a) the sub-chain obtained by throwing away the first $(i+1)$ arcs of $\chi^{(i)}$ is either a single arc (if $i = N-1$) or a $(\delta, T+3/4)$-chain; and (b) for all $j\in \{1,\ldots, i\}$: $\gamma_{j-1}^{(i)}(\tau_{j}^{(i)}),\gamma_{j}^{(i)}(\tau_{j}^{(i)}) \not \in [V']$.
\begin{itemize}
	\item If both $\gamma_i^{(i)}(\tau_{i+1}^{(i)}),\gamma_{i+1}^{(i)}(\tau_{i+1}^{(i)})\in [U']$, then we replace $\gamma_i^{(i)}$ with the arc
	\begin{equation}\label{eq:thm3-gamma-i-replace}
	[\tau_i^{(i)}, \tau_{i+1}^{(i)} + \alpha(\gamma_i^{(i)}(\tau_{i+1}^{(i)}))]\ni t\mapsto \Phi_H^{t - \tau_i^{(i)}}\left(\gamma_i^{(i)}(\tau_i^{(i)})\right)
	\end{equation}
	 and $\gamma_{i+1}^{(i)}$ with the arc 
	\begin{equation}\label{eq:thm3-gamma-i-plus1-replace}
	[\tau_{i+1}^{(i)} + \alpha(\gamma_{i}^{(i)}(\tau_{i+1}^{(i)})),\tau_{i+2}^{(i)}] \ni t  \mapsto \Phi_H^{t + \alpha(\gamma_{i+1}^{(i)}(\tau_{i+1}^{(i)})) - \tau_{i+1}^{(i)} - \alpha(\gamma_{i}^{(i)}(\tau_{i+1}^{(i)}))}\left(\gamma_{i+1}^{(i)}(\tau_{i+1}^{(i)})\right).
	\end{equation} 
	\item If $\gamma_i^{(i)}(\tau_{i+1}^{(i)})\in [V']$ and $\gamma_{i+1}^{(i)}(\tau_{i+1}^{(i)}) \in \Sus_H\setminus \interior([U'])$, then we replace $\gamma_i^{(i)}$ with the arc defined by \eqref{eq:thm3-gamma-i-replace}, but we do not modify $\gamma_{i+1}^{(i)}$.
	\item If $\gamma_{i+1}^{(i)}(\tau_{i+1}^{(i)})\in [V']$ and $\gamma_i^{(i)}(\tau_{i+1}^{(i)}) \in \Sus_H\setminus \interior([U'])$, then we replace $\gamma_{i+1}^{(i)}$ with the arc defined by \eqref{eq:thm3-gamma-i-plus1-replace}, but we do not modify $\gamma_i^{(i)}$.
\end{itemize}
The upper bound $\alpha(\slot) \leq \frac{3}{4}$ and Equations \eqref{eq:nu-continuity}, \eqref{eq:near-top-of-cylinder} can be used to show that, after redefining the sequences $\eta^{(i)}$ and $\tau^{(i)}$ accordingly, the result is a chain 
$\chi^{(i+1)} = (N,\tau^{(i+1)},\eta^{(i+1)},\gamma^{(i+1)})\in {\conley^{\epsilon, T}_{\Phi_H}([x], [y])}$ such that (a) the sub-chain obtained by throwing away the first $(i+2)$ arcs of $\chi^{(i+1)}$ is either empty (if $i = N -1)$, a single arc (if $i = N-2$), or a $(\delta, T+ 3/4)$-chain; and (b) for all $j\in \{1,\ldots, i+1\}$: $\gamma_{j-1}^{(i+1)}(\tau_{j}^{(i+1)}),\gamma_{j}^{(i+1)}(\tau_{j}^{(i+1)}) \not \in [V']$.

Hence by induction we obtain a chain $\chi \in {\conley^{\epsilon, T}_{\Phi_H}([x], [y])}$ satisfying $\gamma_i(\tau_i), \gamma_i(\tau_{i+1}) \not \in [V']$ for every arc of $\chi$.  
This shows that there exists an $(\epsilon, T)$-special chain from $[x]$ to $[y]$ and completes the proof of Step 1.

\textbf{Step 2:} 
Fix $\epsilon > 0$ and define $W'\coloneqq \iota(\st) \cup \pi_0(\gd \times [0,\frac{1}{2}])\subseteq \st'$.
Since $W'$ is compact and $\pi|_{W'}$ is injective, $\pi|_{W'}\colon W' \to [W'] = \Sus_H\setminus [V']$ is a homeomorphism of compact metric spaces \cite[Lem~4.50.d]{lee2010introduction}.
It follows that the inverse homeomorphism $(\pi|_{W'})^{-1}\colon \Sus_H \setminus [V']\to W'$ is uniformly continuous, so there exists $\delta > 0$ such that $d_{\Sus_H}([z],[w]) < \delta$ implies	 that $d_{\st'}(z,w) < \epsilon$ for all $z,w\in W'$.      

Now fix $\epsilon, T > 0$, let $\delta > 0$ as in the above paragraph, and let $x,y\in \st$ be such that $[x],[y]\in [\st]\subseteq [W']$ are Conley related for $\Phi_H$.
By Step 1, there exists a $(\delta,T)$-special chain $\chi = (N,\tau,\eta,\gamma)$ from $[x]$ to $[y]$.
Since $\pi|_{W'}$ is a homeomorphism onto its image, we can define a sequence of ``lifted and reset-subdivided'' continuous arcs $(\tilde{\gamma}_j)$ in $\st'$ by first lifting each component of $\gamma_i^{-1}([W'])$ via $(\pi|_{W'})^{-1}$, then extending each lifted component terminating at a point $\pi_0(z,\frac{1}{2})$ in the boundary of $W'$ via concatenation with $\left(t\in [0,\frac{1}{2}] \mapsto \pi_0(z, t+\frac{1}{2})\right)$.
Since $d_{\Sigma_H}(\gamma_i(\tau_{i+1}), \gamma_{i+1}(\tau_{i+1})) < \delta$ for each $i$, it follows from our choice of $\delta$ that the resulting family $(\tilde{\gamma}_j)$ of arcs yields an $(\epsilon, T)$-chain for $H'$.\footnote{For this step of the proof it is crucial that our definition of $(\epsilon,T)$-chains (Def.~\ref{def:eps-T-chains}) allows for ``double jumps,'' as illustrated in Fig.~\ref{fig:eps-T-chain}.} 
This shows that $x,y\in \st$ are Conley related for $H'$.
By Lemma \ref{lem:relaxed-conley-relation}, this shows that $x,y\in \st$ are also Conley related for $H$ and completes the proof.
\end{proof}

\subsection{Proofs of Theorems \ref{th:conley-decomp} and \ref{th:hybrid-conley}}\label{sec:proofs-main-theorems}
We are now in a position to prove our main theorems, which we restate for convenience.

\ThmConleyDecomp*

\begin{proof}
	Let $H' = (\st',\fs',\gd',\varphi',r')$ and $(\Sus_H, \Phi_H)$ be the relaxed hybrid system and hybrid suspension of Def. \ref{def:relaxed} and \ref{def:hybrid-suspension} (equipped with any compatible extended metrics), $\iota\colon \st\to \st'$ be the obvious embedding, and $\pi\colon \st' \to \Sus_H$ be the quotient map of Def.~\ref{def:hybrid-suspension}.
	Letting $R(\Phi_H)$ denote the chain recurrent set for $\Phi_H$, the Conley decomposition theorem for semiflows \cite[Thm~2]{hurley1995chain} and Ex.~\ref{ex:special-case-continuous-time} imply that
	\begin{equation}\label{eq:conley-decomp-semiflow}
	R(\Phi_H) = \bigcap \{B\cup B^*\mid B \textnormal{ is an attracting set for } \Phi_H.\},
	\end{equation} 
	and that $\pi\circ\iota(x)$ is chain equivalent to $\pi\circ\iota(y)$ if and only if either $\pi\circ\iota(x), \pi\circ\iota(y)\in B$ or $\pi\circ\iota(x), \pi\circ\iota(y)\in B^*$ for every attracting-repelling pair $(B,B^*)$ for $\Phi_H$.
	Prop.~\ref{prop:hybrid-chain-equiv-vs-suspension-chain-equiv} implies that $R(H) = (\pi \circ \iota)^{-1}(R(\Phi_H))$ and, furthermore, that the chain equivalence classes of $H$ are precisely the $(\pi\circ \iota)$-preimages of chain equivalence classes for $\Phi_H$.
	Hence to complete the proof it would suffice to show that $(A,A^*)$ is an attracting-repelling pair for $H$ if and only if  $(A,A^*) = ((\pi\circ \iota)^{-1}(B), (\pi\circ \iota)^{-1}(B^*))$ for some attracting-repelling pair $(B,B^*)$ for $\Phi_H$, but this is the content of Prop. \ref{prop:attracting-repelling-pairs-suspension}.
	This completes the proof.
\end{proof}

\ThmFund*

\begin{proof}
Let $H' = (\st',\fs',\gd',\varphi',r')$ be the relaxed hybrid system of Def. \ref{def:relaxed} and $\iota\colon \st\to \st'$ the obvious embedding, and let $(\Sus_H,\Phi_H)$ be the hybrid suspension of Def. \ref{def:hybrid-suspension} with quotient $\pi\colon \st'\to \Sus_H$.

By Prop.~\ref{prop:hybrid-suspension}, $\Sus_H$ is compact and metrizable.
Since $\Sus_H$ is compact, the Conley relation is independent of the choice of compatible extended metric on $\Sus_H$ (see Remark~\ref{rem:conley-rel-compact-metric-independent}).
Hence (after equipping $\Sus_H$ with any compatible extended metric and appealing to Ex.~\ref{ex:special-case-continuous-time}) we may apply the fundamental theorem of dynamical systems for semiflows \cite[Thm~1.1]{patrao2011existence} to conclude that there exists a complete Lyapunov function $V\colon \Sus_H\to \R$ for $\Phi_H$. 
Letting $R(\Phi_H)$ denote the chain recurrent set for $\Phi_H$, this means that $V$ is a continuous function such that (i) $t\mapsto V(\Phi_H^t(\pi(x)))$ is strictly decreasing for all $\pi(x)\not \in R(\Phi_H)$, (ii) $V(R(\Phi_H))$ is nowhere dense in $\R$, and (iii) for all $\pi(x),\pi(y) \in R(\Phi_H)$:  $\pi(x)$ and $\pi(y)$ are chain equivalent if and only if  $V(\pi(x)) = V(\pi(y))$.

Define  $L\colon \st\to \R$ via $L\coloneqq V\circ \pi \circ \iota$; $L$ is continuous since $L$ is a composition of continuous functions.
By construction, we have that (i) for every $x\in \fs\setminus (\pi\circ \iota)^{-1}(R(\Phi_H))$, $t > 0$, and $y \in \chi_x(t)$, $L(y) < L(x)$; and (ii) if $x\in \gd\setminus (\pi\circ \iota)^{-1}(R(\Phi_H))$, then $L(r(x)) < L(x)$. 
Prop.~\ref{prop:hybrid-chain-equiv-vs-suspension-chain-equiv} implies that $\pi\circ\iota(R(H))\subseteq R(\Phi_H)$, so $L(R(H)) = V\circ \pi\circ \iota(R(H))\subseteq V(R(\Phi_H))$; therefore, $L(R(H))$ is nowhere dense in $\R$. 
It remains only to show that $x,y \in R(H)$ are chain equivalent if and only if  $L(x) = L(y)$.
Prop.~\ref{prop:hybrid-chain-equiv-vs-suspension-chain-equiv} implies that $x,y\in R(H)$ if and only if $\pi\circ\iota(x),\pi\circ\iota(y)\in R(\Phi_H)$, and that $x,y\in R(H)$ are chain equivalent if and only if $\pi\circ\iota(x),\pi\circ\iota(y)\in R(\Phi_H)$ are chain equivalent.
By the previous paragraph, $\pi\circ\iota(x),\pi\circ\iota(y)\in R(\Phi_H)$  are chain equivalent if and only if $V(\pi\circ\iota(x)) = V(\pi\circ\iota(y))$.
Since $L = V\circ\pi\circ\iota$, we have shown that $x,y \in R(H)$ are chain equivalent if and only if  $L(x) = L(y)$.
\end{proof}

\section{Conclusion}\label{sec:conclusion}
Using the language of hybrid systems, we have obtained a simultaneous generalization (Theorem \ref{th:hybrid-conley}) of both the continuous-time and discrete-time versions of Conley's fundamental theorem \cite{conley1978isolated,franks1988variation}.  As in the classical setting, our theorem asserts the existence of a globally-defined \emph{complete Lyapunov function} (Def. \ref{def:lyapunov}).
We have also proved a result (Theorem \ref{th:conley-decomp}) generalizing  Conley's decomposition theorem, which asserts that the chain recurrent set (Def.~\ref{def:hybrid-chain-recurrent}) is the intersection of all attracting-repelling pairs (Def.~\ref{def:attracting-repelling}). 

While this unification of the continuous and discrete is pleasingly parsimonious, our motivation is not
merely parsimony for its own sake. Our long-term aim, motivated particularly by applications in robotics and biomechanics---bearing  not simply  on legged locomotion \cite{Holmes_Full_Koditschek_Guckenheimer_2006,revzen2015data,seipel2017conceptual}, but, indeed, central to the larger prospects for a physically grounded, sensorimotor-coherent, and theoretically sound disciplinary foundation \cite[Sec.~4.1.2]{Koditschek_2021}---is to continue advancing the program\footnote{This program has contributions from many investigators. We only mention a few: \cite{Back_Guckenheimer_Myers_1993,Guckenheimer_1995,ye1998stability,alur2000discrete,simic2000towards,simic2001structural,simic2002hybrid,lygeros2003dynamical,simic2005towards,haghverdi2005bisimulation,Burden_Revzen_Sastry_2015,Goebel_Sanfelice_Teel_2009,lerman2016category,Johnson_Burden_Koditschek_2016,Burden_Sastry_Koditschek_Revzen_2016,burden2018contraction,burden2018generalizing,clark2019poincare,clark2020poincare,lerman2020networks}.} of developing hybrid dynamical systems theory to the same footing
as its more mathematically mature parents, the theories of continous-time and discrete-time dynamical
systems.  
For example, beyond the constructive applications of Conley's theorems discussed in the introduction, practitioners might well choose to use their appearance as a kind of litmus test against which to judge the relative merits of the many different hybrid systems models that have appeared in the literature. 
Models that do not possess such a decomposition into chain-recurrent and gradient-like parts might be subject to greater scrutiny---their questionably disorderly behavior only tolerated in consequence of expressing some physical property essential to the phenomena of interest.\footnote{See the second paragraph of SM \S \ref{app:JBK-relate} for a relevant discussion.
}
As a case in point, the absence of any viable notion of steady state behavior occasioned  by the departure from the trapping guard condition (Def.~\ref{def:trapping-guards}) of Ex.~\ref{ex:omega-chain-pathologies} (see Fig.~\ref{fig:omega-chain-pathology}) begins to suggest the potentially wild incoherence of otherwise seemingly well-formulated \THS.
Thus, by the same token, we hope that our presentation of sufficient conditions for a Conley theory of hybrid systems may encourage  more theorists to help determine  which properties are necessary. Indeed, despite the ubiquity of hybrid systems in engineering, mathematicians have mostly avoided them, perhaps due to the lack of a single concise definition. 
In this light, one contribution of the current paper is a parsimonious definition partially generalizing a physically important \cite{Johnson_Burden_Koditschek_2016} class of hybrid systems (Def.~\ref{def:HybridSystem}), which we hope may be more inviting to the mathematically inclined reader.

Norton \cite{norton1995fundamental} emphasized that the Fundamental Theorem of Dynamical Systems is not the end of the theory but the beginning.  Just as the Fundamental Theorems of Arithmetic, Algebra, and Calculus provide the most basic tools of their respective fields, the Fundamental Theorem of Dynamical Systems indicates that the coarsest building blocks of dynamical systems are the countable components of the steady state (the ``chain-recurrent'') set and their basins (adding in the components of the ``gradient-like''  sets that lead to them).
The primary results of this paper show that these same building blocks fit together in the same way to describe a broad, physically important class of hybrid dynamical systems.

\subsection*{Acknowledgments}
This work is supported in part by the Army Research Office (ARO) under the SLICE Multidisciplinary University Research Initiatives (MURI) Program, award W911NF1810327; in part by UATL 10601110D8Z, a LUCI Fellowship granted by the Basic Research Office of the US Undersecretary of Defense for Research and Engineering; and in part by ONR grant N00014-16-1-2817, a Vannevar Bush Fellowship held by the last author, sponsored by the Basic Research Office of the Assistant Secretary of Defense for Research and Engineering.  
The authors also gratefully acknowledge helpful conversations with Y. Baryshnikov, S. A. Burden, Z. Cooperband, J. Culbertson, D. Guralnik, A. M. Johnson, E. Lerman, and P. F. Stiller.
We owe special gratitude to Culbertson for his careful reading of the manuscript; his generosity has spared the reader several ambiguities and typographical errors.
We thank the two anonymous referees for useful suggestions.

\bibliographystyle{amsalpha}
\bibliography{ref}

\clearpage
\appendix
\section*{Supplementary Materials (SM)}

\section{Relationship with selected prior work} \label{app:relate}

    \subsection{Relationship of Definition~\ref{def:HybridSystem} to \cite{Johnson_Burden_Koditschek_2016}}\label{app:JBK-relate}

	Our definition of \THS~ strictly generalizes \cite[Def.~2]{Johnson_Burden_Koditschek_2016}, modulo our added regularizing assumption requiring that the union of guard sets be closed.
	Since any disjoint union of smooth (paracompact) manifolds with corners is metrizable,
         our definition of \MHS~ similarly strictly generalizes  \cite[Def.~2]{Johnson_Burden_Koditschek_2016}, modulo the closed guard assumption and the choice of a compatible extended metric on state space.
	However, we note that our Theorems~\ref{th:conley-decomp} and \ref{th:hybrid-conley} impose two additional conditions on \MHS~ that are not assumed in \cite{Johnson_Burden_Koditschek_2016}: our theorems require that (i) state space is compact (Assumption~\ref{assump:compact}), and (ii) the trapping guard condition is satisfied (Assumption~\ref{assump:trapping-guard}).
	(We also add that \cite{Johnson_Burden_Koditschek_2016} consider continuations of Zeno executions past the stop time, while we do not; cf. Remark~\ref{rem:def-2-big-remark}.)
	Regarding (i) we note that, as discussed in Remark~\ref{rem:conley-rel-compact-metric-independent}, the specific choice of compatible extended metric is immaterial for the majority of our purposes since Theorems~\ref{th:conley-decomp} and \ref{th:hybrid-conley} require that state space is compact.\footnote{For the interested reader, we briefly mention that specific metrizations of hybrid systems are discussed in \cite{burden2015metrization}.
    However, we caution that, e.g., the pseudometric defined in \cite[Sec.~III.A]{burden2015metrization} is not generally an extended metric compatible with the topology on state space (under certain assumptions it defines a metric on a certain quotient of state space, the \emph{hybrifold} discussed in \S \ref{app:hybrifold}), so it is not generally an admissible extended metric making a metrizable \THS~ into an \MHS.}
	
	Regarding (ii), the hybrid systems model of \cite{Johnson_Burden_Koditschek_2016} allows for the possibility that no hybrid suspension semiflow (Def.~\ref{def:hybrid-suspension}) exists which is continuous-in-state, thereby precluding the trapping guard condition as shown in Appendix~\ref{app:suspension-semiflow-continuous-implies-trapping-guard-condition} which, in turn, may compromise the necessity of a Conley decomposition and Lyapunov function (e.g., see Ex. \ref{ex:omega-chain-pathologies} and \ref{ex:counterexample} for one view of the gap between the sufficiency and the necessity of this condition).
	For other classes of physical models, continuity can fail for different reasons.
	While our Def.~\ref{def:HybridSystem} and \cite[Def.~2]{Johnson_Burden_Koditschek_2016} require reset maps to be continuous, parsimonious hybrid models of certain physical systems may fail even to have this property (though in many applications it might be acceptable to insure continuity---e.g., one might smooth down the model of an exterior wall's outer corner so as to insure that balls bounce off it in a continuous manner). 
	However, the discontinuities of behavior allowed by the \cite{Johnson_Burden_Koditschek_2016} and other hybrid systems models may play a key role in other problem settings, such as legged leaping as explored in \cite[Fig.~7, Sec.~III.b]{brill_etal_2015}, \cite{johnson_kod_2013}. Clearly, more theoretical work is needed to understand the prospect for achieving Conley-style results in these settings, while, at the same time, more empirical work is needed to understand how the phenomena of interest should be formally represented and intuitively understood in their absence.
    
\subsection{Relationship of Definition~\ref{def:HybridSystem} to \cite{ames2005homology}}    
Our definition of \THS~ is particularly similar to the definition of ``classical hybrid system'' in \cite[p.~92]{ames2005homology}.
However, there are some differences.
First, we ignore any underlying graph structure of the hybrid system, although Remark~\ref{rem:wlog-disjoint-union} explains that this is immaterial.
Second, the definition in \cite[p.~92]{ames2005homology} amounts, using our notation, to requiring a \emph{flow} $\Phi$ be defined on $\st$; in contrast, we only require a \emph{semi}flow be defined on $\fs\subseteq \st$.
Finally, we impose the regularizing requirement that the guard $\gd\subseteq \st$ be closed; this requirement is not made in \cite[p.~92]{ames2005homology}.

\subsection{Relationship of Definition~\ref{def:eps-T-chains} to \cite{culbertson2019formal}}\label{app:CGKS-chains}
Our definition of $(\epsilon,T)$-chains (Def.~\ref{def:eps-T-chains}) is closest to that of \cite[Def.~2.18]{culbertson2019formal}.  
While our presentations differ, the only mathematical difference is our requirement that an $(\epsilon,T)$-chain contain at least two arcs.
If this were not the case, then the Conley relation (as defined in Def.~\ref{def:hybrid-conley-relation}) would be reflexive, which would imply that every point is chain recurrent.
It is clear that every $(\epsilon,T)$-chain in our sense is also an $(\epsilon,T)$-chain in the sense of \cite[Def.~2.18]{culbertson2019formal}, but not vice versa.

\subsection{Relationship of the relaxed hybrid system and hybrid suspension to prior work}\label{app:hybrid-suspension}

\subsubsection{Generalized hybrifolds}\label{app:hybrifold}
The \concept{hybrifold} of a hybrid system was introduced in \cite{simic2000towards, simic2005towards} for a class of hybrid systems satisfying various smoothness assumptions: e.g., state space is required to be a disjoint union of manifolds with ``piecewise-smooth boundary,'' and reset maps are required to be diffeomorphisms onto their images.
Our classes \THS~ and \MHS~ of hybrid systems do not assume any such smoothness nor injectivity properties, but we can still give a definition analogous to that of the hybrifold in our setting.
We will refer to this analogous, but (formally) more general, construction as the \emph{generalized hybrifold}.\footnote{The terminology hybri\emph{fold} is unfortunately no longer appropriate since no  manifolds are involved. As one possible alternative, this generalization has also been referred to as a ``colimit'' in \cite[p.~94]{ames2005homology}.} 

Let $H\coloneqq (\st,\fs,\gd,\varphi,r)$ be a \THS.
Using the notation ``$M_H$'' of \cite{simic2000towards,simic2005towards}, we define the \concept{generalized hybrifold} $M_H$ of $H$ to be the topological space obtained by gluing points $z\in \gd\subseteq \st$ to $\st$ along the reset $r$:
$$M_H\coloneqq \st/(z\sim r(z)).$$
\emph{Assuming} there exists a unique semiflow $\Psi_{M_H}$ on $M_H$ such that the quotient $\pi_{M_H}\colon \st\to M_H$ sends $H$ executions into $\Psi_{M_H}$ trajectories while preserving time, we refer to $\Psi_{M_H}$ as the \concept{generalized hybrifold semiflow}.
Using Lemma~\ref{lem:extend} and the universal property of the quotient topology \cite[Thm~3.70]{lee2010introduction}, it can be shown (by mimicking the proof of Prop.~\ref{prop:hybrid-suspension}) that $\Psi_{M_H}$ is a well-defined and continuous semiflow if, e.g., Assumptions \ref{assump:deterministic}, \ref{assump:inf-or-Zeno}, and \ref{assump:trapping-guard} are satisfied and if there are no Zeno executions.\footnote{We emphasize that these conditions---in particular, the trapping guard condition---are only \emph{sufficient} to ensure that a well-defined and continuous generalized hybrifold semiflow exists.
The simple example of a \THS~ $H = (\st,\fs,\gd,\varphi,r)$ with $\st = [0,1]$, $\fs = (0,1]$, $\gd = \{0\}$, $r(0) = 1$, and $\varphi$ generated by the vector field $-x(1-x)\frac{\partial}{\partial x}$ shows that the trapping guard is \emph{not necessary} for the generalized hybrifold semiflow $\Psi_{M_H}$ to be well-defined and continuous.
The reader may wish to contrast this with the converse statement of Cor.~\ref{co:suspension-semiflow-continuous-implies-trapping-guard-condition} in Appendix~\ref{app:suspension-semiflow-continuous-implies-trapping-guard-condition} for the hybrid suspension semiflow $\Phi_H$.}  
A cartoon depiction of a generalized hybrifold (including a trajectory of $\Psi_{M_H}$) is shown in the bottom-left panel of Fig.~\ref{fig:hybrid-suspension-appendix}.

In certain situations we can prove that $M_H$ (with the quotient topology) is metrizable.
In particular, Lemma~\ref{lem:closed-quotient-metrizable} implies that, if $\st$ is metrizable and the reset $r$ is a closed map with compact fibers satisfying $r(\gd) \cap \gd = \varnothing$, then $M_H$ is metrizable.
In particular, if $\st$ is compact and metrizable and $r(\gd)\cap \gd = \varnothing$, then $M_H$ is metrizable.

However, at least if $r(\gd) \cap \gd \neq \varnothing$, the generalized hybrifold $M_H$ of a compact \MHS~ cannot be used to prove Theorems~\ref{th:conley-decomp} and \ref{th:hybrid-conley} for multiple reasons even if $M_H$ happened to be metrizable.
For example (as pointed out in Remarks~\ref{rem:omega-compatibility-for-hybrifold-fails} and \ref{rem:chain-compatibility-for-hybrifold-fails}), $\omega$-limit sets, attracting-repelling pairs, and chain recurrence for $(M_H,\Psi_H)$ are not generally compatible with the corresponding notions for $H$ if $r(\gd) \cap \gd \neq \varnothing$.

Furthermore, even \emph{if} these compatibility issues were not present for a \emph{specific} \MHS~ $H$, a complete Lyapunov function for $\Psi_{M_H}$ will not generally pull back to a complete Lyapunov function for $H$.
More explicitly, if $V\colon M_H\to \R$ is a complete Lyapunov function for $\Psi_{M_H}$, then the function $L\coloneqq V\circ \pi_{M_H}$ will not generally be a complete Lyapunov function for $H$, because it will not satisfy the second condition of Def.~\ref{def:lyapunov} ($L$ will not decrease across resets).
Thus, the technique used in the proof of Theorem~\ref{th:hybrid-conley} would still fail if the hybrifold was used instead of the hybrid suspension.

\begin{figure}
	\centering
	\def\svgwidth{1.0\columnwidth}
	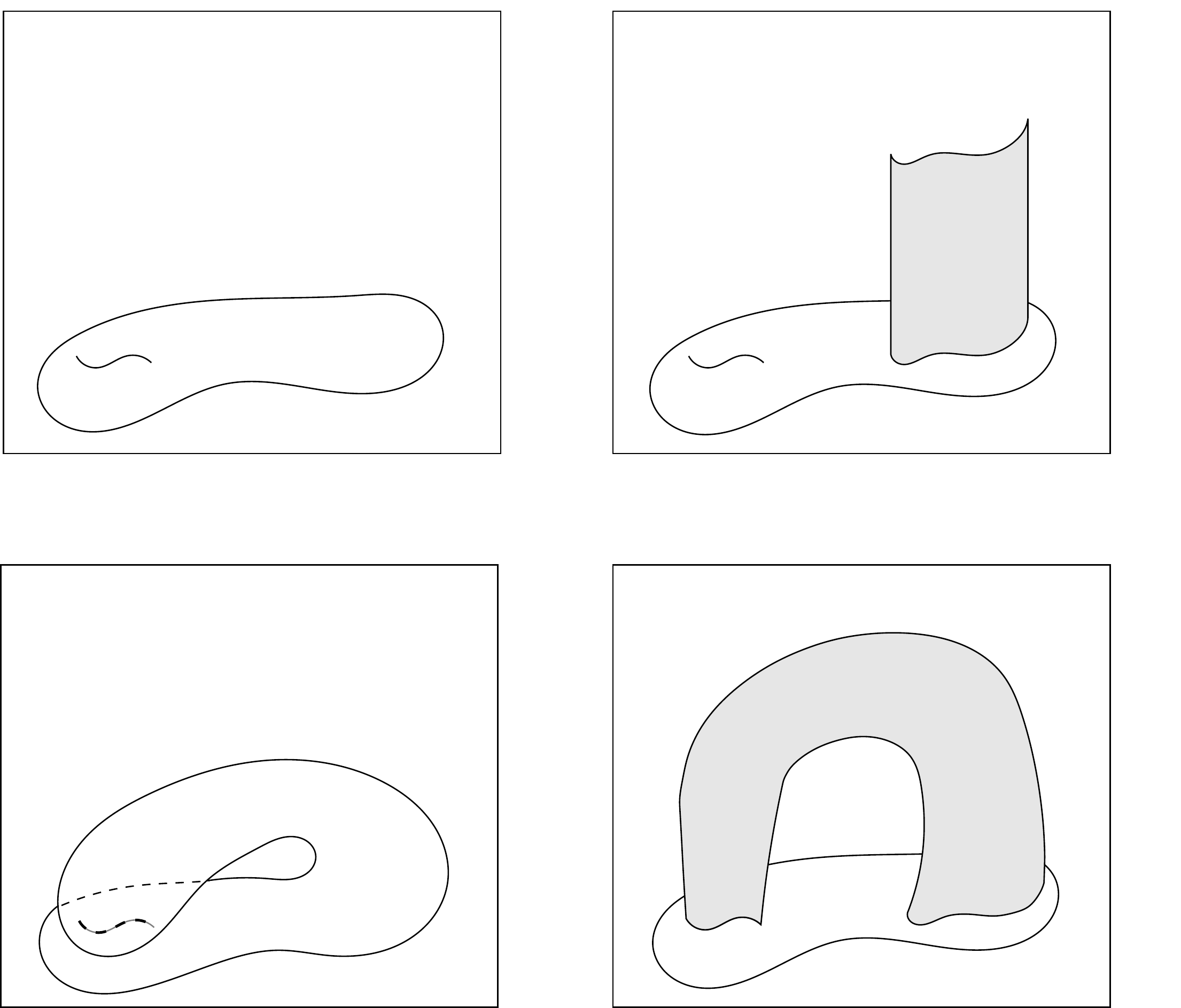
	\caption{Comparison of the constructions from \S \ref{sec:hybrid-suspension} depicted in Fig.~\ref{fig:hybrid-suspension} with the generalized hybrifold $M_H$ of $H$ discussed in \S \ref{app:hybrifold}. 
		Top left: a \THS~ $H = (\st,\fs,\gd,\varphi,r)$. Top right: its relaxed version $H' = (\st',\fs',\gd',\varphi',r')$; see Def. \ref{def:relaxed}. 
		Bottom right: hybrid suspension $(\Sus_H, \Phi_H)$ of $H$; see Def. \ref{def:hybrid-suspension}.
		Bottom left: the generalized hybrifold $M_H$ of $H$; $M_H$ is formed by gluing $\gd$ directly to $r(\gd)$ along $r$, without first embedding $\st$ in a larger space (unlike $\Sus_H = M_{H'}$).
		We mention $M_H$ only for purposes of comparison, i.e., \emph{we do not use $M_H$ in this paper}.  
		We remark that $\Sus_H$ coincides with the generalized hybrifold $M_{H'}$ of $H'$ (but \emph{not} with $M_H$).
		Additionally, we emphasize that $(\Sus_H, \Phi_H)$ strictly generalizes the classical suspension of a discrete-time dynamical system discussed in Appendix~\ref{app:suspension}; indeed, if $\st = \gd$ (cf. Ex.~\ref{ex:special-case-discrete-time}) our construction reduces to the classical one.}\label{fig:hybrid-suspension-appendix}
\end{figure}

\subsubsection{Relaxed hybrid system}\label{app:relaxed}
As mentioned in \S \ref{sec:hybrid-suspension}, the relaxed hybrid system $H'$ (Def.~\ref{def:relaxed}) formalizes the idea of requiring that executions of the hybrid system $H$ ``wait'' one time unit after impacting the guard before resetting.
A cartoon depicting the relaxed system is shown in the top-right panel of Fig.~\ref{fig:hybrid-suspension-appendix}.
The relaxed system is essentially an example of a ``temporal relaxation'' in the sense of \cite{johansson1999regularization}, where it was used to regularize Zeno executions, although we give the definition for \THS~ which are (formally speaking) more general\footnote{\label{foot:semiflow}Engineers might be unimpressed by the apparently slight formal gain of generality. 
Applications generally present models with smooth manifolds carrying (at least piecewise) smooth vector fields. 
In contrast, classical Conley theory is rooted in the tools of topological dynamics whose framework we have thus found it natural to adopt here. 
Furthermore, we hope that the imperative to eliminate smoothness assumptions from the spaces carrying these dynamics may be intuitively apparent when considering the (pinched and creased non-manifold) \emph{topological spaces} that inevitably arise as depicted, for example, in the hybrid suspension $\Sus_H$ of Fig.~\ref{fig:hybrid-suspension}.
}
than the specific examples considered in \cite[Sec.~3--4]{johansson1999regularization} (e.g., the local semiflows for \THS~ are not assumed to be generated by vector fields and, furthermore, the state space of a \THS~ is a general topological space rather than any sort of manifold).  
While we recover this Zeno regularization in our setting (Remark~\ref{rem:relaxed-sys-deterministic-nonblocking}), our primary motivation for the relaxed system is to use it as an intermediate step in constructing the hybrid suspension (Def.~\ref{def:hybrid-suspension}), which has better properties than those of the generalized hybrifold discussed in \S \ref{app:hybrifold}.

\subsubsection{Prior work related to the hybrid suspension}\label{app:hybrid-suspension-subsubsec}
The technique we used to prove Theorems~\ref{th:conley-decomp} and \ref{th:hybrid-conley} involves showing that a \THS~ satisfying the trapping guard condition and certain other assumptions is, in a certain sense (Prop.s~\ref{prop:omega-limit-set-suspension}, \ref{prop:attracting-repelling-pairs-suspension}, and \ref{prop:hybrid-chain-equiv-vs-suspension-chain-equiv}), no different from a certain continuous-time (semi-)dynamical system.
As discussed in \S \ref{app:hybrifold} and Remarks~\ref{rem:omega-compatibility-for-hybrifold-fails} and \ref{rem:chain-compatibility-for-hybrifold-fails}, this continuous-time system is \emph{not} the generalized hybrifold (local) semiflow; it is the \emph{hybrid suspension} semiflow constructed in Def.~\ref{def:hybrid-suspension}.
We choose to use the terminology ``suspension'' because the hybrid suspension strictly generalizes the classical suspension \cite[p.~797,~pp.~21--22]{smale1967differentiable,brin2002introduction} of a discrete-time dynamical system; indeed, if $\st = \gd$ our construction reduces to the classical one (as explained in Appendix~\ref{app:suspension}). 
We give a brief primer on the classical suspension in Appendix~\ref{app:suspension}.

In writing this paper we have become aware that versions of the hybrid suspension have previously appeared in the literature under different names, although (to the best of our knowledge) only for classes of hybrid systems which are not as general as \THS.
As mentioned in \S \ref{sec:related-work}, the hybrid suspension $\Sus_H$ of a \THS~ $H$ could be called a ``1-relaxed hybrid quotient space'' in the terminology of \cite{burden2015metrization} or a ``homotopy colimit'' in the terminology of \cite{ames2005homology}.  

We finally note that, in the terminology introduced in \S \ref{app:hybrifold}, the hybrid suspension $\Sus_H$ of a \THS~ $H$ coincides with the generalized hybrifold $M_{H'}$ (where $H'$ is the relaxed hybrid system) but \emph{not} with the hybrifold $M_H$.  
See Fig.~\ref{fig:hybrid-suspension-appendix}.

\section{Classical suspension of a discrete-time dynamical system}\label{app:suspension}
The purpose of this appendix is to explain, in a self-contained way, (i) the classical suspension of a discrete-time dynamical system and (ii) how the hybrid suspension (Def.~\ref{def:hybrid-suspension}) strictly generalizes the classical notion.

The classical suspension is often considered  in the context of a $C^{r\geq 1}$ diffeomorphism of a $C^r$ manifold \cite[p.~797,~pp.~343--345,~p.~111,~p.~173]{smale1967differentiable,palis1977topological,palis1982geometric,robinson1999dynamical}.
However, relevant for us is the more general context of a discrete-time (semi-)dynamical system defined by a continuous map of a topological space; we now describe the classical suspension in this context, following roughly \cite[pp.~21--22]{brin2002introduction}. 

Let $X$ be a topological space, $f\colon X\to X$ be a continuous map defining a discrete-time (semi-)dynamical system $(n,x)\mapsto f^{\circ n}(x)$, and $c\colon X\to (0,\infty)$ be a continuous function bounded away from zero.
Consider the quotient space 
$$X_c\coloneqq \{(x,t)\in X\times [0,\infty)\colon 0\leq t \leq c(x)\}/\sim,$$
where $\sim$ is the equivalence relation $(x,c(x)) \sim (f(x),0)$.
$X_c$ is called the \concept{suspension with ceiling (or roof) function $c$}.
Letting $[(x,s)]\in X_c$ denote the equivalence class of $(x,s)$, the \concept{suspension semiflow (with ceiling function $c$)} is the semiflow $\phi_c\colon [0,\infty)\times X_c \to X_c$ given by $\phi^t([x,s]) = [(f^{\circ n}(x), s')]$, where $n\in \N$ and $s' \geq 0$ satisfy
\begin{equation*}
\sum_{i=0}^{n-1}c(f^{\circ i}(x)) + s' = t + s, \qquad 0\leq s' \leq c(f^{\circ n}(x)).
\end{equation*} 

It is common to simply take the ceiling function $c$ to be $c(x)\equiv 1$ \cite[p.~797,~pp.~343--345,~p.~111,~p.~173]{smale1967differentiable,robinson1999dynamical,palis1977topological,palis1982geometric,robinson1999dynamical}, and in this case it is common to simply refer to $(X_1, \phi_1)$ as ``\emph{the}'' suspension of (the discrete-time semi-dynamical system defined by) $f$.\footnote{We note that $X_1$ coincides with what topologists call the \emph{mapping torus} of $f$ \cite[p.~53,~p.~151]{hatcher2001algebraic} (but, confusingly, \emph{not} with what topologists call the \emph{suspension of a topological space} \cite[p.~8]{hatcher2001algebraic}).}
Our \emph{hybrid suspension} defined in Def.~\ref{def:hybrid-suspension} strictly generalizes ``the'' suspension of $f$, and this can be seen as follows.
Define a \THS~ $H = (\st,\fs,\gd,\varphi,r)$ by setting $\st = \gd = X$, $\fs = \varnothing$, $r = f$, and (viewed set-theoretically) $\varphi = \varnothing$ (cf. Ex.~\ref{ex:special-case-discrete-time}). 
Then the hybrid suspension $(\Sus_H, \Phi_H)$ coincides precisely with $(X_1,\phi_1)$.

\section{Continuous hybrid suspension semiflow implies the trapping guard condition}\label{app:suspension-semiflow-continuous-implies-trapping-guard-condition}
Let $H = (\st,\fs,\gd,\varphi,r)$ be a \THS~ satisfying Assumptions \ref{assump:deterministic},  \ref{assump:inf-or-Zeno}, and \ref{assump:trapping-guard}.
In \S\ref{sec:hybrid-suspension} (Def.~\ref{def:relaxed} and \ref{def:hybrid-suspension}) we defined
\begin{equation}\label{eq:recall-sus-relaxed-defs}
\begin{gathered}
\st' \coloneqq \frac{\st\sqcup (\gd\times [0,1])}{z\sim(z,0)} \qquad \underbrace{\pi_0\colon \st\sqcup(\gd\times [0,1])\to \st'}_{\textnormal{quotient map}} \qquad \underbrace{\iota\colon \st\to \st'}_{\pi_0|_\st}\\
\Sus_H \coloneqq \frac{\st'}{\pi_0(z,1)\sim \pi_0(r(z))} \qquad \underbrace{\pi\colon \st' \to \Sus_H}_{\textnormal{quotient map}}
\end{gathered}
\end{equation}
and the suspension semiflow $\Phi_H\colon [0,\infty)\times \Sus_H \to \Sus_H$, and we showed that $\Phi_H$ is continuous (Prop.~\ref{prop:hybrid-suspension}).
It is immediate from the definitions that $\Phi_H$ satisfies the following two properties.
\begin{enumerate}[label=\arabic*.\hspace{0.1cm}, ref=\arabic*,leftmargin=1.3cm]
			\item\label{enum:Phi_H-cond-1} $\Phi_H^t(\pi\circ \pi_0(z,s)) = \pi\circ\pi_0(z,t+s)$ for all $z\in \gd$, $s\in [0,1]$, and $t\in [0,1-s]$.
	\item\label{enum:Phi_H-cond-2} For all $(t,x)\in \dom(\varphi)$, $\pi\circ\iota(\varphi^t(x)) = \Phi_H^t(\pi\circ\iota(x))$.
\end{enumerate}

While for convenience of exposition we only defined the quantities in \eqref{eq:recall-sus-relaxed-defs} under Assumptions \ref{assump:deterministic},  \ref{assump:inf-or-Zeno}, and \ref{assump:trapping-guard} (in particular, under Assumption \ref{assump:trapping-guard}), the definitions in \eqref{eq:recall-sus-relaxed-defs} make sense verbatim for \emph{any} \THS.
Thus, for an arbitrary \THS~ $H$, it makes sense to ask the following question: under what circumstances does there exist a well-defined ``suspension semiflow'' $\Phi$ on $\Sus_H$ for $H$ in the sense that $\Phi$ satisfies Conditions \ref{enum:Phi_H-cond-1} and \ref{enum:Phi_H-cond-2} (stated above for $\Phi_H$)?

In this appendix we prove a result (Prop.~\ref{prop:suspension-semiflow-continuous-implies-trapping-guard-condition}) which implies (Cor.~\ref{co:suspension-semiflow-continuous-implies-trapping-guard-condition}) that, if $H = (\st,\fs,\gd,\varphi,r)$ is a \THS~ satisfying Assumptions \ref{assump:deterministic} and \ref{assump:inf-or-Zeno} with Hausdorff $\st$, then there exists a continuous suspension semiflow $\Phi\colon [0,\infty)\times \Sus_H\to \Sus_H$ in the above sense \emph{if and only if} $H$ satisfies the trapping guard condition (Assumption~\ref{assump:trapping-guard}).

We state Prop.~\ref{prop:suspension-semiflow-continuous-implies-trapping-guard-condition} after first establishing the following preliminary result.

\begin{Lem}\label{lem:pi-iota-homeo}
	Let $H=(\st,\fs,\gd,\varphi,r)$ be a \THS.
	Define $\st'$, $\Sus_H$, $\pi_0$, $\iota$, and $\pi$ as in \eqref{eq:recall-sus-relaxed-defs}.
	Then 
	$$\pi\circ \pi_0|_{\gd\times [0,\frac{1}{2}]}\colon \gd\times \Big[0,\frac{1}{2}\Big]\to \Sus_H \qquad \textnormal{ and } \qquad  \pi\circ \iota\colon \st \to \Sus_H$$ are homeomorphisms onto their images.
\end{Lem}
\begin{proof}
	We first show that $\pi|_{\iota(\st)\cup \pi_0(\gd\times [0,\frac{1}{2}])}$ is a closed map.
	Define $\gd'\coloneqq \pi_0(\gd\times \{1\})$ and $r'\colon \gd'\to \st'$ via $r'(\pi_0(z,1))\coloneqq \pi_0(r(z))$, and let $C\subseteq \iota(\st)\cup \pi_0(\gd\times [0,\frac{1}{2}])$ be an arbitrary closed set.
	We compute
	\begin{equation}\label{eq:lem-pi-iota-homeo-1}
	\pi^{-1}(\pi(C)) = C \cup (r')^{-1}(C) \cup r'(\underbrace{C \cap \gd'}_{=\varnothing})
	\end{equation}
	since $r'(\gd')\cap \gd' = \varnothing$.
	Since $r'$ is continuous, the right side of \eqref{eq:lem-pi-iota-homeo-1} is the union of three closed sets (the third is empty since $\iota(\st)\cup \pi_0(\gd\times [0,\frac{1}{2}])$ is disjoint from $\gd'$).
	By the definition of the quotient topology, it follows that $\pi(C)$ is closed in $\Sus_H$, so $\pi|_{\iota(\st)\cup \pi_0(\gd\times [0,\frac{1}{2}])}$ is indeed a closed map.
	
	Since $\pi|_{\iota(\st)\cup \pi_0(\gd\times [0,\frac{1}{2}])}$ is a closed map, $$\pi\circ \pi_0|_{\gd\times [0,\frac{1}{2}]} = \pi|_{\iota(\st)\cup \pi_0(\gd\times [0,\frac{1}{2}])}\circ \pi_0|_{\gd\times [0,\frac{1}{2}]} \qquad \textnormal{ and } \qquad  \pi\circ \iota = \pi|_{\iota(\st)\cup \pi_0(\gd\times [0,\frac{1}{2}])} \circ \iota$$  are closed maps by composition, since $\pi_0$ and  $\iota$ are closed maps.
	That $\pi_0$ is closed follows by repeating the proof of Lemma~\ref{lem:pi_0-closed} verbatim, and $\iota$ is closed since $\pi_0^{-1}(\iota(D)) = D \sqcup ((D\cap \gd)\times \{0\})$ is closed in $\st\sqcup (\gd\times [0,1])$ for any closed set $D\subseteq \st$.
	
	It is immediate from the definitions that both maps in the statement of the lemma are continuous and injective.
	Since we have shown that they are also closed, it follows that they are homeomorphisms onto their images  \cite[Ex.~2.29]{lee2010introduction}. 	
\end{proof}

\begin{Prop}\label{prop:suspension-semiflow-continuous-implies-trapping-guard-condition}
	Let $H=(\st,\fs,\gd,\varphi,r)$ be a \THS~ satisfying Assumption~\ref{assump:deterministic} with Hausdorff $\st$ and with $\st = \fs\cup \gd$.  		
	Define $\st'$, $\Sus_H$, $\pi_0$, $\iota$, and $\pi$ as in \eqref{eq:recall-sus-relaxed-defs}, and suppose there exists a continuous semiflow $\Phi\colon [0,\infty)\times \Sus_H\to \Sus_H$ satisfying the following conditions.
	\begin{enumerate}[label=\ref*{prop:suspension-semiflow-continuous-implies-trapping-guard-condition}.\arabic*.\hspace{0.1cm}, ref=\ref*{prop:suspension-semiflow-continuous-implies-trapping-guard-condition}.\arabic*,leftmargin=1.3cm]
		\item\label{enum:sus-cylinder} $\Phi^t(\pi\circ\pi_0(z,s)) = \pi\circ\pi_0(z,t+s)$ for all $z\in \gd$, $s\in [0,1]$, and $t\in [0,1-s]$.
		\item\label{enum:sus-conjugacy} For all $(t,x)\in \dom(\varphi)$, $\pi\circ\iota(\varphi^t(x)) = \Phi^t(\pi\circ\iota(x))$.
	\end{enumerate}
	Then $H$ satisfies the trapping guard condition (Assumption \ref{assump:trapping-guard}).
\end{Prop}

\begin{Co}\label{co:suspension-semiflow-continuous-implies-trapping-guard-condition}
Let $H=(\st,\fs,\gd,\varphi,r)$ be a \THS~ satisfying Assumptions \ref{assump:deterministic} and \ref{assump:inf-or-Zeno} with Hausdorff $\st$.
Define $\st'$, $\Sus_H$, $\pi_0$, $\iota$, and $\pi$ as in \eqref{eq:recall-sus-relaxed-defs}.
Then there exists a continuous ``suspension semiflow'' $\Phi\colon [0,\infty)\times \Sus_H\to \Sus_H$ for $H$---in the sense that $\Phi$ satisfies conditions \ref{enum:sus-cylinder} and \ref{enum:sus-conjugacy} of Prop.~\ref{prop:suspension-semiflow-continuous-implies-trapping-guard-condition}---if and only if $H$ satisfies the trapping guard condition (Assumption \ref{assump:trapping-guard}).	
\end{Co}
\begin{proof}[Proof of Cor.~\ref{co:suspension-semiflow-continuous-implies-trapping-guard-condition}]
If $H$ satisfies the trapping guard condition, then by Prop.~\ref{prop:hybrid-suspension} the map $\Phi_H$ of Def.~\ref{def:hybrid-suspension} is such a continuous semiflow.

Conversely, assume that a continuous semiflow $\Phi$ satisfying Conditions \ref{enum:sus-cylinder} and \ref{enum:sus-conjugacy} exists.
The assumption that all maximal executions are infinite or Zeno (Assumption \ref{assump:inf-or-Zeno}) implies that $\st = \fs\cup \gd$ (by Remark~\ref{rem:FcupZ}), so the hypotheses of  Prop.~\ref{prop:suspension-semiflow-continuous-implies-trapping-guard-condition} are satisfied.	
By the conclusion of Prop.~\ref{prop:suspension-semiflow-continuous-implies-trapping-guard-condition}, $H$ satisfies the trapping guard condition.
\end{proof}

\begin{proof}[Proof of Prop.~\ref{prop:suspension-semiflow-continuous-implies-trapping-guard-condition}]
		For purposes of readability, we define $f\coloneqq \pi\circ\iota\colon \st\to \pi\circ \iota(\st)$, $B\coloneqq \gd\times [0,\frac{1}{2}]$,  and $g\coloneqq (\pi\circ \pi_0)|_{B}\colon B\to \pi\circ \pi_0(B)$.
		By Lemma~\ref{lem:pi-iota-homeo}, $f$ and $g$ are homeomorphisms.
		
		Letting $\mu\colon \st\to [0,+\infty]$ be the maximum flow time defined in \eqref{eq:max-flow-time-all}, we first show that, for any $x\in \fs \cap \mu^{-1}([0,\infty))$,
		\begin{equation}\label{eq:prop-app-lim}
		\ell(x) \coloneqq \lim_{t\to \mu(x)^-}\varphi^t(x) \in \gd.
		\end{equation}
		That the limit $\ell(x)$ exists follows from continuity of $f$, $f^{-1}$, and $\Phi$ since Condition \ref{enum:sus-conjugacy} implies that $\varphi^t(x) = f^{-1}(\Phi^t(f(x)))$ for all $t\in \{t\mid (t,x)\in \dom(\varphi)\}$, and the properties of a local semiflow imply that $\{t\mid (t,x)\in \dom(\varphi)\} = [0,\mu(x))$ for any $x\in \fs$ \cite[Sec.~1.3]{hirsch2006monotone}.
		Furthermore, it cannot be the case that $\ell(x)\in \fs$, because the trajectory image $\varphi^{[0,\mu(x))}(x)$ would then have compact closure in $\fs$, and this in turn would imply that $\mu(x)$ is infinite \cite[Sec.~1.3]{hirsch2006monotone}, a contradiction. Since we have also assumed that $\st = \fs \cup \gd$, it follows that $\ell(x) \in \gd$.
		
		Next, define $\widetilde{U}\coloneqq \Phi^{-\frac{1}{2}}(g(B))$, $U\coloneqq f^{-1}(\widetilde{U})$, and the continuous maps $h\colon g(B)\to [0,\frac{1}{2}]$ and $\nu\colon U\to [0,\frac{1}{2}]$ via $h(g(z,t))\coloneqq t$ and $\nu \coloneqq \frac{1}{2} - h\circ \Phi^{\frac{1}{2}}\circ f|_U$.
		By the definition of $\nu$ and Condition~\ref{enum:sus-cylinder} it follows that $\nu^{-1}(0)=\gd$ and $\Phi^{\nu(x)}(f(x))\in f(\gd)$ for all $x\in U$.
		We will now show that $\mu|_U = \nu$.
		Since $\varphi$ is $\fs$-valued but $\Phi^{\nu(x)}(f(x)) \in f(\gd)$, it follows from Condition \ref{enum:sus-conjugacy} and the fact that $\fs\cap \gd = \varnothing$ (since $H$ is deterministic by Assumption \ref{assump:deterministic}) that $(\nu(x), x) \not \in \dom(\varphi)$ for any $x\in U$.
		Since $\{t\mid (t,x)\in \dom(\varphi)\} = [0,\mu(x))$ for any $x\in \fs$ \cite[Sec.~1.3]{hirsch2006monotone}, it follows that $\mu|_{U\cap \fs} \leq \nu|_{U\cap \fs}$, and therefore $\mu|_U \leq \nu$ since $\mu|_\gd = \nu|_\gd = 0$.
		We now show the reverse inequality. 
		It follows from \ref{enum:sus-cylinder} that, if $q\in f(\gd)$, then $\Phi^t(q)\not \in f(\gd)$ for all $t\in (0,1)$.
		Since $\nu\leq \frac{1}{2}$ and $\Phi^{\nu(x)}(f(x))\in \gd$ for all $x\in U$, it therefore follows that $\Phi^t(f(x)) \not \in f(\gd)$ for all $t\in [0,\nu(x))$. 
		Therefore, \ref{enum:sus-conjugacy} implies that $\lim_{s\to t^-}\varphi^s(x) = f^{-1}(\Phi^{t}(f(x)))\not \in \gd$ for any $x\in U$ and $t\in [0,\nu(x))$.
		From this and \eqref{eq:prop-app-lim} it follows that $\mu|_U \geq \nu$.
		Since we have already shown that $\mu|_U \leq \nu$, this establishes that $\mu|_U = \nu$. 
		
		Since $\widetilde{U} \cap f(\st) = \Phi^{-\frac{1}{2}}(g(B))\cap f(\st) = \Phi^{-\frac{1}{2}}(g(\gd\times [0,1])) \cap f(\st)$ is a neighborhood of $f(\gd)$ in $f(\st)$, $U= f^{-1}(\widetilde{U})$ is a neighborhood of $\gd$ in $\st$.
		We now define $\widehat{\varphi}\colon \cl(\dom(\varphi))\cap ([0,\infty) \times U)\to \st$ and $\rho\colon U\to \gd$ via
\begin{equation}\label{eq:prop-app-hat-varphi-rho-def}
\begin{split}
\widehat{\varphi}^t(x)\coloneqq f^{-1}\circ \Phi^t\circ f(x),\qquad \qquad
\rho(x)\coloneqq \widehat{\varphi}^{\nu(x)}(x)
\end{split}.
\end{equation}
Condition~\ref{enum:sus-conjugacy} implies that $\widehat{\varphi}$ is a continuous extension of $\varphi|_{\dom(\varphi)\cap ([0,\infty) \times U)}$ which satisfies Equation~\eqref{eq:semiflow-local-extension} of Def.~\ref{def:trapping-guards} since $\nu = \mu|_U$, and this extension is automatically unique since $\st$ is Hausdorff.
The map $\rho$ is a continuous retraction by \eqref{eq:prop-app-hat-varphi-rho-def} and the fact that $\Phi^{\nu(x)}(f(x)) \in f(\gd)$ for all $x\in U$ (as noted in the previous paragraph).
Since $\mu|_U = \nu$ is continuous, it follows that all conditions of Def.~\ref{def:trapping-guards} are satisfied.
Hence $H$ satisfies the trapping guard condition. 	    
\end{proof}

\section{Proofs of Lemmas \ref{lem:extend}, \ref{lem:niceify}, \ref{lem:relaxed-conley-relation}, and \ref{lem:colonius-result-generalized}}\label{app:technical}
In this appendix we prove Lemmas \ref{lem:extend}, \ref{lem:niceify}, \ref{lem:relaxed-conley-relation}, and \ref{lem:colonius-result-generalized}; we also restate these lemmas for convenience.

\Extend*
\begin{proof}
	We first show that $\mu$ is continuous.
	Letting $U\supseteq \gd$ be the domain of a flow-induced retraction, Assumption \ref{assump:trapping-guard} and the definition of the trapping guard condition (Def. \ref{def:trapping-guards}) imply that $\mu|_U$ is continuous.
	Since there is an infinite or Zeno execution starting at every $x\in \st$ (by Assumption \ref{assump:inf-or-Zeno}), $\mu^{-1}([0,\infty)) =  U\cup \bigcup_{t\geq 0}(\varphi^t)^{-1}(U)$ is a union of open subsets of $\st$.\footnote{This follows since the domain of $\varphi^t$ is open in $\fs$ (since $\dom(\varphi)$ is open in $[0,\infty) \times \fs$, by the definition of local semiflow), and $\fs$ is open in $\st$ by Def.~\ref{def:HybridSystem}.}
	Since the restrictions $\mu|_U$ and $\mu|_{(\varphi^t)^{-1}(U)} = t + \mu|_U\circ \varphi^t|_{(\varphi^t)^{-1}(U)}$ are all continuous, $\mu$ is continuous on $\mu^{-1}([0,\infty))$.
	Since $\dom(\varphi)$ is open in\footnote{This follows since $\dom(\varphi)$ is open in $[0,\infty) \times \fs$, and $[0,\infty) \times \fs$ is open in $[0,\infty) \times \st$ (since $\fs$ is open in $\st$, by Def.~\ref{def:HybridSystem}).} $[0,\infty)\times \st$ it follows that, for every $x\in \mu^{-1}(+\infty)$ and $T > 0$, there exists a neighborhood $V\ni x$ with $\mu(V)\subseteq [T,+\infty]$.
	Hence $\mu$ is also continuous at every point of $\mu^{-1}(+\infty)$, so $\mu\colon \st\to [0,+\infty]$ is continuous. 
	
	We now show that $\cl(\dom(\varphi))$ is given by \eqref{eq:lem-ext-closure-expression}.
	Clearly $\dom(\varphi)$ is contained in the sets on both sides of \eqref{eq:lem-ext-closure-expression}.
	If $(t,x)\not \in \dom(\varphi)$ belongs to the set on the right of \eqref{eq:lem-ext-closure-expression}, then $t = \mu(x)$ since $\mu|_{\st\setminus \fs}\equiv 0$ and the properties of a local semiflow imply that $\{t\mid (t,x)\in \dom(\varphi)\} = [0,\mu(x))$ for all $x\in \fs$ \cite[Sec.~1.3]{hirsch2006monotone}.
	If $\mu(x) = 0$, then $(t,x)\in \{0\}\times \cl(\fs) \subseteq \cl(\dom(\varphi))$ since $\{0\}\times \fs \subseteq \dom(\varphi)$.
	If $\mu(x) > 0$, then $x\in \fs$ and $(\mu(x),x) \in \cl([0,\mu(x))) \times \{x\} \subseteq \cl(\dom(\varphi))$ since $[0,\mu(x))\times \{x\} \subseteq \dom(\varphi)$.
	Hence the set on the right of \eqref{eq:lem-ext-closure-expression} is contained in the set on the left.
	On the other hand, if $(t,x)$ does not belong to the set on the right of \eqref{eq:lem-ext-closure-expression}, then either (i) $t > \mu(x)$ or (ii) $x\not \in \cl(\fs)$.
	In case (i), continuity of $\mu$ implies that there are neighborhoods $V\ni x$ and $J \ni t$ such that $s > \mu(y)$ for all $(s,y)\in J\times V$, so $(J\times V)\cap \dom(\varphi) = \varnothing$, and therefore $(t,x)\not \in \cl(\dom(\varphi))$.
	In case (ii), $[0,\infty)\times (\st\setminus \cl(\fs))$ is a neighborhood of $(t,x)$ disjoint from $\dom(\varphi)$, so again $(t,x)\not \in \cl(\dom(\varphi))$.
	Hence the set on the left of \eqref{eq:lem-ext-closure-expression} is also contained in the set on the right.

	From Remark~\ref{rem:FcupZ} we have $\st = \fs \cup \gd$, and this implies that $\mu^{-1}(0) = \gd$.
	Let $\widehat{\varphi}$ be the unique continuous extension of $\varphi|_{\dom(\varphi)\cap ([0,\infty) \times U)}$ to $\cl(\dom(\varphi))\cap ([0,\infty) \times U)$ ensured by Def.~\ref{def:trapping-guards}.	
	We now define $\widetilde{\varphi}\colon \cl(\dom(\varphi))\to \st$ via
	\begin{equation}\label{eq:tilde-varphi-def}
	\widetilde{\varphi}^t(x) = 
	\begin{cases} 
	\varphi^t(x), & (t,x)\in \dom(\varphi),  \\ 
	\widehat{\varphi}^{s}(\varphi^{t-s}(x)), & (t-s,x)\in \varphi^{-1}(U)\\
	\widehat{\varphi}^{t}(x), & x\in U, \, (t,x)\in \cl(\dom(\varphi))
	\end{cases},
	\end{equation}
	where $s \in [0, t]$ ranges over all admissible values.
	It is clear that $\widetilde{\varphi}$ is well-defined by the definition of $\widehat{\varphi}$ and the fact that $\varphi$ satisfies the properties of a local semiflow.
	Since $\widetilde{\varphi}$ is defined by a family of continuous functions defined on open subsets of $\cl(\dom(\varphi))$, it follows that $\widetilde{\varphi}$ is continuous, so $\widetilde{\varphi}$ is indeed a continuous extension of $\varphi$ to $\cl(\dom(\varphi))$.
	Uniqueness of $\widetilde{\varphi}$ follows from uniqueness of $\widehat{\varphi}$ and the local semiflow properties of $\varphi$.
	
	If $x\in \cl(\fs) \cap \gd$, then $\widetilde{\varphi}^{\mu(x)}(x) = x \in \gd$ since $\mu^{-1}(0) = \gd$.
	If instead $x\in \fs\cap \mu^{-1}([0,\infty))$, then $y \coloneqq \varphi^{\mu(x)-s}(x) \in U$ for some $s \in [0,\mu(x))$, so \eqref{eq:tilde-varphi-def} and Def.~\ref{def:trapping-guards} imply that  $\widetilde{\varphi}^{\mu(x)}(x) =  \widehat{\varphi}^{s}(\varphi^{\mu(x)-s}(x)) = \widehat{\varphi}^{\mu(y)}(y) \in \gd$  since $s = \mu(y)$.
	
	It remains only to verify the claimed properties (i--iii). 
	(i) is immediate from \eqref{eq:tilde-varphi-def}, the definition of $\widehat{\varphi}$, and the fact that $\varphi^0 = \id_\fs$.
	To prove (ii) first notice that, since $\widetilde{\varphi}|_{\dom(\varphi)} = \varphi$, the analogous property satisfied by $\dom(\varphi)$ is equivalent to 
	\begin{equation}\label{eq:mu-expr-1}
	\mu(x) = s + \mu(\widetilde{\varphi}^s(x))
	\end{equation}
	for all $x\in \fs$ and $s\in [0,\mu(x))$.
	Taking the limit as $s\to \mu(x)$ and using continuity of $\mu$ implies that \eqref{eq:mu-expr-1} also holds for $x\in \fs$ and $s \in [0,\mu(x)]$. 
	On the other hand, \eqref{eq:mu-expr-1} trivially holds for all $x\in \cl(\fs) \setminus \fs$ and $s\in [0,\mu(x)]$ since then $\mu(x) = 0$ and $\widetilde{\varphi}^0(x) = x$.
	Hence \eqref{eq:mu-expr-1} holds for all $x\in \cl(\fs)$ and $s\in [0,\mu(x)]$, and this is equivalent to the claimed property (ii).
	Finally, the property (iii) is trivially verified for $x\in \cl(\fs)\setminus \fs$, and is  verified for $x\in \fs$ by taking sequences $t_n\nearrow t, s_n \nearrow s$, using continuity of $\widetilde{\varphi}$, and using the analogous property satisfied by $\varphi$. 
\end{proof}

\Nicefy*
\begin{proof}
	It follows immediately from the definitions that $\widehat{\conley_H} \subseteq \conley_H$. 
	
	For the reverse inclusion, suppose that $(x,y) \in \conley_H$.
	Fix $\epsilon, T > 0$, and let $U$ be a retraction domain (by Assumption~\ref{assump:trapping-guard} and Def. \ref{def:trapping-guards}) for $\gd$ with flow-induced retraction $\rho : U \to \gd$.
	Shrinking $U$ if necessary, by Assumption~\ref{assump:compact} we may assume that $U$ is compact and that $\mu|_U$ is strictly bounded above by $T$, where the maximum flow time $\mu\colon \st\to [0,+\infty]$ is defined in \eqref{eq:max-flow-time-all}.
	By the uniform continuity of $r \circ \rho$, there exists $0 < \delta < \epsilon$ such  that $\dist(r\circ \rho(p), r\circ \rho(p')) < \epsilon/2$ whenever $\dist(p,p') < \delta$.  If $\st \setminus U$ is nonempty, we may further assume that $\delta < \dist(\gd, \st \setminus U)$.
	By Assumptions \ref{assump:deterministic}, \ref{assump:inf-or-Zeno}, and \ref{assump:trapping-guard}, Lemma \ref{lem:extend} yields a continuous extension $\widetilde{\varphi}$ of $\varphi$ defined on the closure $\cl(\dom(\varphi))$ of $\dom(\varphi)$ in $[0, \infty) \times \st$. By compactness of $\st$ (Assumption~\ref{assump:compact}), it follows that the restriction of $\widetilde{\varphi}$ to $\cl(\dom(\varphi)) \cap \left([0,2T]\times \st\right)$ is uniformly continuous.  
	Pick $0< \beta < \epsilon/2$  such that $\dist(\widetilde{\varphi}^{t}(p), \widetilde{\varphi}^{t'}(p')) < \delta$ whenever $\dist(p,p') < \beta$,  $|t-t'| < \beta$, and $t, t' \in [0, 2T]$.    Let $\chi = (N, \tau, \eta, \gamma) \in \conley_H^{\beta,2T}(x,y)$ be a $(\beta, 2T)$-chain from $x$ to $y$.  Without loss of generality, we may assume that $\tau_{N+1} - \tau_N \le 2T$ by adding a trivial continuous-time jump at $\max(\tau_N + 2T, \tau_{N+1} - 2T)$ if necessary.
	
	We would like to modify $\chi$ to get a nice $(\epsilon,T)$-chain.   For each $k$ such that $0 < \eta_k < N $ and $\tau_{\eta_k} - \tau_{(\eta_k - 1)} < T$, we will remove the continuous-time jump at $\tau_{\eta_k}$, continue on the execution prior to that jump, and return to a later point on the image of $\chi$ via a new jump.  Let $u = \gamma_{(\eta_k - 1)}(\tau_{\eta_k})$ and $v = \gamma_{\eta_{k}}(\tau_{\eta_k})$. Let $t = \tau_{(\eta_k + 1)} - \tau_{\eta_k}$. One of two cases occurs: either (i)  $t < 2T$ and $\tau_{(\eta_{k} + 1)}$ is a reset jump, or (ii) $t \ge 2T$.   In the former case, we have two subcases based on whether $\widetilde{\varphi}^t(u)$ is either (a) defined or (b) undefined.  
	
	\textbf{Case (i)(a):}  By the uniform continuity of the restriction of $\widetilde{\varphi}$ discussed above, we have $\dist(\widetilde{\varphi}^t(u), \widetilde{\varphi}^t(v)) < \delta$ since $\dist(u,v) < \beta$. 
	Since $\tau_{(\eta_k + 1)}$ is a reset jump, it follows that $\widetilde{\varphi}^t(v) \in \gd$.
	Since $\widetilde{\varphi}^t(u)$ is defined, it follows that $\widetilde{\varphi}^t(u) \in U$ since $\delta < \dist(\gd,\st\setminus U)$.
	Thus, $\dist(r\circ \rho(\widetilde{\varphi}^t(u)), r(\widetilde{\varphi}^t(v))) < \epsilon/2$.  By the definition of $\chi$, we have $\dist(r(\widetilde{\varphi}^t(v)), \gamma_{(\eta_k + 1)}(\tau_{\eta_k +1})) < \beta < \epsilon/2$.   Thus, by the triangle inequality we have
	$$\dist(r\circ \rho(\widetilde{\varphi}^t(u)),\gamma_{(\eta_k + 1)}(\tau_{\eta_k +1})) < \epsilon,$$
	so we can replace the jump at $\tau_{\eta_k}$ by a modified jump at $\tau_{(\eta_k+1)}$ by extending the domain of $\gamma_{(\eta_{k}-1)}$ to $[\tau_{(\eta_{k}-1)},\tau_{(\eta_{k}+1)}]$, i.e., by replacing $\gamma_{(\eta_{k}-1)}$ with $\left( [\tau_{(\eta_{k}-1)},\tau_{(\eta_{k}+1)}]\ni s \mapsto \widetilde{\varphi}^{s-\tau_{(\eta_{k}-1)}}(\gamma_{(\eta_{k}-1)}(\tau_{(\eta_{k}-1)}))\right)$.
	We obtain a modified chain after deleting $\tau_{\eta_k}$ from the sequence $(\tau_j)_{j=0}^N$, deleting $\eta_k$ from the sequence $(\eta_j)_{j=0}^M$, and reindexing the sequences accordingly. 
	
	\textbf{Case (i)(b):} Since there exists a Zeno or infinite execution starting at $u$ and since $\gd$ is closed, there exists a unique ``first impact time'' $t_0 < t < 2T$ such that $\widetilde{\varphi}^{t_0}(u) \in \gd$.
	By our uniform continuity considerations, we have $\dist(\widetilde{\varphi}^{t_0}(u), \widetilde{\varphi}^{t_0}(v)) < \delta$.
	Since $\delta < \dist(\gd,\st\setminus U)$, we have $\widetilde{\varphi}^{t_0}(v) \in U$. Thus, $$\dist(r(\widetilde{\varphi}^{t_0}(u)), r(\gamma_{\eta_k}(\tau_{(\eta_k+1)}))) = \dist(r(\widetilde{\varphi}^{t_0}(u)), r\circ\rho(\widetilde{\varphi}^{t_0}(v))) < \epsilon/2.$$  Moreover, $\dist(r(\gamma_{\eta_k}(\tau_{(\eta_k+1)})), \gamma_{(\eta_k + 1)}(\tau_{\eta_k +1})) < \beta < \epsilon/2$.   Thus, by the triangle inequality we have
	
	$$\dist(r(\widetilde{\varphi}^{t_0}(u)), \gamma_{(\eta_k + 1)}(\tau_{\eta_k +1})) < \epsilon,$$
	so we can replace the jump at $\tau_{\eta_k}$ by a jump at $t_0 + \tau_{\eta_k}$ by extending the domain of $\gamma_{(\eta_{k}-1)}$ to $[\tau_{(\eta_{k}-1)},t_0+\tau_{\eta_{k}}]$, i.e., by replacing $\gamma_{(\eta_{k}-1)}$ with $\left( [\tau_{(\eta_{k}-1)},t_0+\tau_{\eta_{k}}]\ni s \mapsto \widetilde{\varphi}^{s-\tau_{(\eta_{k}-1)}}(\gamma_{(\eta_{k}-1)}(\tau_{(\eta_{k}-1)}))\right)$.
	We obtain a modified chain after replacing $\tau_{(\eta_k+1)}$ with $(t_0 + \tau_{\eta_k})$, deleting $\tau_{\eta_k}$ from the sequence $(\tau_j)_{j=0}^N$, deleting $\eta_k$ from the sequence $(\eta_j)_{j=0}^M$, and reindexing the sequences accordingly. 
	
	\textbf{Case (ii):} We want to replace the jump at $\tau_{\eta_k}$ with a jump at $\tau_{\eta_k} + T$.  
	Since $\dist(u, v) < \beta$, our uniform continuity considerations imply that $\dist(\widetilde{\varphi}^s(u), \widetilde{\varphi}^s(v)) < \delta$ for all $s\in [0,2T]$ such that the expression on the left is defined.
	Since $\widetilde{\varphi}^s(v)$ is defined for all $s\in [0,2T]$ and since $\mu|_U < T$ by our choice of $U$, it follows that $\widetilde{\varphi}^s(v) \in \st\setminus U$ for all $s\in [0,T]$.  
	Since  $\delta < \dist(\gd, \st \setminus U)$, it follows that $\widetilde{\varphi}^s(u)$ is also defined for all $s\in [0,T]$.
	Hence $\dist(\widetilde{\varphi}^T(u), \widetilde{\varphi}^T(v)) < \delta < \epsilon$,  so we can replace the jump at $\tau_{\eta_k}$ with a jump at $T + \tau_{\eta_k}$ from  $\widetilde{\varphi}^T(u)$ to $\gamma_{\eta_{k}}(T +\tau_{\eta_k})$.
	We obtain a modified chain after extending the domain of $\gamma_{(\eta_{k}-1)}$ to $[\tau_{(\eta_{k}-1)},T+\tau_{\eta_{k}}]$, replacing $\tau_{\eta_k}$ with $T+\tau_{\eta_k}$, and replacing $\gamma_{\eta_{k}}$ with its restriction $\gamma_{\eta_{k}}|_{[T+\tau_{\eta_k},\tau_{(\eta_k+1)}]}$.  
	
	After applying the procedure described above in Cases (i)(a-b) and (ii), we obtain an $(\epsilon, T)$-chain for which $\tau_{\eta_k} - \tau_{(\eta_k - 1)} \ge T$ for all $k$ such that $0 < \eta_k < N $.  
	The resulting chain will be nice unless $\eta_M = N$ and $\tau_{N} - \tau_{N-1} < T$;  
	if this is the case, we end up with an $(\epsilon,T)$-chain $\chi = (N, \tau, \eta, \gamma) \in \conleyet_H$ satisfying (I) $\tau_{\eta_k} - \tau_{(\eta_k - 1)} \ge T$ for all $0 < \eta_k < N $, (II) $\eta_M = N$, (III) $\tau_N - \tau_{N-1} < T$, and (IV) $\tau_{N+1} - \tau_N \le 2T$ (from the second paragraph of the proof).   We call such chains \concept{almost-nice}; note that $N\geq 2$ for an almost-nice chain.  By the above argument, for any $\epsilon, T > 0$ we can construct an $(\epsilon, T)$-chain from $x$ to $y$ which is either nice or almost-nice.

	We now claim that, from the above, it follows that a nice $(\epsilon, T)$-chain between $x$ and $y$ exists.
	Suppose (to obtain a contradiction) that this is not the case.
	Then for each $n \in \mathbb{N}$, there exists  an almost-nice $(1/n,T)$-chain $\chi^{(n)} = (N^{(n)}, \tau^{(n)}, \eta^{(n)}, \gamma^{(n)})$ from $x$ to $y$.
	For each $n$, define $t_n = \tau_{N^{(n)}}^{(n)} - \tau_{N^{(n)}-1}^{(n)} < T$ and $t'_n = \tau_{N^{(n)}+1}^{(n)} - \tau_{N^{(n)}}^{(n)} \leq 2T$.
	Similarly, let $u_n = r(\gamma_{N^{(n)}-2}^{(n)}(\tau_{N^{(n)}-1}^{(n)}))$, $v_n = \gamma_{N^{(n)}-1}^{(n)}(\tau_{N^{(n)}-1}^{(n)})$, and $w_n = \gamma_{N^{(n)}}^{(n)}(\tau_{N^{(n)}}^{(n)})$. 
	Since $\st$ is compact, after passing to a subsequence we may assume that $t_n, t'_n \to t,t' \in [0, 2T]$ and $u_n,v_n,w_n \to u,v,w \in \st$.   
	Furthermore, each $(t_n,v_n)$ and $(t'_n,w_n)$ belong to  $\dom(\widetilde{\varphi}) = \cl(\dom(\varphi))$, which is closed in $[0,\infty) \times \st$; hence $(t,v),(t',w)\in \dom(\widetilde{\varphi})$.
	Since also $\widetilde{\varphi}^{t_n}(v_n) \to w$ and $\widetilde{\varphi}^{t'_n}(w_n) \to y$, we have $\widetilde{\varphi}^t(v) = w$ and $\widetilde{\varphi}^{t'}(w) = y$  by continuity of $\widetilde{\varphi}$.
	Hence Lemma~\ref{lem:extend} implies that $(t+t',v)\in \dom(\widetilde{\varphi})$ and $\widetilde{\varphi}^{t+t'}(v) = y$.
	Pick $n$ large enough so that $1/n < \epsilon/2$ and $\dist(v_n, v) < \epsilon/2$.  
	Since $\dist(u_n, v_n) < 1/n <   \epsilon/2$, the triangle inequality implies that $\dist(u_n, v) < \epsilon$.  
	Thus we can modify the chain $\chi^{(n)}$ so that the jump occurring at $\tau_{N^{(n)}-1}^{(n)}$ is from $u_n$ to $v$, so that $\gamma_{N^{(n)}-1}^{(n)}$ is replaced with the arc $$[\tau_{N^{(n)}-1}^{(n)},t+t'+\tau_{N^{(n)}-1}^{(n)}]\ni  s\mapsto \widetilde{\varphi}^{s-\tau_{N^{(n)}-1}^{(n)}}(v)$$ terminating at $y$ and $\tau_{N^{(n)}}^{(n)}$ is replaced with $(t+t'+\tau_{N^{(n)}-1}^{(n)})$, and so that the final arc $\gamma_{N^{(n)}}^{(n)}$ and time $\tau_{N^{(n)}+1}^{(n)}$ are deleted from the sequences $\gamma^{(n)}$, $\tau^{(n)}$.
	Since $N \geq 2$ for an almost-nice chain, the resulting chain is a nice $(\epsilon, T)$-chain from $x$ to $y$ (consisting of $N-1$ arcs), which contradicts our assumption that a nice $(\epsilon, T)$-chain from $x$ to $y$ does not exist.
	This shows that $\widehat{\conley_H} \supseteq \conley_H$ and completes the proof. 
\end{proof}

\RelaxedConley*
\begin{proof}
	Throughout the proof, let $d_\st$ be the extended metric for $\st$ and $d_{\st'}$ be the extended metric for $\st'$.	
	
	Let $\epsilon, T>0$ and $(x,y)\in \conley_H$. 
	Since $\st$ is compact (Assumption~\ref{assump:compact}), $\iota\colon \st \to \st'$ is uniformly continuous, and hence there exists $\delta > 0$ such that $u,v \in \st, \, d_\st(u,v)< \delta$ implies that $d_{\st'}(\iota(u),\iota(v)) < \epsilon$.
	Let $\chi=(N,\tau,\eta,\gamma)\in {\conley^{\delta, T}_H(x,y)}$. 
	For each arc $\gamma_j$ of $\chi$, we define $\widetilde{\gamma}_j$ to be $\iota \circ \gamma_j$ if $\gamma_j$ ends in a continuous-time jump or is the final arc of $\chi$.  
	Otherwise $\gamma_j$ ends in a reset jump, in which case we define $\widetilde{\gamma}_j$ to be the concatenation of $\iota \circ \gamma_j$ with the path $[0,1] \to \st'$ defined by $t \mapsto \pi_0(t, z)$, where $z \in \gd$ is the endpoint of $\gamma_j$ and $\pi_0$ is the quotient map of Def.~\ref{def:relaxed}.
	Then the collection of arcs $\widetilde \gamma_j$ defines a chain $\chi'\in {\conley^{\epsilon,T}_{H'}(\iota(x), \iota(y))}$.
	
	For the converse, let $\epsilon, T > 0$ and $(\iota(x),\iota(y))\in \conley_{H'}$. 
	Let $\rho\colon U\subseteq \st\to \gd$ be a flow-induced retraction. 
	Shrinking $U$ if necessary, we may assume that $U$ is compact and that $\mu|_U$ is strictly bounded above by $T\in (0,\infty)$, where the maximum flow time $\mu\colon \st\to [0,+\infty]$ is defined in Equation~\eqref{eq:max-flow-time-all}.
	Define $U'\coloneqq \iota(U)\cup \pi_0(\gd\times [0,1])\subseteq \st'$, define the flow-induced retraction $\rho'\colon U'\to \gd'$ by $\rho'(\iota(x))\coloneqq \pi_0(\rho(x),1)$ and $\rho'(\pi_0(z,t)) \coloneqq \pi_0(z,1)$, and define $\mu'\colon U' \to [0,1+T)$ by $\mu'(\iota(u)) \coloneqq 1 + \mu(u)$ and $\mu'(\pi_0(z,t))\coloneqq 1-t$. 
	Let $\widetilde{\varphi}$ be the continuous extension of $\varphi'$ to the closure $\cl(\dom(\varphi'))$ of $\dom(\varphi')$ in $[0,\infty)\times \st'$ ensured by Assumptions \ref{assump:deterministic}, \ref{assump:inf-or-Zeno}, and \ref{assump:trapping-guard} and Lemma \ref{lem:extend}.
	Since $\st$, $\gd$, and $U'$ are compact and $\iota\colon \st \to \iota(\st)\subseteq \st'$ is a homeomorphism, there exists $\delta >  \zeta > 0$ such that
	\begin{equation}\label{eq:distance-Z-01-I-minus-U'}
	d_{\st'}(\pi_0(\gd\times [0,1]), \st\setminus \interior(U')) > \zeta
	\end{equation}
	and
	\begin{align}
	u,v\in \st,\, d_{\st'}(\iota(u),\iota(v))< \delta & \implies d_\st(u,v) < \epsilon\label{eq:H-H'-metric-compare}\\
	p,q \in U', \, d_{\st'}(p,q) < \zeta &\implies d_{\st'}(r'\circ \rho'(p),r'\circ \rho'(q)) < \delta.\label{eq:H'-retract-reset-continuity}
	\end{align}
	
	By assumption there exists a chain $\chi^{(0)} = (N,\tau^{(0)},\eta^{(0)},\gamma^{(0)})\in {\widehat{\conley}^{\zeta, T+1}_{H'}(\iota(x), \iota(y))}$; recall from Def.~\ref{def:eps-T-chains} that $N\geq 1$.
	We will now modify the chain $\chi^{(0)}$ inductively.
	Fix $i\in \{0,\ldots, N-1\}$ and assume that, if $i \geq 1$,\footnote{If $i = 0$ we assume nothing, so that the base case of the induction argument is included in this one.}  we have modified the first $(i+1)$ arcs ($\gamma_0^{(0)},\ldots, \gamma_i^{(0)}$) to obtain a chain $\chi^{(i)} = (N,\tau^{(i)},\eta^{(i)},\gamma^{(i)})\in {\widehat{\conley}^{\delta, T+1}_{H'}(\iota(x), \iota(y))}$ such that (a) the sub-chain obtained by throwing away the first $(i+1)$ arcs of $\chi^{(i)}$ is either a single arc (if $i = N-1$) or a $(\zeta, T+1)$-chain, and (b) for all $j\in \{1,\ldots, i\}$: $\gamma_{j-1}^{(i)}(\tau_{j}^{(i)}) \not \in \pi_0(\gd\times [0,1))$ and  $\gamma_{j}^{(i)}(\tau_{j}^{(i)}) \not \in \pi_0(\gd\times (0,1])$. 
	If \emph{both} $\gamma_{i}^{(i)}(\tau_{i+1}^{(i)}),\gamma_{i+1}^{(i)}(\tau_{i+1}^{(i)})\in U'$,\footnote{Note that, if \emph{only one} of $\gamma_{i}^{(i)}(\tau_{i+1}^{(i)}),\gamma_{i+1}^{(i)}(\tau_{i+1}^{(i)})$ is in $U'$, then \eqref{eq:distance-Z-01-I-minus-U'} implies that it cannot belong to $\pi_0(\gd\times [0,1])$.} then we replace $\gamma_{i}^{(i)}$ with the arc $$[\tau_{i}^{(i)}, \tau_{i+1}^{(i)} + \mu'(\gamma_{i}^{(i)}(\tau_{i+1}^{(i)}))]\ni t \mapsto \widetilde{\varphi}^{t - \tau_{i}^{(i)}}\left(\gamma_{i}^{(i)}(\tau_{i}^{(i)})\right)$$ and $\gamma_{i+1}^{(i)}$ with the degenerate arc $\left( \{\tau_{i+1}^{(i)} + \mu'(\gamma_{i}^{(i)}(\tau_{i+1}^{(i)}))\}\ni t \mapsto r'\circ \rho'(\gamma_{i+1}^{(i)}(\tau_{i+1}^{(i)}))\right)$.
	The upper bound $\mu'(\slot) < T$ and Equations \eqref{eq:distance-Z-01-I-minus-U'}, \eqref{eq:H'-retract-reset-continuity} can be used to show that, after redefining the sequences $\eta^{(i)}$ and $\tau^{(i)}$ accordingly, the result is a chain $\chi^{(i+1)}= (N,\tau^{(i+1)},\eta^{(i+1)},\gamma^{(i+1)})\in {\widehat{\conley}^{\delta, T+1}_{H'}(\iota(x), \iota(y))}$ such that
	(a) the sub-chain obtained by throwing away the first $(i+2)$ arcs of $\chi^{(i+1)}$ is either empty (if $i = N-1$), a single arc (if $i = N-2$), or a $(\zeta, T+1)$-chain and (b) for all $j\in \{1,\ldots, i+1\}$: $\gamma_{j-1}^{(i+1)}(\tau_{j}^{(i+1)}) \not \in \pi_0(\gd\times [0,1))$ and $\gamma_{j}^{(i+1)}(\tau_{j}^{(i+1)}) \not \in \pi_0(\gd\times (0,1])$.
	
	Hence by induction we obtain a chain $\chi \in \widehat{\conley}^{\delta, T+1}_{H'}(\iota(x), \iota(y))$ satisfying $\gamma_i(\tau_{i+1}) \not \in \pi_0(\gd \times [0,1))$ and $\gamma_i(\tau_i)\not \in \pi_0(\gd\times (0,1])$ for every arc of $\chi$. 
	This implies that, if any arc of $\chi$ meets $\pi_0(\gd\times [0,1])$, it must pass through $\pi_0(\gd\times \{0\})$ and terminate at $\pi_0(\gd\times \{1\})$. 
	Deleting the portions of the arcs that pass through $\pi_0(\gd\times [0,1])$, composing the resulting arcs with $\iota^{-1}\colon \iota(\st) \to \st$, and using \eqref{eq:H-H'-metric-compare},  we finally obtain an $(\epsilon,T)$-chain from $x$ to $y$ for $H$.
	This completes the proof. 
\end{proof}

\Colonius*
\begin{proof}
	It suffices to show that, for every $\epsilon > 0$, there are $(\epsilon, 2T)$-chains from (i) $x$ to $y$ and (ii) $y$ to $x$.
	By the compactness of $X$, the map $\Phi$ is uniformly continuous on $X \times [0, 4 T]$. 
	Hence there exists $\delta \in (0, \epsilon/2)$ such that for all $a,b \in X$ and $t \in [0,4T]$ we have $\dist(\Phi^t(a), \Phi^t(b)) < \epsilon/2$ whenever $\dist(a,b) < \delta$.
	
	Let $\chi = (N, \tau, \eta, \gamma) \in \conley^{\delta,T}(x,y)$.  
	Without loss of generality, we may assume that $\tau_{i+1} - \tau_{i} \in [T, 2T]$ for all $0\leq i \leq N$.  
	By concatenating $\chi$ with a $(\delta,T)$-chain from $y$ to $x$ followed by a $(\delta,T)$-chain from $x$ to $y$, we may assume that $N \ge 3$ (so that $\chi$ contains at least four arcs). 
	Thus, there exist integers $N' \ge 1$ and $r \in \{1, 2\}$ with $N = 2N' + r$.
	
	Let $\tau' = (\tau_i')_{i=0}^{N'+1}$ be the sequence given by $\tau_i' \coloneqq \tau_{2i}$ for $0 \leq i \leq N'$ and $\tau'_{N'+1} \coloneqq \tau_{N+1}$.  
	We (uniquely) define the corresponding arcs $\gamma_0',\ldots,\gamma'_{N'}$ by their starting points $\gamma_i'(\tau_i') \coloneqq \gamma_i(\tau_i')$ for $0\leq i \leq N'$.  
	Then by the triangle inequality the chain $(N', \tau',\eta', \gamma')$ (where $\eta' = (0,1,\ldots,N')$) defines an $(\epsilon, 2T)$-chain from $x$ to $y$.
	
	Repeating the above argument with the roles of $x$ and $y$ reversed also yields an $(\epsilon,2T)$-chain from $y$ to $x$.
	This completes the proof.
\end{proof}

\end{document}